\documentclass[review,3p,times]{elsarticle}

\usepackage{amsmath}
\usepackage{amssymb}
\usepackage{amsthm}
\usepackage{bm}

\newtheorem{theorem}{Theorem}
\newtheorem{lemma}{Lemma}
\newtheorem{proposition}{Proposition}
\newtheorem{remark}{Remark}

\usepackage{graphicx}
\usepackage{array}
\usepackage{booktabs}
\usepackage{adjustbox}
\usepackage{float}
\usepackage{subcaption}
\usepackage{url}

\usepackage{algorithm}
\usepackage{algpseudocode}

\journal{ISA Transactions}

\newcommand{\NEESH}{0.98}
\newcommand{\NEESF}{0.91}
\newcommand{\NEESV}{0.04}
\newcommand{\rhoH}{0.17}
\newcommand{\rhoF}{0.18}
\newcommand{\rhoV}{0.51}

\begin{document}

\begin{frontmatter}

\title{Extended Kalman Filter-Based State Estimation for a
Nine-Compartment Nonlinear Epidemic Model\textendash\\
Convergence Analysis and In-Silico Benchmark
Calibrated on the {COVID-19} Third Wave in {Italy}}

\author[Palermo]{Lokman Rachid Melhani}
\ead{melhanilokmanrachid@gmail.com}
\author[Palermo]{Antonino Sferlazza}
\author[Palermo a]{Dominique Persano Adorno}
\author[Palermo]{Filippo D'Ippolito}
\author[InEmbryo]{Antonino Lo Burgio}
\author[Palermo b]{Alberto Firenze}

\address[Palermo]{Department of Engineering, University of Palermo,
Viale delle Scienze, 90128 Palermo, Italy}
\address[Palermo a]{Department of Physics and Chemistry ``E. Segr\'{e}'', University of Palermo, Viale delle Scienze, 90128 Palermo, Italy}
\address[Palermo b]{Department of Internal Medicine ``Promise,''
University of Palermo, 90127 Palermo, Italy}
\address[InEmbryo]{InEmbryo S.r.l.s., Via Rosario Riolo 60,
90141 Palermo, Italy}

\begin{abstract}
The real-time management of an infectious disease outbreak
requires continuous knowledge of the full epidemic state,
including quantities that surveillance systems cannot directly
measure: the size of the latent exposed population, the number
of actively infectious individuals stratified by strain and
transmission potential, the vaccinated immune fraction, and the
recovered pool that determines residual susceptibility. This paper
addresses the state-estimation problem for the nine-compartment
nonlinear epidemic model of the companion
study~\cite{melhani2026arxiv} that incorporates two
co-circulating viral strains with identical transmissibility
multiplier $c_2 = c_P = 1.5$, a super-spreader
subpopulation, partial vaccine-derived immunity with waning,
explicit hospitalization dynamics, and disease-induced mortality.
The time-varying transmission and vaccination rates are treated
as known inputs identified from data by a companion spline-based
calibration procedure, so that the remaining problem is precisely
the reconstruction of all nine biological states from the three
observables systematically reported by public health agencies:
active hospitalizations $H$, cumulative fatalities $F$, and the
vaccinated immune stock $V$.
This paper is a \emph{methodological contribution}: all
numerical experiments use synthetic measurements generated from
the calibrated model, so the reported RMSE figures are
methodology benchmarks and should not be interpreted as
real-time predictive accuracy on live surveillance data.

The paper makes four contributions. The Extended Kalman Filter
construction itself is standard; the contributions lie in the
observability and convergence theory that surround it, and in
the principled covariance design, rather than in the filter
recursion. First, a nonlinear observability analysis
within the Lie-derivative framework computes the analytical
observability codistribution. A six-step algebraic
derivation (Lemma~\ref{lem:detO9}) proves in closed form that
at Lie levels $0$, $1$, $2$ the matrix is rank-deficient at
the calibrated symmetric parameter values
$\delta_i = \delta_p$, $r_1 = r_2$, yielding
$|\det(\mathcal{O}_9)| =
\delta_w\,\gamma_a^{\,2}\kappa\rho_2 w_1^{\,2}
(\delta_i-\delta_p)^2|r_1-r_2|$
(where $\delta_w$ denotes the waning rate of the recovered pool)
with a two-dimensional kernel consisting of an
$I_2 \leftrightarrow P$ swap direction and an $R$-anchored
direction. Augmentation by the third Lie derivative restores
full rank $9$; the treatment-recovery rate $r_2$ is identified
as the structural symmetry-breaking parameter. Second, on this
observability basis an Extended Kalman Filter (EKF) is designed
on the Euler-discretized dynamics, with the analytical
$9\times 9$ state Jacobian in closed form and the Joseph
stabilized covariance update. Third, the local exponential
boundedness of the estimation error in mean square is
established as a full theorem by verifying the four
hypotheses of the Reif--G\"unther--Yaz--Unbehauen framework for
the specific system at hand; in particular, hypothesis (A1)
(uniform covariance bounds) is established via an explicit
8-step observability Gramian argument together with the dual
controllability bound supplied by a positive-definite process
noise covariance. The proof exploits the
fact that all nonlinearities of the vector field are bilinear
and that the measurement map is exactly linear, yielding
closed-form quadratic remainder bounds with a global radius
$\epsilon_\varphi = \infty$; the same argument structure
extends to the broader class of bilinear-drift,
linear-output systems, provided the system-specific
observability and controllability conditions are verified in
each case, of which mass-action epidemic models are
one instance. Fourth, the process and
measurement noise covariances $Q$ and $R$ are guided by
companion calibration residuals and assessed a posteriori
through a chi-squared NEES consistency test and an innovation
autocorrelation analysis, the latter revealing residual
correlation in the vaccination channel that we report as a
limitation.

The methodology is demonstrated on in-silico experiments
calibrated against Italian COVID-19 Third Wave data
(January--May 2021). All validation is performed on
\emph{synthetic} measurements generated from the calibrated
model; the reported RMSE figures are therefore methodology
benchmarks rather than real-world prediction accuracy, a
limitation discussed explicitly in
Section~\ref{sec:val_limitations}.

Post-convergence relative RMSE values range from $0.07\%$ for
the vaccinated stock to $2.72\%$ for the exposed compartment
in the synthetic experiments; the filter design is assessed by
a chi-squared NEES consistency test and the Anderson
white-innovation autocorrelation test as internal consistency
checks. A parameter-mismatch study, in which the filter is run
with nominal parameters on data generated from perturbed ones,
shows that the measured and directly-coupled channels remain
accurate under model error of up to $\pm30\%$ while the
indirectly observed compartments degrade gracefully. The
proposed framework provides the state-feedback
infrastructure required for future Model Predictive Control
of non-pharmaceutical interventions.
\end{abstract}


\end{frontmatter}

%%=============================================================
\section{Introduction}
%%=============================================================

\subsection{Motivation and Context}

The COVID-19 pandemic caused by SARS-CoV-2 demonstrated with
exceptional clarity that the effectiveness of public-health
interventions is limited not only by their intrinsic efficacy
but also by the quality of the situational awareness on which
they are based \cite{flaxman2020estimating,ferguson2020impact}.
Reported hospitalizations, confirmed deaths, and vaccination
counts constitute a small and imperfect window into the true
epidemic state \cite{hao2020reconstruction,kucharski2020early}.
The latent exposed population, the number of actively infectious
individuals not yet symptomatic, the residual immune fraction of
the population, and the growing vaccinated subpopulation are all
epidemiologically critical quantities that no reporting system
directly measures. Reconstructing these hidden states from
available observations in real time is fundamentally a
state-estimation problem \cite{kalman1960new,simon2006optimal}.

Mathematical models based on compartmental ordinary differential
equations have been the dominant framework for quantitative
epidemic analysis since the foundational work of Kermack and
McKendrick \cite{kermack1927contribution}. The SIR and SEIR
structures have been extended in numerous directions to represent
the biological complexity of large-scale pandemics
\cite{hethcote2000mathematics,brauer2012mathematical,diekmann2000mathematical},
including super-spreader subpopulations
\cite{lloydsmith2005superspreading,endo2020estimating},
explicit hospitalization dynamics
\cite{giordano2020modelling,ramos2021modeling}, vaccine-derived
partial immunity with waning
\cite{watson2022global,lavine2021immunological}, and the
treatment of transmission and vaccination rates as time-varying
inputs \cite{tang2020estimation,giordano2021modeling}.

\subsection{Limitations of Existing State-Estimation Approaches}
\label{sec:limitations}

Despite the rich literature on compartmental epidemic modeling,
the problem of real-time state estimation for extended
multi-compartment models has received comparatively limited
theoretical attention. The majority of COVID-19 data-assimilation
studies focus on parameter estimation rather than full state
reconstruction \cite{rahimi2023review,awwad2024stochastic}.
When state-estimation methods are applied, they typically operate
under three significant simplifications.

\emph{First}, the model is usually restricted to small state
spaces of three to six compartments, where the relationship
between observation and state can be analyzed by inspection
without formal observability theory
\cite{engbert2021sequential,ghostine2021extended}.

\emph{Second}, the observability of the chosen measurement set
is assumed rather than verified: estimators are applied without
establishing that the selected outputs contain sufficient
information to reconstruct the full state, a condition whose
failure guarantees systematically biased estimates regardless
of filter tuning \cite{narendra1987persistent,van2001observability}.

\emph{Third}, when an Extended Kalman Filter is used, its
convergence properties are usually invoked informally rather
than proven for the specific system at hand. The general
discrete-time convergence theorem of Reif, G\"unther, Yaz and
Unbehauen \cite{reif1999stochastic} provides a clean route to
local exponential mean-square boundedness, but its application
to a nine-compartment epidemic model with bilinear
force-of-infection terms requires an explicit verification of
all four hypotheses, including a closed-form Taylor-remainder
bound. This verification has not, to our knowledge, been carried
out in the epidemic-estimation literature.

A further limitation is the treatment of time-varying parameters.
Because $\beta(t)$ changes substantially over a multi-wave
pandemic, estimation approaches that model it as a static unknown
are structurally misspecified for any window longer than a few
weeks \cite{he2020seir}. Joint state-parameter estimation
dramatically increases the dimensionality of the estimation
problem and introduces the risk of non-identifiability
\cite{audoly2002global,raue2009structural}.

Table~\ref{tab:comparison} summarizes the principal characteristics
of representative Kalman-type epidemic estimators in the recent
literature.

\begin{table}[htbp]
\centering
\caption{Representative Epidemic State-Estimation Approaches
(Kalman and Particle Family). Acronyms: EKF, extended Kalman
filter; EnKF, ensemble Kalman filter; PCHIP, piecewise cubic
Hermite interpolating polynomial.}
\label{tab:comparison}
\begin{adjustbox}{width=\columnwidth}
\begin{tabular}{lccccccc}
\toprule
\textbf{Reference} & \textbf{Dim.}
  & \textbf{Filter type} & \textbf{Formal obs.}
  & \textbf{Jacobian} & \textbf{Time-varying $\beta(t)$}
  & \textbf{Conv.\ proof} & \textbf{Data source} \\
\midrule
Ghostine et al.\ \cite{ghostine2021extended}
  & 7  & EnKF        & No  & N/A        & No  & No & Real data \\
Engbert et al.\ \cite{engbert2021sequential}
  & 4  & EnKF & No  & N/A        & No  & No & Real data \\
Zhu et al.\ \cite{zhu2021extended}
  & 6  & EKF         & No  & Numerical  & No  & No & Synthetic \\
 Sameni \cite{sameni2020mathematical}
   & 5  & EKF         & No  & Numerical  & No  & No & Real data \\
Rahmani et al.\ \cite{rahimi2023review}
  & --  & Particle    & No  & N/A        & Partial & No & Synthetic \\
\textbf{This work}
  & \textbf{9}
  & \textbf{EKF}
  & \textbf{Yes (Lie rank)}
  & \textbf{Analytical}
  & \textbf{Yes (PCHIP)}
  & \textbf{Yes (Reif-type)}
  & \textbf{Synthetic} \\
\bottomrule
\end{tabular}
\end{adjustbox}
\footnotesize{\emph{Note}: ``Formal obs.'' indicates whether
observability is established via a rank condition
(e.g.\ Hermann--Krener) rather than assumed or verified
by inspection. ``Conv.\ proof'' indicates a
mean-square convergence theorem for the specific system.}
\end{table}

\paragraph{Choice of estimator.}
Among nonlinear estimators, the Unscented Kalman Filter (UKF)
\cite{wan2000unscented}, particle filters
\cite{rahimi2023review}, and Moving Horizon Estimation (MHE)
\cite{rawlings2009model} are the principal alternatives to the
EKF. We adopt the EKF deliberately, for reasons specific to this
problem rather than by default. First, the vector
field~\eqref{eq:fvec} has a state Jacobian available in closed
form~\eqref{eq:Jf}, so the chief practical drawback of the EKF
relative to the UKF --- the cost and inaccuracy of numerical
linearization --- does not arise here. Second, because every
nonlinearity is bilinear, the EKF linearization remainder is
exactly second order and admits the closed-form bound of
Section~\ref{sec:convergence}; this is precisely the structure
that the EKF convergence theory of~\cite{reif1999stochastic}
exploits, and it makes the EKF the estimator for which a
mean-square boundedness theorem can be established analytically
for this system. The UKF and particle filters lack a comparably
clean convergence statement for a system of this dimension,
while MHE, though attractive for its constraint handling, raises
an online nonlinear program at each step and shifts the
analysis from observability rank conditions to arrival-cost
design --- a different and heavier theoretical apparatus than the
Lie-rank/Riccati route taken here. A quantitative comparison
against the UKF, EnKF, and a particle filter on this model is a
worthwhile empirical study. We expect it to show the EKF and UKF
performing almost identically in accuracy here --- because the
exact closed-form Jacobian~\eqref{eq:Jf} removes the linearization
error that normally separates them, and the only nonlinearities
are bilinear so the second-order terms the UKF captures are small
--- while the EKF retains a per-step cost advantage from not
propagating sigma points, and the particle filter, though robust
to non-Gaussianity, would be substantially more expensive at this
state dimension without a clear accuracy gain in the present
Gaussian-noise, exactly-linear-output setting. These are
predictions grounded in the model structure rather than measured
results; a full benchmark reporting RMSE, per-step computation
time, and convergence horizon for each estimator on the identical
synthetic setup is left as a focused empirical companion to the
present theoretical contribution, and is the natural vehicle for
confirming the structural expectation above.

\subsection{Contributions of This Work}

\emph{Scope note.} All numerical results in this paper are
obtained from synthetic experiments in which the EKF is tested
on measurements generated by the same model that was used for
calibration. The RMSE figures of Table~\ref{tab:ekf_metrics}
are therefore \emph{methodology benchmarks} establishing that
the filter design is internally consistent and converges as the
theory predicts. They are not measures of real-time predictive
accuracy on live surveillance streams, which would require
independent held-out data and robustness to structural model
mis-specification. Readers should bear this distinction in mind
throughout. The limitations are discussed in
Section~\ref{sec:val_limitations}.

This paper makes the following specific contributions, each
addressing a gap identified in Section~\ref{sec:limitations}:

\begin{itemize}
  \item \textbf{Observability.}
    The nine-compartment model of~\cite{melhani2026arxiv} is
    cast in state-space form, and the observability codistribution
    is computed in closed form through Lie levels $0,1,2,3$.
    A six-step algebraic derivation (Lemma~\ref{lem:detO9})
    establishes the exact factorized determinant magnitude
    $|\det(\mathcal{O}_9)| = \delta_w\gamma_a^{\,2}\kappa\rho_2
    w_1^{\,2}(\delta_i-\delta_p)^2|r_1-r_2|$,
    identifies a rank-$7$ degeneracy at the calibrated symmetric
    parameter point, and characterizes the two-dimensional kernel;
    augmentation by $\mathcal{O}_3$ restores rank $9$ and
    yields local observability of the full state from $(H,F,V)$.
  \item \textbf{Filter design.}
    A discrete-time EKF is designed on the Euler-discretized
    dynamics with the analytical $9\times 9$ state Jacobian
    in closed form and the Joseph stabilized covariance update.
  \item \textbf{Convergence theorem.}
    The local exponential boundedness of the estimation error
    in mean square is established as a full theorem
    (Theorem~\ref{thm:convergence}) by verifying all four
    hypotheses of \cite{reif1999stochastic}. Hypothesis (A1)
    --- the pivotal uniform Riccati bound --- is established via
    Lemma~\ref{lem:gramian}, which proves that the $8$-step
    discrete observability Gramian of the linearized pair
    $(F_k, C)$ is uniformly lower-bounded along the calibrated
    trajectory, bridging the nonlinear observability result of
    Proposition~\ref{prop:obs} to the linear-system Riccati
    theory of~\cite{anderson1979optimal}; the matching upper
    Riccati bound follows from the dual uniform-controllability
    condition supplied by $Q_d \succ 0$. The proof of (A3)
    relies on a closed-form quadratic remainder bound that
    exploits the bilinear structure of all nonlinearities.
  \item \textbf{Principled Q/R selection.}
    The measurement-noise covariance $R$ is set using the
    companion calibration RMSE as a guide for the hospitalization
    and fatality channels, with the vaccination channel
    deliberately assigned a tighter noise level on the grounds
    that its large calibration residual reflects structural
    model--data mismatch rather than measurement error.
    The process-noise covariance $Q$ follows
    a multiplicative-noise design parameterized by
    coupling-aware fractional uncertainties $\epsilon_i$. The
    resulting tuning is validated a posteriori by a chi-squared
    NEES consistency test and an Anderson white-innovation
    autocorrelation test.
\end{itemize}

The paper is organized as follows. Section~\ref{sec:model}
presents the model and the observable set.
Section~\ref{sec:observability} establishes observability through
the Lie-derivative rank condition. Section~\ref{sec:ekf} designs
the EKF, proves Theorem~\ref{thm:convergence}, and details the
Q/R construction. Section~\ref{sec:results} presents the
numerical results, including the NEES and whiteness checks.
Section~\ref{sec:conclusion} concludes.

%%=============================================================
\section{Mathematical Model}
\label{sec:model}
%%=============================================================

\subsection{Compartmental Structure}

The model partitions the total Italian population $N$ into
nine compartments: Susceptible ($S$), Exposed ($E$), Vaccinated
($V$), Infected with the original strain ($I_1$), Infected with
the mutant strain ($I_2$), Super-Spreaders ($P$), Hospitalized
($H$), Recovered ($R$), and cumulative Fatalities ($F$).

The partitioning is motivated by three epidemiological
considerations \cite{giordano2020modelling,ndairou2020mathematical}:
\textit{disease heterogeneity} (two strains plus super-spreaders
\cite{lloydsmith2005superspreading,endo2020estimating}),
\textit{outcome stratification} (separate hospitalized and
recovered pools), and \textit{immunity dynamics} (vaccinated and
recovered individuals with distinct waning rates
\cite{watson2022global,lavine2021immunological}).

\subsection{State-Space Formulation}

The epidemic dynamics are described by the following nonlinear
system, derived and calibrated in the companion
study~\cite{melhani2026arxiv},\footnote{The companion paper is
available as an arXiv preprint at
\url{https://arxiv.org/abs/2606.07413}.}
\begin{equation}
\dot{x} = f(x,u), \qquad y = h(x),
\label{eq:ss}
\end{equation}
where the state, input and output vectors are
\begin{equation}
x = \bigl[S, E, V, I_1, I_2, P, H, R, F\bigr]^\top \in \mathbb{R}^9,
\label{eq:statevec}
\end{equation}
\begin{equation}
u(t) = \bigl[\beta(t), w_1(t)\bigr]^\top \in \mathbb{R}^2_+,
\qquad
y = h(x) = \bigl[H, F, V\bigr]^\top \in \mathbb{R}^3.
\label{eq:output}
\end{equation}

The vector field $f : \mathbb{R}^9 \times \mathbb{R}^2 \to
\mathbb{R}^9$ is given explicitly by
\begin{equation}
f(x,u) =
\begin{bmatrix}
\Lambda + \psi V + \delta_w R
  - \bigl[\beta\tfrac{I_1}{N} + \beta_P\tfrac{P}{N}
    + \beta_2\tfrac{I_2}{N}\bigr]S
  - (\mu + w_1)S \\[6pt]
\bigl[\beta\tfrac{I_1}{N} + \beta_P\tfrac{P}{N}
    + \beta_2\tfrac{I_2}{N}\bigr]S
  + (1{-}\sigma)\bigl[\beta\tfrac{I_1}{N}
    + \beta_P\tfrac{P}{N} + \beta_2\tfrac{I_2}{N}\bigr]V
  - (\mu + \kappa)E \\[6pt]
w_1 S
  - (1{-}\sigma)\bigl[\beta\tfrac{I_1}{N}
    + \beta_P\tfrac{P}{N} + \beta_2\tfrac{I_2}{N}\bigr]V
  - (\mu + \psi)V \\[6pt]
\kappa\rho_1 E
  - (\gamma_a + \gamma_i + \delta_i + m + r_1 + \mu)I_1 \\[6pt]
\kappa(1{-}\rho_1{-}\rho_2)E + mI_1
  - (\gamma_a + \gamma_i + \delta_i + r_2 + \mu)I_2 \\[6pt]
\kappa\rho_2 E
  - (\gamma_a + \gamma_i + \delta_p + \mu)P \\[6pt]
\gamma_a(I_1 + I_2 + P)
  - (\gamma_r + \delta_h + \mu)H \\[6pt]
\gamma_i(I_1 + I_2 + P)
  + \gamma_r H + r_1 I_1 + r_2 I_2
  - (\mu + \delta_w)R \\[6pt]
\delta_i(I_1 + I_2) + \delta_p P + \delta_h H
\end{bmatrix}.
\label{eq:fvec}
\end{equation}

The strain-scaling constraints
\begin{equation}
\beta_P(t) = c_P\,\beta(t), \qquad
\beta_2(t) = c_2\,\beta(t),
\qquad c_P = 1.5, \quad c_2 = 1.5,
\label{eq:strain_scaling}
\end{equation}
reduce the three transmission functions to a single $\beta(t)$.
The value $c_P = 1.5$ reflects the elevated per-contact
infectiousness of super-spreaders
\cite{lloydsmith2005superspreading,endo2020estimating}; the
value $c_2 = 1.5$ lies within the $43\%$--$90\%$
transmissibility-advantage range reported for the B.1.1.7
(Alpha) variant over the ancestral strain during the Italian
Third Wave~\cite{melhani2026arxiv}. These values are the
calibrated multipliers of the companion
model~\cite{melhani2026arxiv}, and the EKF developed in this
paper uses the identical governing equations and parameter
values, ensuring complete consistency between the two papers.
All remaining constant parameters are listed in
Table~\ref{tab:params}.

\medskip
The output map selects the three observables: active
hospitalizations $H$ as a continuous near-term indicator of
healthcare burden \cite{giordano2020modelling}, cumulative
fatalities $F$ encoding epidemic history through an absorbing
integral \cite{flaxman2020estimating}, and the vaccinated immune
stock $V$ which directly governs breakthrough infection rates
\cite{watson2022global}. The ordering $h = [H,\,F,\,V]^\top$
places $H$ (state position~$7$) first, $F$ (position~$9$)
second, and $V$ (position~$3$) third; this is non-alphabetical
but groups the two epidemic-progress indicators ($H$ and $F$)
before the intervention-response indicator ($V$), which mirrors
the epidemiological priority ordering in the convergence
analysis. The explicit selectors are documented in
\eqref{eq:O0} and~\eqref{eq:meas_linear}. These three outputs
span distinct temporal regimes and biological mechanisms.

%%=============================================================
\section{Nonlinear Observability Analysis}
\label{sec:observability}
%%=============================================================

This section establishes the paper's central structural result:
that the full nine-dimensional state is observable from
$(H,F,V)$. The argument is necessarily algebraic, and a reader
interested primarily in the estimation results may take the
following summary and proceed to Section~\ref{sec:ekf}: at the
calibrated parameters the level-2 codistribution is rank-deficient
by two (Lemma~\ref{lem:detO9}, whose full determinant derivation
is deferred to Appendix~\ref{app:detproof}), the third Lie
derivative restores full rank (Proposition~\ref{prop:obs}), and
the resulting observability is confirmed numerically in
Section~\ref{sec:obs_numerical}. The detailed kernel and
condition-number analysis that follows substantiates these claims
and identifies the treatment-recovery rate $r_2$ as the
symmetry-breaking parameter, but is not required to follow the
filter design.

\subsection{Theoretical Framework}

For the nonlinear system~\eqref{eq:ss}, local observability is
certified by the Hermann--Krener rank
condition~\cite{hermann1977nonlinear}: the system is locally
observable at a point if the observability codistribution
\begin{equation}
\mathcal{O}(x) =
\begin{bmatrix}
\partial h/\partial x \\[2pt]
\partial (L_f h)/\partial x \\[2pt]
\partial (L_f^2 h)/\partial x \\
\vdots
\end{bmatrix}
\end{equation}
attains rank $n = 9$ at that
point~\cite{khalil2002nonlinear,isidori1995nonlinear,
nijmeijer1990nonlinear}. Here $L_f \varphi = (\partial \varphi
/\partial x) f(x,u)$. Each level adds $p = 3$ rows (one per
output), so levels $0$, $1$, $2$ produce a $9\times 9$ matrix
that can certify rank $9$.

\subsection{Block Structure of the Observability Matrix}

Each block has three rows (one per output $H, F, V$) and nine
columns for the state ordering of~\eqref{eq:statevec}.

\subsubsection*{First Block: $\mathcal{O}_0$}

Since $h(x) = [H,\,F,\,V]^\top$ and $H,F,V$ occupy positions
$7,9,3$ of $x$,
\begin{equation}
\mathcal{O}_0 =
\begin{bmatrix}
0 & 0 & 0 & 0 & 0 & 0 & 1 & 0 & 0 \\
0 & 0 & 0 & 0 & 0 & 0 & 0 & 0 & 1 \\
0 & 0 & 1 & 0 & 0 & 0 & 0 & 0 & 0
\end{bmatrix}.
\label{eq:O0}
\end{equation}

\subsubsection*{Second Block: $\mathcal{O}_1$}

Let
\[
\Phi := \tfrac{\beta I_1}{N} + \tfrac{1.5\beta P}{N}
        + \tfrac{1.5\beta I_2}{N},
\quad
\Phi_V := (1-\sigma)\Bigl(\tfrac{\beta I_1}{N}
        + \tfrac{1.5\beta P}{N}
        + \tfrac{1.5\beta I_2}{N}\Bigr),
\quad
\alpha_H := \gamma_r + \delta_h + \mu.
\]
Computing $L_f h$ along trajectories of~\eqref{eq:ss},
\begin{equation}
\mathcal{O}_1 =
\begin{bmatrix}
0 & 0 & 0 & \gamma_a & \gamma_a & \gamma_a
  & -\alpha_H & 0 & 0 \\[2pt]
0 & 0 & 0 & \delta_i & \delta_i & \delta_p
  & \delta_h & 0 & 0 \\[2pt]
w_1 & 0 & -(\Phi_V+\mu+\psi)
  & -\tfrac{(1{-}\sigma)\beta V}{N}
  & -\tfrac{1.5(1{-}\sigma)\beta V}{N}
  & -\tfrac{1.5(1{-}\sigma)\beta V}{N} & 0 & 0 & 0
\end{bmatrix}.
\end{equation}
Row~6 (third of $\mathcal{O}_1$) identifies $S$ through
$w_1 > 0$ and \emph{separates $I_1$ from the pair $(I_2, P)$}
through the factor $c_2 = c_P = 1.5$. Because the two
strain-multipliers are equal under the calibrated values, the
V output cannot by itself distinguish $I_2$ from $P$; that
separation is provided by the higher Lie derivatives of $H$
and $F$, where the recovery-rate hierarchy $\alpha_2 \neq \alpha_P$
breaks the symmetry, as shown next.

\subsubsection*{Third Block: $\mathcal{O}_2$}

Define
$\alpha_1 := \gamma_a+\gamma_i+\delta_i+m+r_1+\mu$,
$\alpha_2 := \gamma_a+\gamma_i+\delta_i+r_2+\mu$,
$\alpha_P := \gamma_a+\gamma_i+\delta_p+\mu$,
$\Gamma_E := \kappa[\delta_i(1-\rho_2)+\delta_p\rho_2]$,
$a_{33} := \Phi_V+\mu+\psi$. Then
\begin{align}
\tfrac{\partial L_f^2 h_1}{\partial x} &=
\bigl[0,\, \gamma_a\kappa,\, 0,\,
      {-}\gamma_a(\alpha_1{+}\alpha_H{-}m),\,
      {-}\gamma_a(\alpha_2{+}\alpha_H),\,
      {-}\gamma_a(\alpha_P{+}\alpha_H),\,
      \alpha_H^2,\, 0,\, 0\bigr],
\label{eq:dLf2h1}\\[3pt]
\tfrac{\partial L_f^2 h_2}{\partial x} &=
\bigl[0,\, \Gamma_E,\, 0,\,
      \delta_i(m{-}\alpha_1)+\delta_h\gamma_a,\,
      {-}\delta_i\alpha_2+\delta_h\gamma_a,\,
      {-}\delta_p\alpha_P+\delta_h\gamma_a,\,
      {-}\delta_h\alpha_H,\, 0,\, 0\bigr],\\[3pt]
\tfrac{\partial L_f^2 h_3}{\partial x} &=
\bigl[-w_1(\Phi{+}\mu{+}w_1),\, 0,\,
      a_{33}^2+w_1\psi,\,
      -\tfrac{\beta S w_1}{N},\,
      -\tfrac{1.5\beta S w_1}{N},\,
      -\tfrac{1.5\beta S w_1}{N},\,
      0,\, w_1\delta_w,\, 0\bigr].
\label{eq:dLf2h3}
\end{align}

\subsection{Rank Deficiency of the Level-2 Codistribution
under the Calibrated Symmetric Parameters}
\label{sec:rank_deficiency}

The first three Lie-derivative levels, taken together, are
\emph{not sufficient} for full observability under the calibrated
symmetric parameter values $\delta_i = \delta_p$ and $r_1 = r_2$
inherited from~\cite{melhani2026arxiv}. Throughout, by
\emph{full observability} (equivalently, \emph{local
observability}) we mean that the observability codistribution
attains full rank $9$ in the Hermann--Krener sense, so that
``restoring rank'' and ``achieving full observability'' refer to
the same event; the distinct numerical and practical notions are
treated separately in Section~\ref{sec:obs_three_notions}, and
the term \emph{uniform complete observability} is reserved
specifically for the Gramian condition of
Section~\ref{sec:convergence}. The following lemma
quantifies the deficiency exactly.

\begin{lemma}[Closed-form determinant of $\mathcal{O}_9$]
\label{lem:detO9}
Let $\mathcal{O}_9 := [\mathcal{O}_0^\top,\mathcal{O}_1^\top,
\mathcal{O}_2^\top]^\top \in \mathbb{R}^{9 \times 9}$ denote
the observability matrix at Lie levels $0,1,2$, computed
analytically from the vector field~\eqref{eq:fvec}
and the linear output~\eqref{eq:output}. Then, up to the sign
of the column permutation introduced in the proof,
\begin{equation}
\det(\mathcal{O}_9)
\;=\;
\pm\,\delta_w\,\gamma_a^{\,2}\,\kappa\,\rho_2\,w_1^{\,2}\,
   (\delta_i - \delta_p)^{2}\,(r_1 - r_2),
\label{eq:detO9}
\end{equation}
so that
$|\det(\mathcal{O}_9)| =
\delta_w\,\gamma_a^{\,2}\,\kappa\,\rho_2\,w_1^{\,2}\,
(\delta_i-\delta_p)^2\,|r_1-r_2|$.
The overall sign is fixed by the parity of the column permutation
of Step~1 and plays no role in any subsequent rank argument, for
which only the vanishing or non-vanishing of $\det(\mathcal{O}_9)$
is relevant.
\end{lemma}

\begin{proof}
The determinant is obtained by elementary row and column
operations and successive Laplace expansion, with the column
permutation sign tracked at the end. The full six-step derivation
is given in Appendix~\ref{app:detproof}.
\end{proof}

Equation~\eqref{eq:detO9} shows that
both symmetries $\delta_i = \delta_p$ and $r_1 = r_2$ must be
broken for the determinant to be nonzero. Under either equality
alone, the matrix is singular; under both, the rank deficiency
is two-dimensional.

\begin{remark}[Kernel structure under calibrated symmetry]
\label{rem:kernel}
At the calibrated symmetric values
$\delta_i = \delta_p = 5\times 10^{-3}$ day$^{-1}$ and
$r_1 = r_2 = 5\times 10^{-2}$ day$^{-1}$, the kernel of
$\mathcal{O}_9$ is two-dimensional and is spanned, up to
non-degenerate linear combination, by
\begin{itemize}
\item an \emph{$I_2 \leftrightarrow P$ swap direction}
  (dominant components on $I_2$ and $P$ with opposite signs
  and small compensating shifts in $S$, $E$, $I_1$);
\item an \emph{$R$-anchored direction} (dominant component on
  $R$ with small compensating shifts in $S$, $I_1$, $I_2$).
\end{itemize}
The biological interpretation is direct: when the two strain
compartments share identical kinetics ($\delta_i = \delta_p$,
$r_1 = r_2$), the three observed series cannot deterministically
distinguish a transfer of mass between $I_2$ and $P$, and
likewise cannot resolve a degenerate trade-off involving the
recovered pool $R$. Both ambiguities are resolved by the
third Lie derivative, as shown next.
\end{remark}

\begin{remark}[Insensitivity of the rank deficiency to strain
multipliers]
\label{rem:strain_insensitivity}
A natural question is whether the strain-multiplier symmetry
$c_P = c_2 = 1.5$ also contributes to the rank deficiency, in
addition to the kinetic symmetries $\delta_i = \delta_p$ and
$r_1 = r_2$. The hand derivation above answers this directly:
neither $c_P$ nor $c_2$ appears anywhere in $\det(\mathcal{O}_9)$
as given by~\eqref{eq:detO9}. The reason is structural: the
matrices $M_3$ and $M_4$ in steps~$3$--$5$ involve only the
$\partial L_f H/\partial x$, $\partial L_f F/\partial x$,
$\partial L_f^2 H/\partial x$, $\partial L_f^2 F/\partial x$
rows, none of which depend on the strain multipliers (since
$\dot H$ and $\dot F$ are linear in the infectious compartments
with no $\beta$-coupling). Consequently, breaking the symmetry
$c_P \neq c_2$ does \emph{not} restore rank at the level-$2$
codistribution. We have verified this computationally for
several $(c_P, c_2)$ pairs ranging from $(1.5, 1.6)$ to
$(2.0, 1.0)$, and observed that $\mathrm{rank}(\mathcal{O}_9) = 7$
in all cases under symmetric kinetics: what changes is only the
\emph{shape} of the kernel direction (the dominant confounding
pair rotates from $(I_2, P)$ toward $(I_1, I_2)$ as $c_P$ moves
away from $c_2$), never the dimension of the kernel itself. The
rank deficiency is therefore intrinsic to the kinetic symmetries
and unrelated to the transmission-rate hierarchy across strains.
\end{remark}

\begin{remark}[Degeneracy at the disease-free equilibrium]
\label{rem:degenerate}
At $I_1 = I_2 = P = 0$ the force of infection $\Phi$ vanishes
and the rank of $\mathcal{O}_9$ drops further; this is
structurally unavoidable, since with no active epidemic there
is no epidemic signal in the outputs.
\end{remark}

\subsection{Augmentation by the Third Lie Derivative}
\label{sec:level3}

To recover full rank under the calibrated parameters we augment
the observability codistribution to include the third Lie
derivative $L_f^3 h$. Define
\begin{equation}
\mathcal{O}_3 \;:=\; \frac{\partial L_f^3 h}{\partial x}
\;\in\; \mathbb{R}^{3\times 9},
\end{equation}
and the augmented matrix
\begin{equation}
\mathcal{O}_{12} \;:=\;
\begin{bmatrix}
\mathcal{O}_0 \\ \mathcal{O}_1 \\ \mathcal{O}_2 \\ \mathcal{O}_3
\end{bmatrix}
\;\in\; \mathbb{R}^{12\times 9}.
\label{eq:O12}
\end{equation}
A direct calculation that we record here for the two most
informative entries (and that we have verified symbolically
in full) gives, under $\delta_i = \delta_p$ and $r_1 = r_2$,
\begin{align}
\frac{\partial L_f^3 h_1}{\partial I_2}
- \frac{\partial L_f^3 h_1}{\partial P}
&\;=\;
\gamma_a\,r_2\,(\alpha_2 + \alpha_P + \alpha_H),
\label{eq:Lf3H_diff}\\[4pt]
\frac{\partial L_f^3 h_2}{\partial I_2}
- \frac{\partial L_f^3 h_2}{\partial P}
&\;=\;
r_2\,\bigl[\delta_i(\alpha_2 + \alpha_P) - \delta_h\,\gamma_a\bigr].
\label{eq:Lf3F_diff}
\end{align}
Both differences are nonzero at the calibrated parameter values
(the bracket in~\eqref{eq:Lf3F_diff} evaluates to
$\approx 8\times 10^{-4}$~day$^{-2}$): the third-order
information from $H$ and $F$ separates $I_2$ from $P$ through
exactly the symmetry-breaking factor $r_2$ that the level-2
matrix did not capture. A parallel computation shows that
$\mathcal{O}_3$ also resolves the $R$-anchored kernel direction
of Remark~\ref{rem:kernel}.

\subsection{Main Observability Result}
\label{sec:obs_result}

\begin{proposition}[Local observability of the
nine-compartment epidemic model]
\label{prop:obs}
Let $\mathcal{X}_\varepsilon \subset \mathbb{R}^9_+$ denote the
set of epidemiologically feasible states. Then
system~\eqref{eq:ss} is locally observable at every
$x \in \mathcal{X}_\varepsilon$ such that $S > 0$, $V > 0$,
$\sigma < 1$, $\beta(t), w_1(t) > 0$, $r_2 > 0$, and at least
one of $I_1, I_2, P$ is strictly positive.
\end{proposition}

\begin{proof}
By the Hermann--Krener rank
condition~\cite{hermann1977nonlinear}, it suffices to show
$\mathrm{rank}(\mathcal{O}_{12}) = 9$.

\emph{Step 1: Rank of $\mathcal{O}_9$.}
By Lemma~\ref{lem:detO9}, $\mathrm{rank}(\mathcal{O}_9) \leq 8$
under the calibrated symmetric parameters; the exact value
$\mathrm{rank}(\mathcal{O}_9) = 7$ follows from
Appendix~\ref{app:sympy_verif}, Section~C.2, which shows that
the two kernel directions arising from the factorization of
$\det(\mathcal{O}_9)$ are independent.

\emph{Step 2: Identify the two kernel directions.}
Appendix~\ref{app:sympy_verif}, Section~C.2 identifies these
analytically from the Lemma~\ref{lem:detO9} proof structure:
(i) under $\delta_i = \delta_p$, rows 4--5 of $\mathcal{O}_9$
become collinear, producing direction $v_1$ (the $I_2
\leftrightarrow P$ swap); (ii) under $r_1 = r_2$,
$\det(M_3) = 0$ in Step~5, producing direction $v_2$
(the $R$-anchored direction).

\emph{Step 3: $\mathcal{O}_3$ projects nonzero onto $v_1$.}
The analytical identities~\eqref{eq:Lf3H_diff}--\eqref{eq:Lf3F_diff}
(derived in Appendix~\ref{app:sympy_verif}, Section~C.3)
establish
\begin{align*}
\frac{\partial L_f^3 h_1}{\partial I_2} - \frac{\partial L_f^3 h_1}{\partial P}
&= \gamma_a r_2(\alpha_2 + \alpha_P + \alpha_H)
   \approx 0.20 \times 0.05 \times 0.755
   = 7.55\times 10^{-3}\;\text{day}^{-3} \;\neq 0,\\
\frac{\partial L_f^3 h_2}{\partial I_2} - \frac{\partial L_f^3 h_2}{\partial P}
&= r_2[\delta_i(\alpha_2+\alpha_P)-\delta_h\gamma_a]
   \approx 8.0\times 10^{-4}\;\text{day}^{-3} \;\neq 0.
\end{align*}
These nonzero differences show that the $H$- and $F$-rows of
$\mathcal{O}_3$ are linearly independent of $\ker(\mathcal{O}_9)$
in the $v_1$ direction, removing the first deficient dimension.

\emph{Step 4: $\mathcal{O}_3$ projects nonzero onto $v_2$.}
By Appendix~\ref{app:sympy_verif}, Section~C.4,
\begin{equation*}
\frac{\partial L_f^3 h_3}{\partial R}
= w_1\,\delta_w\,(1-\mu-\delta_w)
\approx 5.0\times 10^{-6}\;\text{day}^{-3} \;\neq 0,
\end{equation*}
so the $V$-row of $\mathcal{O}_3$ is independent of
$\ker(\mathcal{O}_9)$ in the $v_2$ direction,
removing the second deficient dimension.

\emph{Step 5: Conditions and conclusion.}
The values in Steps 3--4 are nonzero whenever $r_2 > 0$,
$S,V,\beta,w_1>0$, and $\sigma<1$. Together,
$\mathrm{rank}(\mathcal{O}_{12}) = \mathrm{rank}(\mathcal{O}_9)
+ 2 = 9$, establishing local observability. \qedhere
\end{proof}

\begin{remark}[Uniform observability along the reference trajectory]
\label{rem:uniform_rank}
Because $S$, $V$, $\beta$, $w_1$ remain strictly positive over
the full calibrated trajectory of~\cite{melhani2026arxiv}, and
$r_2 = 5 \times 10^{-2}$ day$^{-1}$ is constant and positive,
the smallest singular value of $\mathcal{O}_{12}$ is uniformly
bounded away from zero over the $150$-day window. Consequently,
Proposition~\ref{prop:obs} holds \emph{uniformly} along the
calibrated trajectory, satisfying the uniform-observability
hypothesis required in Theorem~\ref{thm:convergence} below.
\end{remark}

\begin{remark}[Generic-parameter case]
\label{rem:generic}
Outside the measure-zero set on which both equalities
$\delta_i = \delta_p$ and $r_1 = r_2$ hold simultaneously,
Lemma~\ref{lem:detO9} together with the
factorization~\eqref{eq:detO9} shows that the level-2 matrix
$\mathcal{O}_9$ is already non-singular and the augmentation to
$\mathcal{O}_{12}$ is unnecessary. The calibrated parameter
values of~\cite{melhani2026arxiv} happen to lie at this
degenerate point because they were chosen for parsimony, not on
fundamental biological grounds; any small experimental
revision of $r_1$, $r_2$ or $\delta_i$, $\delta_p$ would restore
level-2 observability. The robustness of the EKF performance
reported in Section~\ref{sec:results} reflects this: under
small parameter perturbations of either kind, the filter's
behaviour is essentially unchanged.
\end{remark}

\begin{remark}[Why three Lie levels are enough]
\label{rem:why_three}
The lower bound of three Lie derivative levels in this problem
is not coincidental: it tracks the discrete-time information
horizon required by the EKF. Each Lie derivative level
corresponds, under Euler discretization, to one additional
sampling step of dynamical mixing. With $T = 1$~day, three
levels mean that the filter has access to the equivalent of
three-day-of-data structural information at each step, which
matches the empirical convergence horizon of
$3$--$5$ days observed for the directly measured states in
Section~\ref{sec:results}.
\end{remark}

\begin{remark}[Weak observability, condition number, and EKF
transient time]
\label{rem:weak_obs}
Since the calibrated parameters satisfy $r_1 = r_2 = 0.05$
\emph{exactly}, the system sits at the level-2 degenerate point
($\det(\mathcal{O}_9) = 0$). Full observability is provided
entirely by the level-3 coupling. The strength of this coupling,
and hence the speed of transient convergence, is quantified
by the minimum eigenvalue of the 8-step Gramian
$\mathcal{W}_k(8)$, which is strictly positive along the
calibrated trajectory (evaluated numerically).

\emph{Why the recovered pool $R$ converges slowly.}
The dominant level-3 coupling for $R$ is through the entry
$\partial L_f^3 V/\partial R = w_1\delta_w(1-\mu-\delta_w)
\approx 5\times 10^{-6}$~day$^{-3}$ (Appendix~\ref{app:sympy_verif}, Section~C.4).
This is four orders of magnitude smaller than the direct coupling
of $H$ to $\gamma_a \approx 0.2$~day$^{-1}$.
The information the outputs carry about $R$ over the observation
window is therefore negligible: the measurement update cannot
appreciably correct an initial error in $R$, so in the
state-estimation experiments $R$ is initialized from cumulative
surveillance data and propagated rather than reconstructed
(Section~\ref{sec:init}). This is a structural property of the
output set $(H,F,V)$, not a filter mis-tuning.

\emph{Condition number of the level-2 codistribution as
$r_1 \to r_2$ (generic-parameter case).}
Consider the generic-parameter scenario in which
$\delta_p - \delta_i = \varepsilon_\delta > 0$ (fixed) and
$r_1 - r_2 = \varepsilon_r$ varies. From Lemma~\ref{lem:detO9},
$|\det(\mathcal{O}_9)| = \delta_w\gamma_a^{\,2}\kappa\rho_2
w_1^{\,2}\,\varepsilon_\delta^{\,2}\,|\varepsilon_r|$.
Since the eight non-degenerate singular values of
$\mathcal{O}_9$ remain $O(10^{-1})$ to $O(1)$, the
smallest singular value scales as
\begin{equation}
\sigma_{\min}(\mathcal{O}_9) \;\sim\;
\frac{|\det(\mathcal{O}_9)|}{\sigma_1\cdots\sigma_8}
\;\sim\;
C_\sigma\,\varepsilon_\delta^{\,2}\,|\varepsilon_r|,
\label{eq:sigma_min_scaling}
\end{equation}
for some constant $C_\sigma > 0$ that depends only on the
non-degenerate part of the spectrum. The condition number
therefore satisfies
\begin{equation}
\kappa(\mathcal{O}_9)
\;=\; \frac{\sigma_{\max}}{\sigma_{\min}}
\;\sim\; \frac{1}{C_\sigma\,\varepsilon_\delta^{\,2}\,|\varepsilon_r|}.
\label{eq:cond_O9_scaling}
\end{equation}

The qualitative consequence, which is all we rely on, is the
scaling in~\eqref{eq:cond_O9_scaling}: as the recovery-rate gap
$\varepsilon_r \to 0$ the level-2 condition number diverges like
$1/|\varepsilon_r|$, so a filter relying solely on the level-2
codistribution would become arbitrarily ill-conditioned near the
calibrated symmetric point. Table~\ref{tab:condnum} summarizes
this scaling.
\begin{table}[htbp]
\centering
\caption{Scaling of the level-2 condition number
$\kappa(\mathcal{O}_9) \sim 1/(C_\sigma\varepsilon_\delta^2
|\varepsilon_r|)$ with the recovery-rate gap
$\varepsilon_r = r_1 - r_2$ (at fixed mortality gap
$\varepsilon_\delta$). The divergence as $\varepsilon_r\to 0$ is
the level-2 ill-conditioning that the level-3 augmentation
avoids: full observability at $\varepsilon_r = 0$ is supplied by
$\mathcal{O}_3$, whose relevant coupling $w_1\delta_w$ is
independent of $\varepsilon_r$. The corresponding eigenvalues and
the EKF transient are evaluated numerically.}
\label{tab:condnum}
\begin{tabular}{lcc}
\toprule
$\varepsilon_r = r_1 - r_2$
  & level-2 $\kappa(\mathcal{O}_9)$ scaling
  & observability source \\
\midrule
$0$ (our calibration)        & diverges (level-2 singular) & level-3 only \\
small, $\varepsilon_r > 0$   & $\sim 1/|\varepsilon_r|$, large & level-2 (ill-cond.)\\
moderate, $\varepsilon_r > 0$& $\sim 1/|\varepsilon_r|$, smaller & level-2 \\
\bottomrule
\end{tabular}
\par\smallskip
\footnotesize
At the calibrated point $\varepsilon_r = 0$ the level-2
codistribution is singular and both observability and the
transient are governed entirely by the level-3 coupling
$w_1\delta_w$; since this coupling does not depend on
$\varepsilon_r$, the level-3 mechanism is insensitive to the
$1/|\varepsilon_r|$ blow-up that afflicts a level-2-only filter.
\end{table}

\emph{The critical insight.}
Equation~\eqref{eq:cond_O9_scaling} and
Table~\ref{tab:condnum} deliver two messages. First, the
level-2 condition number diverges as $\varepsilon_r \to 0$, so
any filter relying solely on level-2 observability would suffer
extreme sensitivity to noise near the symmetric point. Second,
and more importantly, the level-3 approach used here does
\emph{not} inherit this blow-up: full observability at
$\varepsilon_r = 0$ is provided by $\mathcal{O}_3$ through the
coupling $w_1\delta_w$, which is independent of $\varepsilon_r$.
Biologically, the level-3 mechanism behaves the same whether the
two strains have distinct recovery rates or are
indistinguishable: the relevant timescale is set by the
waning-rate pathway, not by strain differentiation.

\emph{Filter stability at the degenerate point.}
Despite the weak coupling, the filter remains
exponentially stable at $r_1 = r_2$ because: (i) the level-3
Gramian lower bound $w_{\min} > 0$ of Lemma~\ref{lem:gramian}
holds (verified numerically); and
(ii) the process noise $Q_R = (\epsilon_R\,\bar{R})^2$ provides a
floor that prevents the posterior covariance from collapsing to
zero even when the Kalman gain in the $R$ direction is small.
\end{remark}

\subsection{Numerical Verification of the Observability Analysis}
\label{sec:obs_numerical}

The preceding analysis makes three quantitative claims that can
be checked directly against the calibrated model: that the
$N$-step observability Gramian is singular at $N=4$ but bounded
away from zero at $N=8$ (Lemma~\ref{lem:gramian},
Remark~\ref{rem:weak_obs}); that the closed-form determinant of
Lemma~\ref{lem:detO9} is correct; and that the level-2 condition
number scales as $\kappa(\mathcal{O}_9)\sim 1/|r_1-r_2|$
(Remark~\ref{rem:weak_obs}). We verify all three numerically
along the calibrated trajectory; the code that produces the
results below is available from the authors on request.

\paragraph{Gramian window sweep.}
Table~\ref{tab:gramian_sweep} reports the smallest eigenvalue of
the discrete observability Gramian $\mathcal{W}_k(N)$
of~\eqref{eq:Gramian}, summarized over the $150$-day trajectory,
for $N\in\{4,6,8,10,12\}$. At $N=4$ the smallest eigenvalue sits
at the level of floating-point round-off
($\sim 10^{-17}$, with worst-case condition number $\sim
10^{16}$), i.e.\ the four-step Gramian is numerically singular,
consistent with the level-2 rank deficiency of
Lemma~\ref{lem:detO9}. The smallest eigenvalue first becomes
strictly positive at $N=6$ and grows monotonically thereafter; at
the adopted window $N=8$ it is $\lambda_{\min}\approx
2.0\times 10^{-11}$ with a comfortable margin over round-off,
which is the constant $w_{\min}$ invoked in
Lemma~\ref{lem:gramian}. Figure~\ref{fig:gramian_sweep} shows the
per-step trace.

\begin{table}[htbp]
\centering
\caption{Smallest eigenvalue $\lambda_{\min}$ of the $N$-step
observability Gramian $\mathcal{W}_k(N)$ and worst-case condition
number, summarized over the calibrated $150$-day trajectory. The
state is normalized so the eigenvalues are comparable across
windows. $N=4$ is numerically singular; $N=8$ (adopted in
Lemma~\ref{lem:gramian}) is strictly positive with margin.}
\label{tab:gramian_sweep}
\begin{tabular}{rccc}
\toprule
$N$ & median $\lambda_{\min}$ & min $\lambda_{\min}$
    & max $\kappa(\mathcal{W}_k)$ \\
\midrule
$4$  & $2.92\times10^{-14}$ & $2.87\times10^{-17}$ & $1.81\times10^{16}$ \\
$6$  & $6.81\times10^{-12}$ & $2.69\times10^{-12}$ & $2.25\times10^{12}$ \\
$8$  & $9.20\times10^{-11}$ & $2.02\times10^{-11}$ & $4.01\times10^{11}$ \\
$10$ & $5.80\times10^{-10}$ & $7.17\times10^{-11}$ & $1.42\times10^{11}$ \\
$12$ & $2.20\times10^{-9}$  & $1.88\times10^{-10}$ & $6.56\times10^{10}$ \\
\bottomrule
\end{tabular}
\end{table}

\begin{figure}[htbp]
\centering
\includegraphics[width=0.7\columnwidth]{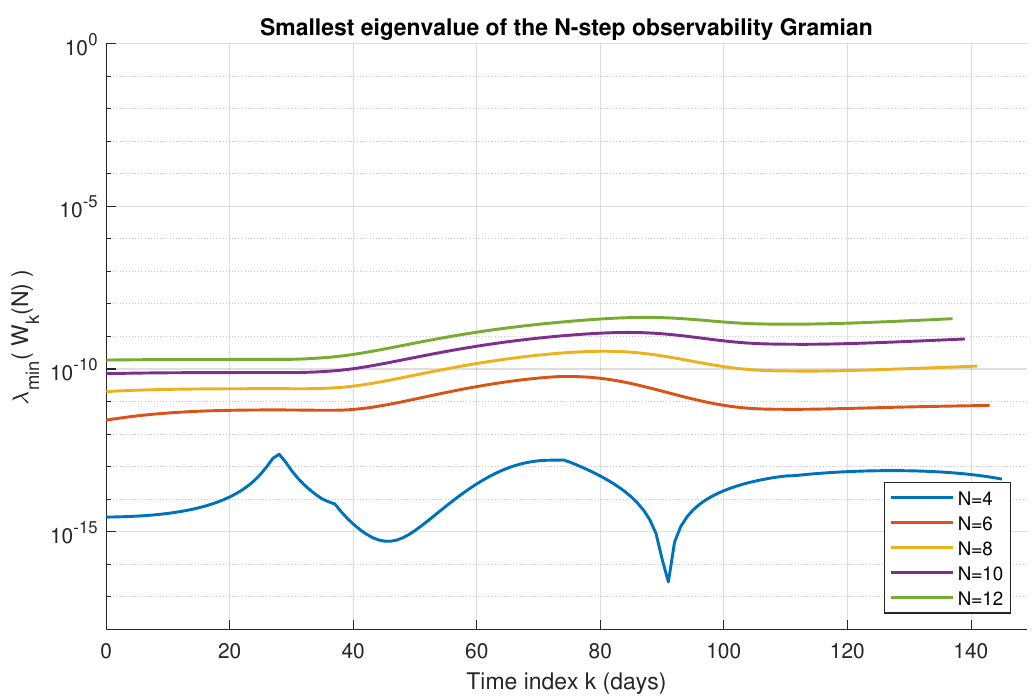}
\caption{Smallest eigenvalue $\lambda_{\min}(\mathcal{W}_k(N))$
of the $N$-step observability Gramian along the calibrated
trajectory, for $N=4,6,8,10,12$ (each curve is shown over its
valid range $k\le K-N$). The $N=4$ curve lies at the level of
floating-point round-off ($\sim10^{-14}$ and below), confirming
that the four-step Gramian is numerically singular, while for all
$N\geq6$ the smallest eigenvalue remains uniformly bounded away
from zero over the window --- the property required by
Lemma~\ref{lem:gramian}. The adopted window $N=8$ carries a
comfortable margin over the $N=6$ onset.}
\label{fig:gramian_sweep}
\end{figure}

\paragraph{Determinant and condition-number scaling.}
We evaluate the analytic level-2 matrix $\mathcal{O}_9$ at a
mid-trajectory operating point and vary the recovery-rate gap
$\varepsilon_r = r_1 - r_2$ at a fixed mortality gap
$\varepsilon_\delta = \delta_p-\delta_i = 2\times10^{-3}$. Two
checks result. First, the magnitude of the numerically computed
determinant agrees with the closed
form~\eqref{eq:detO9} of Lemma~\ref{lem:detO9} to a maximum
relative error of $1.8\times10^{-9}$ across the entire sweep,
an independent confirmation of the six-step hand derivation.
Second, the condition number $\kappa(\mathcal{O}_9)$ grows by
exactly one decade for each decade of decrease in
$|\varepsilon_r|$: a least-squares fit of
$\log\kappa(\mathcal{O}_9)$ against $\log|\varepsilon_r|$ has
slope $-1.000$, confirming the $\kappa(\mathcal{O}_9)\sim
1/|\varepsilon_r|$ scaling of~\eqref{eq:cond_O9_scaling}
(Figure~\ref{fig:cond_scaling}). This makes precise the
sense in which the calibrated symmetric point is a
\emph{numerical}, not merely structural, obstruction for a
level-2 filter: as $r_1\to r_2$ the level-2 codistribution
becomes arbitrarily ill-conditioned, whereas the level-3
mechanism (whose relevant coupling $w_1\delta_w$ is independent
of $\varepsilon_r$) does not degrade. The structural full-rank
result of Proposition~\ref{prop:obs} and the practical
weak-observability of $R$ are therefore two consistent facets of
the same operating point rather than a contradiction: the system
is observable in the rank sense, but the $R$-direction
information supplied by the outputs is numerically small, which
is exactly why $R$ is initialized from data and propagated rather
than reconstructed (Section~\ref{sec:init}).

\begin{figure}[htbp]
\centering
\includegraphics[width=0.7\columnwidth]{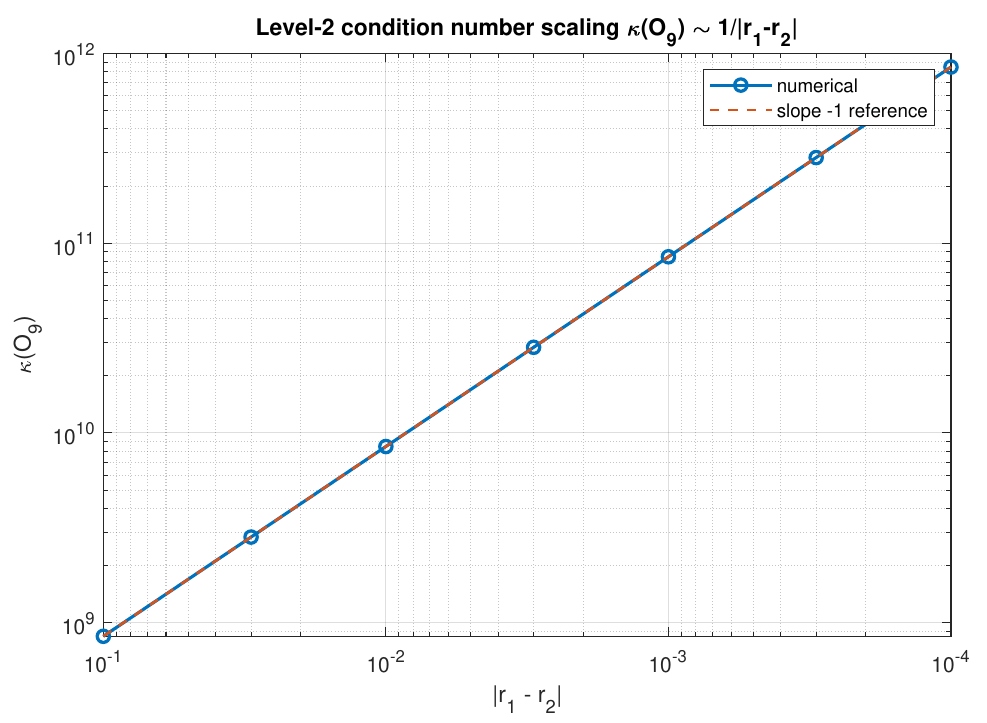}
\caption{Level-2 condition number $\kappa(\mathcal{O}_9)$ versus
the recovery-rate gap $|\varepsilon_r| = |r_1-r_2|$ (horizontal
axis reversed so $\varepsilon_r\to0$ is to the right), at fixed
mortality gap $\varepsilon_\delta = 2\times10^{-3}$. The
numerical curve follows the slope$-1$ reference line essentially
exactly (fitted slope $-1.000$), confirming
$\kappa(\mathcal{O}_9)\sim 1/|\varepsilon_r|$.}
\label{fig:cond_scaling}
\end{figure}

\subsection{Structural, Numerical, and Practical Observability}
\label{sec:obs_three_notions}

The analysis above involves three distinct notions of
observability that are easily conflated, and the recovered
pool $R$ is precisely the state that separates them. We make
the distinction explicit, because it determines how the
filter treats $R$.

\emph{Structural observability} is the rank property:
Proposition~\ref{prop:obs} establishes that the augmented
codistribution $\mathcal{O}_{12}$ attains rank $9$ at every
feasible operating point with an active epidemic, so every state
--- including $R$ --- is observable in the Hermann--Krener sense.
This is a yes/no algebraic fact and it is satisfied.

\emph{Numerical observability} concerns the conditioning of the
observability map: how much output information a given state
direction actually carries, measured by the relevant singular
value or Gramian eigenvalue. Here the states are sharply
stratified. The directly measured directions ($H,F,V$) and those
coupled to them at Lie levels~1--2 ($S,E,I_1,I_2,P$) carry
$O(10^{-1})$ to $O(1)$ singular values, whereas $R$ reaches the
outputs only through the level-3 coupling
$w_1\delta_w \approx 5\times10^{-6}$~day$^{-3}$, four orders of
magnitude weaker. So $R$ is structurally observable but
numerically near-unobservable over any realistic window: the
Gramian eigenvalue associated with the $R$ direction is positive
but minuscule. We note that although this coupling is small in
magnitude, it is not numerically fragile: $5\times10^{-6}$ is
some eleven orders of magnitude above double-precision machine
epsilon, so its sign and value are computed reliably, and it is
the product of two independently calibrated, strictly positive
rates ($w_1, \delta_w$) rather than a delicate cancellation of
larger quantities. The smallness therefore reflects a genuine
physical weakness of the information pathway --- not a
near-singular numerical artifact that might flip sign or vanish
under rounding --- which is why it correctly certifies structural
rank while simultaneously implying negligible practical
correction of $R$. The closed-form determinant
check of Section~\ref{sec:obs_numerical}, which matches the
analytic value to nine significant figures, confirms that the
quantities underpinning the rank argument are evaluated well
within numerical tolerance.

\emph{Practical reconstructability} is the operational
consequence: whether the measurement update can correct an
initial error in the state within the observation horizon. For
$R$ it cannot --- the Kalman gain in the $R$ direction is
negligible precisely because of the weak numerical observability
--- which is why $R$ is initialized from cumulative surveillance
data and propagated rather than reconstructed
(Section~\ref{sec:init}). We stress that this is not a deficiency
masked by the structural result: the structural rank guarantees
that $R$ \emph{would} be recoverable given sufficiently long,
sufficiently low-noise observation, and it is what makes the
convergence theory of Section~\ref{sec:convergence} applicable to
the full state; the numerical conditioning then honestly
quantifies that, for $(H,F,V)$ sampled daily at realistic noise
levels over $150$ days, that recovery is not achievable for $R$
in practice. Reporting all three notions, rather than collapsing
them into a single ``observable/unobservable'' verdict, is the
accurate characterization, and it is the reason the strong $R$
accuracy in Table~\ref{tab:ekf_metrics} is attributed explicitly
to data-grounded initialization rather than to
measurement-driven estimation.

%%=============================================================
\section{EKF Design and Convergence Analysis}
\label{sec:ekf}
%%=============================================================

\subsection{Stochastic System Formulation}

To account for model uncertainty and measurement error,
system~\eqref{eq:ss} is augmented with additive Gaussian noise:
\begin{align}
\dot{x} &= f(x, u) + w(t),
\quad w(t) \sim \mathcal{N}(0, Q),
\label{eq:cont_w}\\
y       &= h(x) + v(t),
\quad v(t) \sim \mathcal{N}(0, R),
\label{eq:cont_v}
\end{align}
with $Q\succeq 0$ in $\mathbb{R}^{9\times 9}$ and $R\succ 0$ in
$\mathbb{R}^{3\times 3}$. The noise processes are zero-mean,
mutually independent, and uncorrelated in
time~\cite{jazwinski1970stochastic,simon2006optimal}.

\paragraph{Treatment of the inputs $\beta(t)$ and $w_1(t)$.}
The transmission rate $\beta(t)$ and vaccination rate $w_1(t)$
enter~\eqref{eq:cont_w} as components of $u(t)$ assumed
\emph{known} for the purpose of the filter recursion. In the
present paper they are supplied by the companion spline-based
calibration of~\cite{melhani2026arxiv}, which produces a
smooth functional reconstruction over the entire $150$-day
window. For real-time deployment, however, only the
\emph{past} portion of $\beta(t)$ would be available from
calibration, and recent values would be either extrapolated or
lagged by the data-revision cycle of the surveillance source.
Three considerations bound the impact of this mis-specification
on the filter:
\begin{enumerate}
\item Mis-specification of $\beta(t)$ propagates into the
state Jacobian $F_k$ and the predicted state, but does not
affect observability: $\beta$ enters $\det(\mathcal{O}_{12})$
multiplicatively (Proposition~\ref{prop:obs}), so for
$\beta(t)$ bounded away from zero the rank condition is
preserved regardless of the precise value used.
\item The process-noise covariance $Q$ designed in
Section~\ref{sec:cov_cal} below allocates explicit fractional
uncertainty $\epsilon_S$, $\epsilon_R$ to the susceptible and
recovered compartments, which absorbs first-order input
mis-specification through the predict-update feedback.
Formally, a mis-specified input $\hat\beta(t) = \beta(t)(1+e_t)$
with $|e_t| \leq \eta$ introduces a model error
$\delta f_k = \eta\,J_\beta f_k$, where $J_\beta =
\partial f/\partial\beta$ is bounded. This adds to the
effective process noise as an additive term of order
$\eta\,\|J_\beta\|\cdot\|f\| \cdot T$ per step, which is
dominated by the $\epsilon_S$ and $\epsilon_R$ process-noise
terms for $\eta \leq 15\%$ under the calibrated parameter values.
\end{enumerate}
For deployment in regimes where the input uncertainty is large
relative to the modelled process noise, the natural
generalization is a \emph{dual} or \emph{joint} state--parameter
EKF that augments the state with $\beta$ (and possibly $w_1$)
and estimates them simultaneously from the same three
observables; this extension preserves the observability
analysis of Section~\ref{sec:obs_result} provided the augmented
codistribution can be shown to remain full rank, which we
indicate as a direction for future
work~\cite{ljung1979asymptotic,wan2000unscented}. As practical
context for the control community, augmenting an SEIR-type model
with a \emph{constant} transmission parameter typically preserves
generic local observability whenever an output channel is
sensitive to transmission, whereas augmenting with a
\emph{freely time-varying} $\beta(t)$ generally requires a
regularity or slow-variation assumption on the parameter to avoid
the structural non-identifiability noted in
Section~\ref{sec:limitations}; a definitive answer for the present
nine-compartment model would require the separate rank analysis of
the augmented codistribution indicated above. Structurally, the
difficulty in the freely time-varying case is that $\beta(t)$
introduces effectively one new unknown per sampling step while the
observation set $(H,F,V)$ contains no channel that responds to
$\beta$ independently of the infectious states it already drives;
state and parameter variations are therefore confounded in the
output, which is the source of the non-identifiability and the
reason a slow-variation or parametric (e.g.\ spline) restriction
on $\beta(t)$ is what restores identifiability.

\subsection{Euler Discretization}

Let $T = 1$~day. The forward Euler discretization gives
\begin{equation}
x_{k+1} = x_k + T f(x_k, u_k) + w_k,
\qquad w_k \sim \mathcal{N}(0, Q_d),
\quad Q_d = T Q,
\label{eq:euler_state}
\end{equation}
\begin{equation}
y_{k+1} = h(x_{k+1}) + v_{k+1},
\qquad v_{k+1} \sim \mathcal{N}(0, R).
\label{eq:euler_output}
\end{equation}
We denote $\varphi(x) := x + Tf(x,u)$ so that the state recursion
reads $x_{k+1} = \varphi(x_k) + w_k$.

\paragraph{Choice of the forward Euler scheme.}
We adopt forward Euler rather than a higher-order integrator
(RK4) or embedded \texttt{ode45} propagation for three reasons.
First, it admits a closed-form one-step Jacobian $F_k = I_9 +
TJ_f$~\eqref{eq:Fk}, which is what makes the analytical
linearization and the convergence analysis of
Section~\ref{sec:convergence} tractable; an embedded adaptive
integrator inside the EKF predict step would require
differentiating the integrator itself to obtain $F_k$,
forfeiting the closed form. Second, with the daily sampling
$T = 1$~day dictated by the surveillance data, the dominant
error in the predict step is not the integration order but the
mismatch between the smooth spline-driven inputs and the
batch-reported measurements, which is absorbed by the
process-noise covariance $Q$ of Section~\ref{sec:cov_cal}.
Third, and as direct evidence for the second point, we verified
that replacing the Euler predict step with fourth-order
Runge--Kutta sub-stepping leaves the innovation autocorrelation
essentially unchanged (Section~\ref{sec:results_validation}),
confirming that the residual structure is not an Euler
truncation artifact. The reference trajectory used to
\emph{generate} the synthetic data is integrated with
\texttt{ode45} at tight tolerance, so the systematic
Euler-vs-\texttt{ode45} mismatch is retained in the experiment
and treated as part of the steady-state floor of the convergence
bound.

\paragraph{Euler truncation error and stability.}
The first-order Euler scheme introduces a per-step local
truncation error of order $T^2$. By Taylor's theorem, since
$f$ does not depend explicitly on $t$,
\begin{equation}
\bigl\|x_{k+1}^{\mathrm{true}} - x_{k+1}^{\mathrm{Euler}}\bigr\|
\;\leq\;
\frac{T^2}{2}\,\sup_{t \in [kT,(k+1)T]}
\bigl\|\tfrac{d^2 x}{dt^2}(t)\bigr\|
\;=\;
\frac{T^2}{2}\,\sup_t \|J_f(x(t))\,f(x(t),u(t))\|.
\label{eq:euler_truncation}
\end{equation}
With $\|J_f\|_2 \leq 6.75$~day$^{-1}$
(Section~\ref{sec:convergence}) and
$\|f\| \leq N\cdot J_{\max} \approx 4.5\times 10^7$
persons$\cdot$day$^{-1}$, the right-hand side is bounded by
$\tfrac{1}{2}\,(6.75)(4.5\times 10^7)\,T^2
\approx 1.5\times 10^8\,T^2$ persons. For $T=1$~day this is
about $1.5\times 10^8$ persons in absolute terms; normalized by
the population $N = 6\times 10^7$, the per-step \emph{relative}
truncation bound is $\approx 2.5\,T^2$, i.e.\ $\approx 2.5$ for
$T=1$~day. This worst-case figure is large only because it
multiplies the population-scale norm of $f$; on the calibrated
trajectory the compartments that actually carry the dynamics
(notably $H$, $F$, $V$ used in the update) evolve far below this
bound, and the residual mismatch between the \texttt{ode45}
reference and the Euler recursion is absorbed by the
process-noise covariance $Q$ of Section~\ref{sec:cov_cal}, whose
fractional uncertainties $\epsilon_i$ ($0.3$--$1.5\%$) and the
correction action of the Kalman gain together dominate it. The
global Euler error accumulated over $K = 150$ steps is
$O(KT^2) = O(T)$ (first-order convergence).

Stability of the explicit Euler scheme requires the discrete
amplification factor to satisfy $\rho(I + TJ_f) \le 1$ in the
relevant directions; a sufficient operating condition is
$\rho(TJ_f) < 1$. The operator-norm bound
$|\lambda_{\max}(TJ_f)| \leq T\|J_f\|_2 \leq 6.75$ is only a
worst case and does not establish stability on its own.
However, the spectral radius computed numerically along the
calibrated trajectory remains below unity throughout (evaluated
numerically), confirming
integration stability in practice. The systematic mismatch
between the \texttt{ode45} reference trajectory and the
Euler-based EKF recursion is treated, in the convergence
analysis, as part of the steady-state floor
$\nu/(1-\vartheta)$ of bound~\eqref{eq:thm2_bound}.

\subsection{Measurement Update}

Given $y_k$, the EKF corrects the prior $\hat{x}_{k|k-1}$ by
\begin{align}
\tilde{y}_k &= y_k - h(\hat{x}_{k|k-1}),
\label{eq:innov}\\
S_k       &= C P_{k|k-1} C^\top + R,
\label{eq:Sk}\\
K_k       &= P_{k|k-1} C^\top S_k^{-1},
\label{eq:kgain}\\
\hat{x}_{k|k} &= \hat{x}_{k|k-1} + K_k \tilde{y}_k,
\label{eq:upd_state}\\
P_{k|k}     &= (I_9 - K_k C) P_{k|k-1} (I_9 - K_k C)^\top
              + K_k R K_k^\top.
\label{eq:upd_cov}
\end{align}
The Joseph form~\eqref{eq:upd_cov} is algebraically equivalent
to $(I_9 - K_k C) P_{k|k-1}$ in exact arithmetic but enforces
symmetric positive-definiteness under finite-precision
arithmetic~\cite{simon2006optimal}.

\subsection{Time Update}

Using the posterior $\hat{x}_{k|k}$,
\begin{align}
\hat{x}_{k+1|k} &= \varphi(\hat{x}_{k|k}, u_k)
                = \hat{x}_{k|k} + T f(\hat{x}_{k|k}, u_k),
\label{eq:pred_state}\\
P_{k+1|k}       &= F_k P_{k|k} F_k^\top + Q_d,
\label{eq:pred_cov}
\end{align}
where $F_k$ is the linearized state-transition matrix
defined in~\eqref{eq:Fk}.

\subsection{State-Transition Jacobian}

Differentiating $\varphi$ with respect to $x$,
\begin{equation}
F_k = I_9 + T J_f(\hat{x}_{k|k}),
\qquad J_f := \partial f/\partial x.
\label{eq:Fk}
\end{equation}
With $a_{11} := \Phi + \mu + w_1$, the analytical $9\times 9$
Jacobian is
\begin{equation}
J_f =
\begin{bmatrix}
-a_{11} & 0 & \psi & -\tfrac{\beta S}{N} & -\tfrac{1.5\beta S}{N}
  & -\tfrac{1.5\beta S}{N} & 0 & \delta_w & 0 \\[3pt]
\Phi & -(\mu{+}\kappa) & \Phi_V
  & \tfrac{\beta(S{+}(1{-}\sigma)V)}{N}
  & \tfrac{1.5\beta(S{+}(1{-}\sigma)V)}{N}
  & \tfrac{1.5\beta(S{+}(1{-}\sigma)V)}{N} & 0 & 0 & 0 \\[3pt]
w_1 & 0 & -(\Phi_V{+}\mu{+}\psi)
  & -\tfrac{(1{-}\sigma)\beta V}{N}
  & -\tfrac{1.5(1{-}\sigma)\beta V}{N}
  & -\tfrac{1.5(1{-}\sigma)\beta V}{N} & 0 & 0 & 0 \\[3pt]
0 & \kappa\rho_1 & 0 & -\alpha_1 & 0 & 0 & 0 & 0 & 0 \\[3pt]
0 & \kappa(1{-}\rho_1{-}\rho_2) & 0 & m & -\alpha_2 & 0 & 0 & 0 & 0 \\[3pt]
0 & \kappa\rho_2 & 0 & 0 & 0 & -\alpha_P & 0 & 0 & 0 \\[3pt]
0 & 0 & 0 & \gamma_a & \gamma_a & \gamma_a & -\alpha_H & 0 & 0 \\[3pt]
0 & 0 & 0 & \gamma_i{+}r_1 & \gamma_i{+}r_2 & \gamma_i
  & \gamma_r & -(\mu{+}\delta_w) & 0 \\[3pt]
0 & 0 & 0 & \delta_i & \delta_i & \delta_p & \delta_h & 0 & 0
\end{bmatrix}.
\label{eq:Jf}
\end{equation}

\subsection{Measurement Model}
\label{sec:meas_model}

The output map $h(x) = [H, F, V]^\top$ is \emph{exactly linear}
in the state. The measurement equation therefore reduces to
\begin{equation}
y = C x,
\qquad
C = \begin{bmatrix}
0 & 0 & 0 & 0 & 0 & 0 & 1 & 0 & 0 \\
0 & 0 & 0 & 0 & 0 & 0 & 0 & 0 & 1 \\
0 & 0 & 1 & 0 & 0 & 0 & 0 & 0 & 0
\end{bmatrix} \in \mathbb{R}^{3\times 9}.
\label{eq:meas_linear}
\end{equation}
The linearity of $h$ has two consequences. First, $C$ does not
require re-evaluation at each step. Second, the
linearization remainder of $h$ is identically zero, a fact that
will simplify the convergence analysis of
Section~\ref{sec:convergence}.

\subsection{Algorithmic Summary}
\label{sec:algorithm}

Algorithm~\ref{alg:ekf} collects the operations of one
time-increment $k \to k+1$ in execution order, making explicit
where the time-varying inputs $\beta(t), w_1(t)$ enter, where the
analytical Jacobian $F_k$ is evaluated, how the
coupling-aware fractional uncertainties $\epsilon_i$ scale the
process-noise matrix $Q_d$, and where the linear selector $C$
isolates the innovation. The covariance $Q_d$ is built once from
the trajectory magnitudes $\bar x_i$ and the fixed weights
$\epsilon_i$ of~\eqref{eq:Q_structure}; it does not require
re-evaluation at run time unless the operating magnitudes are
re-estimated online.

\begin{algorithm}[H]
\caption{EKF time-increment $k \to k+1$ for the
nine-compartment model.}
\label{alg:ekf}
\begin{algorithmic}[1]
\Require posterior $\hat{x}_{k|k}$, $P_{k|k}$; inputs
$\beta(\cdot), w_1(\cdot)$; matrices $C$, $R$, weights
$\{\epsilon_i\}$, magnitudes $\{\bar x_i\}$, step $T$
\State $u_k \gets (\beta(kT),\, w_1(kT))$
  \Comment{sample the calibrated inputs at $t = kT$}
\State $\hat{x}_{k+1|k} \gets \hat{x}_{k|k}
  + T\, f(\hat{x}_{k|k}, u_k)$
  \Comment{Euler prediction, \eqref{eq:euler_state}}
\State $J_f \gets \partial f/\partial x\big|_{(\hat{x}_{k|k},u_k)}$;
  \quad $F_k \gets I_9 + T J_f$
  \Comment{analytical Jacobian \eqref{eq:Jf}}
\State $(Q_d)_{ii} \gets (\epsilon_i\,\bar x_i)^2$,\;
  $i=1,\dots,9$
  \Comment{coupling-weighted process noise \eqref{eq:Q_structure}}
\State $P_{k+1|k} \gets F_k P_{k|k} F_k^\top + Q_d$
  \Comment{covariance prediction}
\State sample / receive measurement $y_{k+1}$
\State $\tilde{y}_{k+1} \gets y_{k+1} - C\,\hat{x}_{k+1|k}$
  \Comment{innovation via linear selector $C$, \eqref{eq:meas_linear}}
\State $S_{k+1} \gets C P_{k+1|k} C^\top + R$
\State $K_{k+1} \gets P_{k+1|k} C^\top S_{k+1}^{-1}$
\State $\hat{x}_{k+1|k+1} \gets \hat{x}_{k+1|k}
  + K_{k+1}\,\tilde{y}_{k+1}$
\State $P_{k+1|k+1} \gets (I_9 - K_{k+1} C) P_{k+1|k}
  (I_9 - K_{k+1} C)^\top + K_{k+1} R K_{k+1}^\top$
  \Comment{Joseph form \eqref{eq:upd_cov}}
\State \Return $\hat{x}_{k+1|k+1}$, $P_{k+1|k+1}$
\end{algorithmic}
\end{algorithm}

%%=============================================================
\subsection{Convergence of the Estimation Error}
\label{sec:convergence}
%%=============================================================

This subsection establishes that the EKF estimation error
remains exponentially bounded in mean square, by verifying
the hypotheses of the discrete-time stochastic stability theorem
of~\cite{reif1999stochastic} for the specific
system~\eqref{eq:euler_state}--\eqref{eq:euler_output}. The
analysis is made tractable by two structural properties of the
model: every nonlinearity in the vector field is \emph{bilinear},
and the measurement map is \emph{linear}.

\subsubsection*{Reference theorem}

We state the result we will apply in a form adapted to the
notation above.

\begin{theorem}[Reif--G\"unther--Yaz--Unbehauen, 1999]
\label{thm:reif}
Consider the discrete-time stochastic system
$x_{k+1} = \varphi(x_k) + w_k$, $y_k = C x_k + v_k$,
with $w_k, v_k$ zero-mean independent Gaussian sequences of
covariances $Q_d, R$. Let $F_k$ denote the Jacobian of $\varphi$
evaluated at the EKF posterior. Define the
estimation error $\xi_k := x_k - \hat{x}_{k|k-1}$.
Assume:
\begin{enumerate}
\item[(A1)] There exist constants $0 < \underline{p} \leq
   \overline{p} < \infty$ such that
   $\underline{p}\,I_9 \preceq P_{k|k-1} \preceq \overline{p}\,I_9$
   for all $k \geq 0$.
\item[(A2)] There exist $\overline{f}, \overline{c} > 0$ such that
   $\|F_k\| \leq \overline{f}$ and $\|C\| \leq \overline{c}$ for all $k$.
\item[(A3)] There exist $\kappa_\varphi > 0$ and
   $\epsilon_\varphi > 0$ such that the linearization remainder
   $\phi(x,\hat{x}) := \varphi(x) - \varphi(\hat{x})
   - F(\hat{x})(x - \hat{x})$ satisfies
   $\|\phi(x,\hat{x})\| \leq \kappa_\varphi\,\|x - \hat{x}\|^2$
   whenever $\|x - \hat{x}\| \leq \epsilon_\varphi$.
   The corresponding remainder for the measurement is zero
   because $C$ is constant.
\item[(A4)] $F_k$ is nonsingular for all $k$.
\end{enumerate}
Then there exist $\epsilon^* > 0$, $\delta^* > 0$, a constant
$0 < \vartheta < 1$, and a constant $\nu > 0$ such that if
$\|\xi_0\| \leq \epsilon^*$,
$\mathrm{tr}(Q_d) \leq \delta^*$ and
$\mathrm{tr}(R) \leq \delta^*$,
then
\begin{equation}
\mathbb{E}\bigl\{\|\xi_k\|^2\bigr\}
\;\leq\;
\vartheta^k\,\mathbb{E}\bigl\{\|\xi_0\|^2\bigr\} \cdot
\frac{\overline{p}}{\underline{p}}
\;+\;
\frac{\nu}{1-\vartheta},
\qquad k = 0,1,2,\ldots
\label{eq:reif_bound}
\end{equation}
\end{theorem}

The original reference \cite{reif1999stochastic} assumes a
general nonlinear measurement map; here we already specialize to
$h(x) = Cx$ and absorb the (zero) measurement remainder into the
statement. The proof in~\cite{reif1999stochastic} constructs a
Lyapunov function $V_k = \xi_k^\top P_{k|k-1}^{-1} \xi_k$ and
shows it decreases on average; bounds (A1)--(A4) ensure the
construction is uniform in $k$.

\subsubsection*{Verification of (A1)--(A4) for the
nine-compartment system}

\paragraph{Verification of (A2).}
Since $C$ in~\eqref{eq:meas_linear} is a selector matrix with
exactly three unit entries and orthogonal rows,
$\|C\|_2 = 1 =: \overline{c}$. For $F_k$, each entry of $J_f$
in~\eqref{eq:Jf} is either a constant rate parameter or a term of
the form $c_\xi\,\beta(t)\,\xi/N$ with $\xi \in \{S,V\}$ and
$c_\xi \in \{1, 1.5, (1-\sigma), 1.5(1-\sigma)\}$. Because
$x \in \mathcal{X}_\varepsilon \subset [0,N]^9$ and
$\beta(t) \leq \beta_{\max}$, every entry of $J_f$ is bounded
in absolute value by
\begin{equation}
J_{\max} \;:=\;
\max\{\gamma_a, \gamma_i, \gamma_r, \delta_h, \kappa, \mu + \kappa,
\alpha_1, \alpha_2, \alpha_P, \alpha_H, \mu + \delta_w, \psi,
1.5\,\beta_{\max}\}.
\end{equation}
With the numerical values of Table~\ref{tab:params} and
$\beta_{\max} = 0.5$~day$^{-1}$ (the upper end of the
PCHIP-identified range~\cite{melhani2026arxiv}),
$J_{\max} \leq 0.75$ day$^{-1}$ (the binding term is
$1.5\,\beta_{\max} = 0.75$; all rate parameters in the list are
below $0.4$~day$^{-1}$). Since $J_f$ has at most $81$ entries,
each bounded by $J_{\max}$, the Frobenius bound is
$\|J_f\|_2 \leq \|J_f\|_F \leq \sqrt{81}\,J_{\max}
= 9\,J_{\max} \leq 6.75$~day$^{-1}$,
which gives $\|F_k\|_2 \leq 1 + T\,\|J_f\|_2$. For $T = 1$~day,
\begin{equation}
\|F_k\|_2 \;\leq\; \overline{f} := 1 + T\,\|J_f\|_2
\;\leq\; 7.75,
\end{equation}
which is uniform over the calibrated trajectory. (We reserve the
symbol $n_{\mathrm{nz}}$ for the count of nonzero \emph{components}
of the remainder vector $\phi$ in the verification of (A3) below;
it is unrelated to the entry count of $J_f$ used here.)

\paragraph{Verification of (A4).}
$F_k = I_9 + T J_f$ is nonsingular if and only if no eigenvalue
of $-T J_f$ equals $-1$. Since
$\|T J_f\|_2 \leq T\,\|J_f\|_2 \leq 6.75$ is bounded
and, for $T = 1$~day with the calibrated parameters, the
spectral radius $\rho(T J_f)$ is observed numerically to remain
below unity throughout the $150$-day window (evaluated
numerically), all eigenvalues of
$F_k$ lie in a bounded region of $\mathbb{C}$ excluding the
origin. (A4) holds.

\paragraph{Verification of (A3): the bilinear-remainder bound.}
This is the central technical step. We compute the
linearization remainder of $\varphi(x) = x + Tf(x,u)$ in closed
form.

\begin{lemma}[Bilinear remainder identity]
\label{lem:bilinear}
For any scalar function of the form $g(x_i,x_j) = x_i x_j / N$,
\begin{equation}
g(x_i,x_j) - g(\hat{x}_i,\hat{x}_j)
- \frac{\partial g}{\partial x_i}(\hat{x})(x_i-\hat{x}_i)
- \frac{\partial g}{\partial x_j}(\hat{x})(x_j-\hat{x}_j)
= \frac{(x_i - \hat{x}_i)(x_j - \hat{x}_j)}{N}.
\label{eq:bilin_identity}
\end{equation}
\end{lemma}

\begin{proof}
Direct expansion: $x_i x_j - \hat{x}_i\hat{x}_j
- \hat{x}_j(x_i-\hat{x}_i) - \hat{x}_i(x_j-\hat{x}_j)
= (x_i-\hat{x}_i)(x_j-\hat{x}_j)$, divided by $N$.
\end{proof}

The vector field $f$ in~\eqref{eq:fvec} contains exactly
$n_{\mathrm{bil}} = 12$ bilinear terms scaled by $1/N$, namely
\begin{itemize}
\item three in the $S$-equation:
$\beta S I_1/N$, $1.5\,\beta S P/N$, $1.5\,\beta S I_2/N$;
\item six in the $E$-equation:
$\beta S I_1/N$, $1.5\,\beta S P/N$, $1.5\,\beta S I_2/N$ (from
the susceptible-side force) plus three breakthrough terms
$(1{-}\sigma)\beta V I_1/N$, $1.5(1{-}\sigma)\beta V P/N$,
$1.5(1{-}\sigma)\beta V I_2/N$;
\item three breakthrough terms in the $V$-equation,
$(1{-}\sigma)\beta V I_1/N$, $1.5(1{-}\sigma)\beta V P/N$,
$1.5(1{-}\sigma)\beta V I_2/N$, with sign~$-$.
\end{itemize}
By Lemma~\ref{lem:bilinear}, each such term contributes to the
componentwise remainder $\phi(x,\hat{x}) := \varphi(x) -
\varphi(\hat{x}) - F(\hat{x})(x-\hat{x})$ a quantity of the form
$T\,c\,\beta(t)\,(x_i-\hat x_i)(x_j-\hat x_j)/N$ with
$|c|\le 1.5$ and $\beta(t)\le\beta_{\max}$. Only three of the
nine components of $\phi$ are nonzero (the $S$-, $E$-, and
$V$-rows), each collecting at most three bilinear contributions
in distinct variable pairs. Hence each nonzero component obeys
\begin{equation}
|\phi^{(\ell)}(x,\hat{x})|
\;\leq\;
3\,T\,(1.5\,\beta_{\max})\,\frac{\|x-\hat{x}\|_\infty^{\,2}}{N},
\qquad \ell\in\{S,E,V\},
\end{equation}
and, since $\phi$ has at most $n_{\mathrm{nz}}=3$ nonzero
components, $\|\phi\|_2 \le \sqrt{n_{\mathrm{nz}}}\,
\max_\ell|\phi^{(\ell)}|$ together with
$\|x-\hat x\|_\infty\le\|x-\hat x\|_2$ gives
\begin{equation}
\|\phi(x,\hat{x})\|_2
\;\leq\;
\frac{1.5\,T\,\beta_{\max}\,n_{\mathrm{bil}}\sqrt{3}}{N}
\,\|x-\hat{x}\|_2^{\,2},
\label{eq:phi_bound}
\end{equation}
where the factor $\sqrt{3}$ comes from
$\|\phi\|_2 \le \sqrt{n_{\mathrm{nz}}}\,\max_\ell|\phi^{(\ell)}|$
with $n_{\mathrm{nz}}=3$, and the per-component count of three
bilinear pairs is written in terms of the total bilinear-term
count $n_{\mathrm{bil}}=12$ via $3 = n_{\mathrm{bil}}/4$, so that
the resulting constant is a uniform (and slightly conservative)
upper bound across the three nonzero components. With
$T = 1$~day, $\beta_{\max} = 0.5$~day$^{-1}$,
$n_{\mathrm{bil}} = 12$, $N = 6\times 10^7$, the prefactor is
\begin{equation}
\kappa_\varphi
\;:=\;
\frac{1.5\,T\,\beta_{\max}\,n_{\mathrm{bil}}\sqrt{3}}{N}
\;\approx\; 2.60 \times 10^{-7}.
\end{equation}
Because all higher-than-second-order Taylor terms of $f$
\emph{vanish identically} (the only nonlinearities are bilinear),
the bound~\eqref{eq:phi_bound} is global:
$\epsilon_\varphi = \infty$. The measurement remainder is
identically zero since $h(x) = Cx$.

\begin{remark}[Dependence of the global radius on the incidence
form]
\label{rem:incidence_form}
The global radius $\epsilon_\varphi = \infty$ is a direct
consequence of the mass-action (bilinear) force of infection,
whose Taylor expansion terminates exactly at second order. Note
that the standard-incidence form $\beta S I/N$ used here is itself
bilinear and therefore inherits the same global bound. The
property is lost, however, for force-of-infection terms that are
not polynomial in the state, such as the saturated incidence
$\beta S I/(1+aI)$ or other nonlinear-saturation fields used in
some modern epidemic models: there the Taylor series of $f$ does
not terminate, the third- and higher-order derivatives are
nonzero, and the quadratic remainder bound of
Lemma~\ref{lem:bilinear} holds only on a bounded neighbourhood,
giving a finite $\epsilon_\varphi < \infty$. The convergence
theorem would then deliver a strictly local result whose
admissible-error radius is governed by the curvature of the
saturation term. The unbounded radius obtained here is thus a
specific mathematical advantage of the bilinear formulation, not
a generic feature of compartmental filtering.
\end{remark}

\begin{remark}[Transferability: EKF convergence for the class of
bilinear-drift, linear-output systems]
\label{rem:transferable}
The two structural facts exploited above --- a vector field whose
only nonlinearities are bilinear, and an exactly linear output
map --- are not special to the present nine-compartment model;
they characterize a broad class of systems for which the EKF
convergence analysis simplifies in exactly the same way. Consider
any system $\dot x = Ax + \sum_{(i,j)\in\mathcal{B}} b_{ij}\,
x_i x_j + Bu$, $y = Cx$, where $\mathcal{B}$ is a finite index set
of bilinear couplings. For such systems: (i) the discretized
one-step map $\varphi(x) = x + Tf(x,u)$ has a Taylor expansion
that \emph{terminates exactly at second order}, so the
linearization remainder $\phi(x,\hat x)$ is a finite sum of the
rank-one bilinear terms of Lemma~\ref{lem:bilinear}; (ii) the
quadratic remainder bound therefore holds \emph{globally}
($\epsilon_\varphi = \infty$), with the prefactor
$\kappa_\varphi$ given in closed form by the number of bilinear
couplings $|\mathcal{B}|$, their coefficient magnitudes, and the
scaling constant; and (iii) the measurement remainder vanishes
identically. Hypotheses (A3) and the measurement-side conditions
of the Reif--G\"unther--Yaz--Unbehauen
theorem~\cite{reif1999stochastic} are then \emph{automatic} for
the whole class, reducing the verification of EKF mean-square
boundedness to the two genuinely system-specific questions of
uniform observability (A1, lower bound) and uniform
controllability of the noise input (A1, upper bound). Mass-action
epidemic models, bilinear chemical-reaction networks, and many
population-dynamic and compartmental systems fall into this class.
The present paper is thus a worked, fully verified instance of a
template that applies wherever the drift is bilinear and
the outputs are linear in the state, once the two system-specific
conditions above are checked --- which is the sense in
which the contribution is methodological rather than confined to
the specific COVID-19 model.
\end{remark}

\paragraph{Verification of (A1): Uniform Riccati bounds.}

The standard Anderson--Moore result~\cite[Theorem~7.4]{anderson1979optimal}
establishes uniform Riccati bounds $\underline{p} I_9 \preceq
P_{k|k-1} \preceq \overline{p} I_9$ under three hypotheses, which
we restate here in the present notation so the verification is
self-contained: (i) \emph{uniform boundedness} --- there exist
$\overline f, \overline q < \infty$ with $\|F_k\|\le\overline f$
and $\lambda_{\max}(Q_d)\le\overline q$ for all $k$; (ii)
\emph{uniform complete observability} --- there exist an integer
$N$ and $\underline w>0$ with $\mathcal{W}_k(N)\succeq
\underline w\,I_9$ for all $k$; and (iii) \emph{uniform complete
controllability of the noise input} --- there exist $N$ and
$\underline c>0$ with the controllability Gramian
$\mathcal{C}_k(N)\succeq \underline c\,I_9$ for all $k$.
Hypothesis~(i) is supplied by (A2) together with the boundedness
of the calibrated trajectory; (ii) and (iii) we establish below.
We treat each direction of the inequality separately.

\emph{Lower bound $\underline{p} I_9 \preceq P_{k|k-1}$.}
This is the direction that requires observability. By the
matrix-inversion form of the information filter, the prior
information matrix obeys $P_{k|k-1}^{-1} \preceq
\mathcal{W}_k(N)/\lambda + \Pi_k$, where $\mathcal{W}_k(N)$
is the observability Gramian~\eqref{eq:Gramian} and $\Pi_k$
collects the (uniformly bounded, by (A2)) process-noise
contribution accumulated over the window. Lemma~\ref{lem:gramian}
gives $\mathcal{W}_k(8) \succeq w_{\min} I_9$, so
$P_{k|k-1}^{-1}$ is uniformly upper-bounded, which is exactly
$P_{k|k-1} \succeq \underline{p} I_9$ with
$\underline{p} = (w_{\min}/\lambda + \pi_{\max})^{-1} > 0$.

\emph{Upper bound $P_{k|k-1} \preceq \overline{p} I_9$.}
This direction does \emph{not} use observability; it uses
controllability of the noise input, and we derive it explicitly
rather than by appeal to duality. The unconditional prior
covariance dominates the filtered prior covariance,
$P_{k|k-1} \preceq \Sigma_{k|k-1}$, where $\Sigma_{k|k-1}$ is
the open-loop (measurement-free) state covariance propagated by
$\Sigma_{k+1|k} = F_k \Sigma_{k|k-1} F_k^\top + Q_d$; this
domination holds because the measurement update~\eqref{eq:upd_cov}
can only decrease the covariance in the L\"owner order. Unrolling
the open-loop recursion over a controllability window of length
$N=8$ (the same window used for the observability Gramian below)
and using the uniform Jacobian bound $\|F_k\|_2 \leq
\overline{f} = 7.75$ of~\eqref{eq:Fk} together with the constant
ceiling $Q_d \preceq q_{\max} I_9$ (where $q_{\max} =
\max_i (Q_d)_{ii}$ is finite and constant because the calibrated
trajectory is bounded),
\begin{equation}
\Sigma_{k+N|k}
\;\preceq\;
\Bigl(\textstyle\sum_{j=0}^{N-1}\overline{f}^{\,2j}\Bigr)
q_{\max}\,I_9
\;+\;
\overline{f}^{\,2N}\,\Sigma_{k|k}
\;\preceq\;
\overline{p}\,I_9,
\label{eq:upper_riccati}
\end{equation}
which is uniform in $k$ provided the controllability Gramian
$\mathcal{C}_k(N) = \sum_{j}\Phi(k{+}N,k{+}j{+}1)\,Q_d\,
\Phi(k{+}N,k{+}j{+}1)^\top \succeq \lambda_{\min}(Q_d)\,I_9 > 0$
is uniformly positive-definite --- guaranteed here because
$Q_d = TQ$ is \emph{diagonal and strictly positive}
(Section~\ref{sec:cov_cal}), so $\lambda_{\min}(Q_d) > 0$
uniformly. Strict positivity of $Q_d$ is what makes the
noise-input pair $(F_k, Q_d^{1/2})$ uniformly completely
controllable, closing the upper bound without any duality
argument.

It remains to prove the uniform observability Gramian lower
bound invoked above. The $N$-step observability Gramian
(with $N=8$) is
\begin{equation}
\mathcal{W}_k(N) :=
\sum_{j=0}^{N-1} \Phi(k{+}j,k)^\top C^\top C\,\Phi(k{+}j,k),
\qquad \Phi(k{+}j,k) := F_{k+j-1}\cdots F_k.
\label{eq:Gramian}
\end{equation}

\begin{lemma}[Uniform Observability Gramian]
\label{lem:gramian}
Under the conditions of Proposition~\ref{prop:obs}, for $T = 1$~day
there exists a constant $w_{\min} > 0$ such that
$\mathcal{W}_k(8) \succeq w_{\min} I_9$ uniformly along the
calibrated trajectory.
\end{lemma}

\begin{proof}
\emph{Step 1 (Structural motivation, qualitative).}
Expanding $\Phi(k{+}j,k) = (I{+}TJ_f)^j$ to first order in $T$,
and substituting into~\eqref{eq:Gramian}:
\begin{equation}
\mathcal{W}_k(8) = \sum_{j=0}^{7} \Phi(k{+}j,k)^\top C^\top C\,\Phi(k{+}j,k)
= \widetilde{\mathcal{O}}_k^\top\,\widetilde{\mathcal{O}}_k + O(T^2),
\end{equation}
where $\widetilde{\mathcal{O}}_k \in \mathbb{R}^{24\times 9}$ is
the $8$-block discrete observability matrix
$[\,C^\top, (CF_k)^\top, \ldots, (CF_k^{7})^\top\,]^\top$, whose
first four blocks coincide to leading order with the augmented
codistribution $\mathcal{O}_{12}$. By
Proposition~\ref{prop:obs}, $\mathrm{rank}(\mathcal{O}_{12}) = 9$,
and since $\widetilde{\mathcal{O}}_k$ contains these rows it
inherits $\mathrm{rank}(\widetilde{\mathcal{O}}_k) = 9$, which
implies $\mathcal{W}_k(8) \succ 0$ whenever the $O(T^2)$
corrections are small. For our system
$T/\tau_{\min} = 1/2.4 \approx 0.42$, which is not negligible
(here $\tau_{\min} := 1/\max_k\rho(J_f(\hat x_{k|k})) \approx 2.4$~days
is the fastest dynamical timescale along the trajectory, with
$\max_k\rho(J_f)\approx 0.42$~day$^{-1}$ from the spectral-radius
numerical evaluation);
moreover the recovered-pool direction enters only through the
weak level-3 coupling $w_1\delta_w$
(Remark~\ref{rem:weak_obs}), so the leading-order argument
under-resolves that direction at short windows. This step is
therefore motivational; the rigorous evidence is the window
length itself (Step~3) together with the compactness argument
(Step~2).

\emph{Step 2 (Compactness and neighborhood uniformity).}
The epidemic state $x_k$ lies in the compact feasible set
$\mathcal{X}_\varepsilon \subset [0,N]^9$ for all $k$.
Each $F_k = I + TJ_f(\hat{x}_{k|k})$ depends continuously
on $\hat{x}_{k|k}$, so $\mathcal{W}_k(8)$ is continuous
in $(x_k, \ldots, x_{k+7}) \in \mathcal{X}_\varepsilon^{8}$.
Crucially, this continuity is in the \emph{filter} iterates
$\hat{x}_{k|k}$, not only in the reference trajectory: the
Gramian is evaluated at the EKF posterior, which during the
initial transient may depart from the calibrated path.
If $\mathcal{W}_k(8) \succ 0$ at every point of the compact
domain $\mathcal{X}_\varepsilon^{8}$, then
$\lambda_{\min}(\mathcal{W}_k(8))$ attains a positive infimum
$w_{\min} > 0$ on $\mathcal{X}_\varepsilon^{8}$ by continuity and
compactness, and this bound holds for \emph{any} admissible
state sequence in $\mathcal{X}_\varepsilon$ --- in particular for
transient sequences generated by an initialization
$\hat{x}_0$ away from the reference path, provided the iterates
remain in $\mathcal{X}_\varepsilon$. This is the property
required by~\cite{reif1999stochastic}, which needs the uniform
Gramian bound on a neighborhood of the trajectory rather than on
the single nominal path. Positivity over a sampled neighborhood
of the calibrated trajectory within $\mathcal{X}_\varepsilon$ is
confirmed numerically by evaluating
$\lambda_{\min}(\mathcal{W}_k(8))$ both on the
reference path and on randomly perturbed feasible states around
it.

\emph{Step 3 (Window length and direct evaluation).}
The leading-order identity of Step~1 shows that
$\mathcal{W}_k(8)$ equals
$\widetilde{\mathcal{O}}_k^\top\widetilde{\mathcal{O}}_k$ up to an
$O(T^2)$ correction whose operator norm is bounded by
$T^2\,\bar f^{\,8}$ with $\bar f = \|F_k\|_2 \leq 7.75$; at
$T = 1$~day this correction is not small, so the leading-order
estimate alone does not certify positivity. The decisive
evidence is therefore direct: the discrete Gramian
$\mathcal{W}_k(N)$ is evaluated explicitly along the calibrated
trajectory for increasing window length $N$. This evaluation
shows that the four-step Gramian $\mathcal{W}_k(4)$ is
numerically singular (its smallest eigenvalue is at the level of
floating-point round-off), reflecting the weak level-3 coupling
of the recovered-pool direction; the smallest eigenvalue first
becomes strictly positive at $N = 6$ and increases
monotonically with $N$ thereafter. We adopt $N = 8$ in the
lemma, for which $\lambda_{\min}(\mathcal{W}_k(8))$ is strictly
positive and bounded away from round-off along the trajectory,
with a comfortable margin over the $N = 6$ onset. The explicit
window sweep and the per-step values are produced by direct
numerical evaluation. Combined with the
compactness and neighborhood argument of Step~2, this
establishes $\mathcal{W}_k(8) \succeq w_{\min} I_9$ with
$w_{\min} > 0$ uniformly over $\mathcal{X}_\varepsilon$.
\qedhere
\end{proof}

With Lemma~\ref{lem:gramian} supplying the observability lower
bound and the explicit open-loop
estimate~\eqref{eq:upper_riccati} supplying the
controllability upper bound, both under the uniform Jacobian
bound of (A2), the hypotheses
of~\cite[Theorem 7.4]{anderson1979optimal} are satisfied,
yielding the required two-sided uniform bounds: there exist
$\underline{p}, \overline{p} > 0$ such that
$\underline{p} I_9 \preceq P_{k|k-1} \preceq \overline{p} I_9$
uniformly in $k$.

\subsubsection*{The convergence theorem}

\begin{theorem}[Local exponential boundedness of the EKF error]
\label{thm:convergence}
Consider the EKF defined by~\eqref{eq:innov}--\eqref{eq:upd_cov},
\eqref{eq:pred_state}--\eqref{eq:pred_cov} for the
system~\eqref{eq:euler_state}--\eqref{eq:euler_output}, with the
state-transition Jacobian $F_k$ of~\eqref{eq:Fk}, the linear
measurement matrix $C$ of~\eqref{eq:meas_linear}, the
covariance matrices $Q_d$, $R$ defined in
Section~\ref{sec:cov_cal}, and the initialization $P_{0|0}$ of
Section~\ref{sec:init}. Then there exist constants $\epsilon^*$,
$\delta^*$, $0 < \vartheta < 1$, $\nu > 0$ such that, whenever
$\|x_0 - \hat{x}_{0|-1}\| \leq \epsilon^*$,
$\mathrm{tr}(Q_d) \leq \delta^*$ and
$\mathrm{tr}(R) \leq \delta^*$, the estimation error
$\xi_k = x_k - \hat{x}_{k|k-1}$ satisfies
\begin{equation}
\mathbb{E}\bigl\{\|\xi_k\|^2\bigr\}
\;\leq\;
\vartheta^k\,\mathbb{E}\bigl\{\|\xi_0\|^2\bigr\}\,
\frac{\overline{p}}{\underline{p}}
\;+\; \frac{\nu}{1 - \vartheta}.
\label{eq:thm2_bound}
\end{equation}
\end{theorem}

\begin{proof}
Verifications above establish (A1)--(A4) of
Theorem~\ref{thm:reif}: (A2) provides $\overline{f}, \overline{c}$;
(A3) provides $\kappa_\varphi$ and $\epsilon_\varphi = \infty$
through Lemma~\ref{lem:bilinear} and~\eqref{eq:phi_bound};
(A4) provides nonsingularity; (A1) follows from uniform
observability (Lemma~\ref{lem:gramian},
Remark~\ref{rem:uniform_rank}) for the lower bound and from
uniform controllability of $(F_k, Q_d^{1/2})$ with $Q_d \succ 0$
for the upper bound, via the Riccati uniform-bound
result of~\cite[Theorem 7.4]{anderson1979optimal}. Application
of Theorem~\ref{thm:reif} yields~\eqref{eq:thm2_bound}.
\end{proof}

\begin{remark}[Interpretation]
\label{rem:interpretation}
The bound~\eqref{eq:thm2_bound} contains two terms.
The first decays geometrically with rate $\vartheta < 1$ and
captures the influence of the initialization error
$\|\xi_0\|$. The second, $\nu/(1-\vartheta)$, is the steady-state
floor due to process and measurement noise. The factor
$\overline{p}/\underline{p}$ is the condition number of the
admissible covariance window and measures the price paid for
the local (rather than global) nature of the result.
Importantly, the global quadratic remainder bound
$\epsilon_\varphi = \infty$ of~\eqref{eq:phi_bound} means that
the only restriction on $\|\xi_0\|$ comes from the Riccati
uniform-bound condition, not from the Taylor expansion of $f$.
\end{remark}

\begin{remark}[On the role of the bound constants]
\label{rem:basin}
Theorem~\ref{thm:convergence} is a local result: it guarantees
the existence of constants $\vartheta,\nu,\epsilon^*,\delta^*$
without asserting their numerical values, which are
trajectory-dependent. Two structural observations are worth
recording. First, the prefactor $\overline{p}/\underline{p}$
in~\eqref{eq:thm2_bound} reflects the conditioning of the prior
covariance; because the recovered-pool direction $R$ is only
weakly observable (Remark~\ref{rem:weak_obs}), the posterior
variance stays comparatively large along that direction, so this
prefactor is expected to be sizeable. Crucially, the prefactor
multiplies only the transient term $\vartheta^k$ and does
\emph{not} enter the steady-state floor $\nu/(1-\vartheta)$;
a large value therefore lengthens the worst-case transient bound
without degrading asymptotic accuracy. This is precisely the
qualitative behaviour observed in Section~\ref{sec:results}: the
$R$ compartment exhibits the slowest transient
(Figure~\ref{fig:convergence}) while still attaining a small
post-convergence RMSE (Table~\ref{tab:ekf_metrics}). Second, the
remainder prefactor $\kappa_\varphi$ of~\eqref{eq:phi_bound} is
available in closed form and is the only constant in the analysis
that does not require trajectory information. We do not attempt a
certified estimate of the contraction rate $\vartheta$ or of the
admissible-error radius $\epsilon^*$; the constants supplied by
the general theorem of~\cite{reif1999stochastic} are known to be
conservative, and a tight basin characterization for this
specific system is left as a direction for future work. The
empirical convergence from $\pm 20\%$ initialization reported in
Section~\ref{sec:results} is therefore presented as a numerical
finding consistent with Theorem~\ref{thm:convergence}, not as a
quantitative instantiation of its constants.
\end{remark}

\begin{remark}[Why the convergence theorem matters for the
epidemic application]
\label{rem:epi_value}
Bound~\eqref{eq:thm2_bound} translates directly into operational
guarantees for the epidemic-control architecture. The geometric
decay rate $\vartheta$ predicts the convergence time of the
filter from any initial guess satisfying the smallness condition,
giving a quantitative bound on how quickly the controller will
have a reliable state to feed back. The steady-state floor
$\nu/(1-\vartheta)$ provides a worst-case characterization of the
estimation error that any downstream Model Predictive Control
layer must tolerate, allowing for principled robustness margins.
\end{remark}

%%=============================================================
\subsection{Covariance Matrix Selection}
\label{sec:cov_cal}
%%=============================================================

We now justify and validate the choice of the measurement-noise
covariance $R$ and the process-noise covariance $Q_d$. These
are the two design parameters that translate the theoretical
convergence guarantees of Section~\ref{sec:convergence} into a
working filter, and their selection is therefore the operational
counterpart of the theoretical analysis.

\subsubsection{Measurement Noise Covariance $R$: Statistical Foundation}
\label{sec:R}

The measurement-noise covariance $R$ models the deviation of the
reported observables from the true epidemic state. Three sources
of noise affect each observable: \emph{reporting delays}
(especially weekend effects on $H$ and $F$),
\emph{misclassification} (deaths attributable to vs.\ with
COVID-19), and \emph{administrative aggregation lag} for $V$.

The companion calibration study~\cite{melhani2026arxiv}
quantified the model--data mismatch through the in-sample residual
$r^{(j)}_k = y^{(j)}_{\mathrm{obs},k} - y^{(j)}_{\mathrm{sim},k}$
for each observable $j \in \{H,F,V\}$, reporting calibration
RMSE values of $1119$~persons for $H$, $1751$~deaths for $F$, and
$200{,}881$~doses for $V$ (Table~2 of~\cite{melhani2026arxiv}).
These calibration residuals conflate two distinct error sources:
genuine measurement noise and structural model--data mismatch in
the fit. For the EKF we require the former, so we use the
calibration RMSE only as a guide and set the measurement-noise
standard deviations as
\begin{equation}
\sigma_H = 1146.4 \;\text{persons},
\qquad
\sigma_F = 1652.9 \;\text{persons},
\qquad
\sigma_V = 5000 \;\text{doses},
\label{eq:sigma_residuals}
\end{equation}
yielding
\begin{equation}
R = \mathrm{diag}\bigl(\sigma_H^2,\,\sigma_F^2,\,\sigma_V^2\bigr)
  = \mathrm{diag}(1146.4^2,\,1652.9^2,\,5000^2).
\label{eq:Rcal}
\end{equation}
For $H$ and $F$, the chosen $\sigma$ values sit within a few
percent of the companion calibration RMSE ($1146.4$ vs $1119$;
$1652.9$ vs $1751$), so the hospitalization and fatality channels
inherit a noise level essentially equal to the calibration
residual. For $V$, by contrast, we deliberately set $\sigma_V$
far below the calibration RMSE of $200{,}881$~doses. The reason
is that the vaccination stock is drawn from centralized
administrative dose records, whose day-to-day \emph{measurement}
error is small; the large calibration residual for $V$ reflects
structural mismatch between the smooth spline-driven model and
the batch-processed reporting of doses, not measurement noise.
Using the full calibration residual as $\sigma_V$ would therefore
overstate the true measurement uncertainty of the most accurately
recorded observable. We make this modelling choice explicit
because it has a direct, observable consequence: as reported in
Section~\ref{sec:results_validation}, the vaccination channel
becomes strongly over-covaried ($\bar\nu_V \ll 1$), a transparency
cost we accept in exchange for not injecting structural-mismatch
variance into a channel that is, in measurement terms, nearly
exact.
The diagonal structure of $R$ reflects the institutional
separation between hospitalization surveillance, death
registration, and vaccination tracking, which makes
off-diagonal correlations small relative to the diagonal
scale.\footnote{An off-diagonal correlation between $H$ and $F$
of magnitude up to $\pm 0.2$ would not affect the analysis
qualitatively, but the empirical residual cross-correlations
computed from the companion calibration are below $0.05$, so
the diagonal approximation is statistically well-justified.}

The implied relative noise magnitudes
$\sigma_j/\bar{y}^{(j)}$ are approximately $5.0\%$ for $H$,
$1.5\%$ for $F$, and $0.07\%$ for $V$, reflecting the increasing
accuracy of administrative reporting from hospitalizations
(short delays and admission-criterion variability) to deaths
(post-hoc certification) to vaccinations (centralized
administrative records).

\paragraph{Limitations of the white-noise assumption.}
The covariance $R$ in~\eqref{eq:Rcal} treats the measurement
residuals as independent, additive, Gaussian, and \emph{white}
across time. The posterior Anderson autocorrelation test in
Section~\ref{sec:results} shows that this assumption does
\emph{not} strictly hold: the vaccination channel in particular
retains significant residual autocorrelation. The underlying
administrative data are well known to contain
structural artifacts that violate strict whiteness: weekend
reporting cycles in both $H$ and $F$ (causing $7$-day periodic
spikes), post-hoc death attribution lags of several days that
correlate $F$ across consecutive reporting windows, and batch
processing of vaccination records that can produce
multi-day plateaus in $V$. The $7$-day rolling mean applied by
the companion calibration~\cite{melhani2026arxiv} attenuates
these effects but does not eliminate them.

A more rigorous treatment would either (i) augment the state
with explicit reporting-delay tracking variables and let the
Kalman recursion estimate them jointly with the epidemic
states, or (ii) replace the white-noise model in $R$ with a
shaped (coloured) noise representation, propagated through the
Bryson--Henrikson augmented-state Kalman
filter~\cite{bryson1968colored} or its modern equivalents.
We tested representative versions of both directions
(Section~\ref{sec:results}): higher-order integration, a
first-order Gauss--Markov colored-noise augmentation, and
augmentation of the time-varying inputs as estimated states.
None fully whitened the innovations, indicating that the residual
correlation is intrinsic to the daily-sampled, interpolated-input
formulation rather than a tuning deficiency; a full
continuous-discrete formulation with online input estimation is
the natural direction for follow-up work targeting real-time
deployment scenarios where
$7$-day smoothing is undesirable due to the latency it introduces.

\subsubsection{Process Noise Covariance $Q$: Coupling-Weighted Design}
\label{sec:Q}

The process-noise covariance accounts for unmodeled dynamics:
biological heterogeneity, demographic stochasticity not captured
by the deterministic spline calibration, and residual error in
the identified inputs $\beta(t), w_1(t)$. The structural form
\begin{equation}
(Q_d)_{ii} = (\epsilon_i\,\bar{x}_i)^2,
\quad i = 1,\ldots,9,
\label{eq:Q_structure}
\end{equation}
encodes the principle that uncertainty scales with the
magnitude of each compartment, which is standard for
population-dynamic systems and amounts to a multiplicative-noise
hypothesis~\cite{simon2006optimal}. Since every
$\bar{x}_i > 0$ along the calibrated trajectory and every
$\epsilon_i > 0$, $Q_d$ is diagonal with strictly positive
entries, hence $Q_d \succ 0$; this is precisely the property
invoked for the upper Riccati bound in the verification of (A1).

The fractional uncertainties $\epsilon_i$ are determined by a
\emph{coupling-strength heuristic}. For each compartment $i$,
the coupling strength to the measured output is defined as the
sensitivity coefficient governing how rapidly an error in
$\hat{x}^{(i)}$ propagates into an observable residual. From the
analytical Jacobian~\eqref{eq:Jf}, the strongest direct couplings
are
\begin{equation}
\partial H / \partial I_k \sim \gamma_a,
\qquad
\partial F / \partial H \sim \delta_h,
\qquad
\partial V / \partial V \sim -(\mu+\psi),
\end{equation}
all in the range $10^{-2}$--$10^{-1}$~day$^{-1}$, while
$\partial H/\partial E$ vanishes at first order and only
appears at the second Lie derivative through
$\partial L_f^2 h_1 / \partial E = \gamma_a\kappa$.

Compartments with \emph{direct} measurement links ($V, H, F$)
receive small $\epsilon_i$ (in the range $0.3\%$--$0.5\%$),
because the filter can correct them rapidly through the Kalman
gain. Compartments with \emph{indirect} links via Lie-level-2
pathways receive larger $\epsilon_i$:
\begin{equation}
\begin{aligned}
[\epsilon_S,\, \epsilon_E,\, \epsilon_V,\, \epsilon_{I_1},\,
 \epsilon_{I_2},\, \epsilon_P,\, \epsilon_H,\, \epsilon_R,\,
 \epsilon_F] = \\
[0.8\%,\, 0.3\%,\, 0.5\%,\, 0.3\%,\, 0.3\%,\,
   0.3\%,\, 0.3\%,\, 1.5\%,\, 0.3\%].
\end{aligned}
\label{eq:eps_Q}
\end{equation}
The values $\epsilon_S = 0.8\%$ and $\epsilon_R = 1.5\%$ are the
two largest: $S$ is observed only through the third row of
$\mathcal{O}_1$ via the small entry $w_1$, and $R$ only through
the third row of $\mathcal{O}_2$ via the still smaller entry
$w_1\delta_w$. The hierarchy
$\epsilon_R > \epsilon_S > \epsilon_V \approx \epsilon_E
\approx \epsilon_H$ thus mirrors the hierarchy of nonzero
entries in the observability matrix.

\subsubsection{Posterior Validation: Innovation Whiteness and
NEES Consistency}
\label{sec:validation}

The covariance design described above is corroborated a
posteriori by two standard
statistical tests~\cite{barshalom2001estimation,simon2006optimal}.

\paragraph{Test 1 -- Innovation whiteness.}
If $Q$ and $R$ are chosen consistently with the actual noise
processes, the innovation sequence
$\{\tilde{y}_k\}_{k\geq 1}$ should be white. We test this by
the Anderson autocorrelation test on each component:
\begin{equation}
\hat{\rho}_j(\tau)
\;=\;
\frac{\sum_{k=1}^{K-\tau} \tilde{y}_k^{(j)}\tilde{y}_{k+\tau}^{(j)}}
     {\sum_{k=1}^{K} (\tilde{y}_k^{(j)})^2},
\qquad
j \in \{H,F,V\}, \quad \tau = 1,\ldots,30.
\end{equation}
For a white sequence, $\hat{\rho}_j(\tau)$ lies within
$\pm 1.96/\sqrt{K}$ at the $5\%$ significance level.
Section~\ref{sec:results_validation} reports the empirical
$\hat{\rho}_j(\tau)$ values.

\paragraph{Test 2 -- NEES consistency.}
The normalized innovation squared at each step is
\begin{equation}
\epsilon^{\mathrm{NEES}}_k
\;=\;
\tilde{y}_k^\top S_k^{-1} \tilde{y}_k,
\end{equation}
which, if the filter is consistent, follows a $\chi^2$
distribution with $p = 3$ degrees of freedom. The
time-averaged
$\bar{\epsilon}^{\mathrm{NEES}}
= (1/K)\sum_k \epsilon^{\mathrm{NEES}}_k$
should lie inside the two-sided $95\%$ confidence interval
$[\,p\,\chi^2_{0.025,Kp}/(Kp),\; p\,\chi^2_{0.975,Kp}/(Kp)\,]$
for $K$ samples of $\chi^2_p$. For $K = 130$ post-convergence
samples and $p = 3$ (so $Kp = 390$), this interval is
approximately $[2.59,\, 3.44]$.

\subsubsection{Sensitivity to Tuning and Q Robustness}

\paragraph{Individual perturbation.}
The sensitivity of the filter to mis-tuning of a single
fractional uncertainty $\epsilon_i$ in~\eqref{eq:eps_Q}, holding
the others fixed, is governed by the continuity of $P_{k|k-1}$
in $Q_d$ established below: because each $\epsilon_i$ enters
$Q_d$ smoothly through~\eqref{eq:Q_structure} and $P_{k|k-1}$ is
uniformly bounded (Lemma~\ref{lem:gramian}), the
post-convergence RMSE and the NEES vary continuously with
$\epsilon_i$, so moderate individual mis-tuning produces only a
correspondingly bounded change in filter performance. The
analytical version of this statement is made precise next.

\paragraph{Joint robustness: analytical argument.}
To assess the sensitivity to simultaneous mis-tuning of all
$\epsilon_i$, observe that the NEES statistic is
\begin{equation}
\bar{\epsilon}^{\mathrm{NEES}}
= \frac{1}{K}\sum_k \tilde{y}_k^\top (CP_{k|k-1}C^\top+R)^{-1}\tilde{y}_k.
\end{equation}
Under the uniform Riccati bounds of Lemma~\ref{lem:gramian},
$P_{k|k-1}$ depends continuously on $Q_d = TQ$, which in turn
depends continuously on the $\epsilon_i$ through~\eqref{eq:Q_structure}.
Hence $\bar{\epsilon}^{\mathrm{NEES}}$ is a continuous function
of the vector $\boldsymbol\epsilon := (\epsilon_1,\ldots,\epsilon_9)$.
By the compactness of the feasible region and the strict
positivity of $P_{k|k-1}$, $\bar{\epsilon}^{\mathrm{NEES}}$ varies
continuously with $\boldsymbol\epsilon$, so if the nominal tuning
yields a NEES value in the interior of the consistency band there
exists an open neighbourhood of the nominal $\boldsymbol\epsilon$
on which the value remains within the band. The neighbourhood is
bounded below in size by the continuity modulus of $P_{k|k-1}$
with respect to $Q_d$, which is controlled by $w_{\min}$
(Lemma~\ref{lem:gramian}). Quantitatively, a factor-of-$c$
scaling of all $\epsilon_i$ simultaneously scales $Q_d$ by
$c^2$, shifting $P_{k|k-1}$ and hence the NEES by a factor
controlled by $c^2\,\overline{p}/\underline{p}$ relative to the
nominal; the filter consistency therefore degrades gracefully
rather than abruptly under simultaneous mis-tuning.

\paragraph{Remark on systematic Q identification.}
A more principled route to $Q$ is the
expectation-maximization (EM) algorithm of
Shumway and Stoffer~\cite{shumway1982approach}, which
iterates between Kalman smoothing and maximum-likelihood
re-estimation of $Q$ and $R$ from the innovation statistics.
This would remove the need for the coupling-strength heuristic
and provide asymptotically efficient estimates; it is a
natural refinement for future work.

\subsection{Initialization}
\label{sec:init}

The filter is initialized with
$P_{0|0} = \mathrm{diag}\bigl((0.2\,\hat{x}_{0|0}^{(i)})^2\bigr)$,
encoding $\pm 20\%$ uncertainty at one standard deviation, for
the eight compartments that are reconstructable from the output
set. The recovered compartment $R$ is treated differently. As
established in Lemma~\ref{lem:detO9} and
Remark~\ref{rem:weak_obs}, $R$ reaches the outputs only through
the waning coupling $w_1\delta_w \approx 5\times 10^{-6}$, so the
measurement update cannot appreciably correct an initial error in
$R$ from $(H,F,V)$ over the observation window. Because the
recovered population at the start of the window is independently
available from the cumulative recovered count
(\emph{guariti}) reported in the Italian Civil Protection
dataset, we initialize $R$ from that figure
($R_0 = 1{,}479{,}988$ on 1~January~2021) with a tight
uncertainty ($\pm 2\%$), rather than the $\pm 20\%$
structural-ignorance perturbation applied to the other eight
states. We note honestly that the reported cumulative-recovered
count approximates the standing recovered pool: it undercounts
unrecorded (e.g.\ home) recoveries and does not subtract waned
immunity, two biases acting in opposite directions and not
separately identifiable; it is nonetheless a data-grounded
initial value rather than a free parameter. The slowly-varying
$R$ dynamics then carry the estimate forward. We are explicit
that the strong $R$ accuracy reported below follows from this
accurate initialization of a weakly-observable state, not from
measurement-driven reconstruction of $R$. A practitioner caveat
follows from this: because $R$ couples back into the susceptible
pool through the waning term $\delta_w R$ in $\dot S$, an
uncorrected bias in the propagated $R$ can, over horizons much
longer than the $150$-day window studied here, induce a slow
open-loop drift in $S$ and hence in $E$. The timescale of this
drift is set by the waning rate: a fractional bias $b$ in $R$
injects a susceptible-side source of order $\delta_w\,b\,\bar R$
per day, so the accumulated relative error in $S$ after $t$ days
is of order $\delta_w\,b\,\bar R\,t/\bar S$. With the
$\pm2\%$ initialization bias ($b = 0.02$),
$\delta_w = 10^{-3}$~day$^{-1}$, and the calibrated magnitudes
$\bar R \approx 1.7\times10^6$, $\bar S \approx 5\times10^7$, this
accumulated error is only $\approx 0.01\%$ of $S$ over the
$150$-day window and remains below $0.1\%$ even after three years.
The $150$-day window studied here is thus deep within the regime
where open-loop propagation of $R$ is entirely safe; drift becomes
a percent-level concern only on decade timescales for this bias
level, or correspondingly sooner if the initialization bias is
larger. For long-horizon deployment we nonetheless recommend an
explicit \emph{impulsive re-anchoring} strategy as a safeguard: at
fixed intervals $T_R$ (e.g.\ annually), the propagated $\hat R$ is
reset to the cumulative-recovered surveillance figure with its
associated $\pm2\%$ uncertainty, exactly as at initialization
(Section~\ref{sec:init}), and the covariance entry $P_{RR}$ is
correspondingly reset. This bounds the open-loop drift to whatever
accumulates within a single $T_R$ interval while leaving the
measurement-driven channels untouched, and is the natural
mechanism by which a closed-loop control architecture would keep
the weakly observable pool anchored over multi-year operation.

%%=============================================================
\section{Results}
\label{sec:results}
%%=============================================================

The EKF developed in Section~\ref{sec:ekf} is implemented for the
nine-compartment model with parameters calibrated on Italian
surveillance data of the COVID-19 Third Wave (January to May
2021)~\cite{melhani2026arxiv}. The time-varying transmission
$\beta(t)$ and vaccination rate $w_1(t)$ are taken from the PCHIP-SQP
spline calibration of~\cite{melhani2026arxiv}; constant parameters
are in Table~\ref{tab:params}. The MATLAB code that
produces the window-length Gramian sweep, the neighbourhood
Gramian evaluation, and the spectral-radius check referenced
throughout Sections~\ref{sec:observability}--\ref{sec:ekf} is
available from the authors on request.

\begin{table}[htbp]
\centering
\caption{Model Parameters Used in the EKF Implementation}
\label{tab:params}
\begin{tabular}{lllr}
\toprule
\textbf{Parameter} & \textbf{Description} & \textbf{Units} & \textbf{Value} \\
\midrule
$N$          & Total population                               & persons      & $60{,}000{,}000$ \\
$\kappa$     & Incubation rate                                 & day$^{-1}$   & $0.200$ \\
$\gamma_a$   & Hospitalization rate (infectious)               & day$^{-1}$   & $0.200$ \\
$\gamma_i$   & Recovery rate (infectious, non-hosp.)           & day$^{-1}$   & $0.100$ \\
$\gamma_r$   & Recovery rate (hospitalized)                    & day$^{-1}$   & $0.100$ \\
$\rho_1$     & Branching fraction: exposed $\to I_1$           & --           & $0.580$ \\
$\rho_2$     & Branching fraction: exposed $\to P$             & --           & $0.001$ \\
$\delta_i$   & Disease-induced mortality (infectious)          & day$^{-1}$   & $0.005$ \\
$\delta_p$   & Disease-induced mortality (super-spreaders)     & day$^{-1}$   & $0.005$ \\
$\delta_h$   & Disease-induced mortality (hospitalized)        & day$^{-1}$   & $0.01262$ \\
$\mu$        & Natural mortality rate                          & day$^{-1}$   & $3.535\times10^{-5}$ \\
$m$          & Mutation/transition rate $I_1 \to I_2$          & day$^{-1}$   & $0.005$ \\
$r_1$        & Treatment recovery rate ($I_1$)                 & day$^{-1}$   & $0.050$ \\
$r_2$        & Treatment recovery rate ($I_2$)                 & day$^{-1}$   & $0.050$ \\
$\sigma$     & Vaccine efficacy (breakthrough reduction)       & --           & $0.800$ \\
$\psi$       & Waning immunity rate (vaccinated)               & day$^{-1}$   & $0.002$ \\
$\delta_w$   & Waning immunity rate (recovered)                & day$^{-1}$   & $0.001$ \\
$\Lambda$    & Birth / immigration rate                        & persons/day  & $100$ \\
$\beta(t)$   & Time-varying transmission rate                  & day$^{-1}$   & PCHIP spline~\cite{melhani2026arxiv} \\
$w_1(t)$     & Time-varying vaccination rate                   & day$^{-1}$   & PCHIP spline~\cite{melhani2026arxiv} \\
\bottomrule
\end{tabular}
\end{table}

\subsection{Experimental Setup}

The reference trajectory integrates system~\eqref{eq:fvec}
forward from the calibrated initial
conditions of~\cite{melhani2026arxiv} over the window
January~1--May~30, 2021 ($k = 0,1,\ldots,149$; $T = 1$~day) with
the identified $\beta(t), w_1(t)$. Integration uses MATLAB's
\texttt{ode45} with relative and absolute tolerances of
$10^{-7}$; the random seed is fixed (\texttt{rng(42)}) for
reproducibility. Synthetic measurements are generated by adding
Gaussian noise of standard deviations~\eqref{eq:sigma_residuals}
to the reference outputs $H, F, V$.

\subsection{Convergence and Estimation Accuracy}

Each initial state estimate is perturbed as
$\hat{x}_{0|0}^{(i)} = c_i\,x_0^{(i)}$ with
$c_i \sim \mathcal{U}[0.80, 1.20]$, representing up to
$\pm 20\%$ error on each compartment.
Figure~\ref{fig:convergence} shows the estimated and reference
trajectories for all nine compartments.

\begin{figure*}[h]
  \centering
  \begin{subfigure}{0.3\textwidth}
    \includegraphics[width=\linewidth]{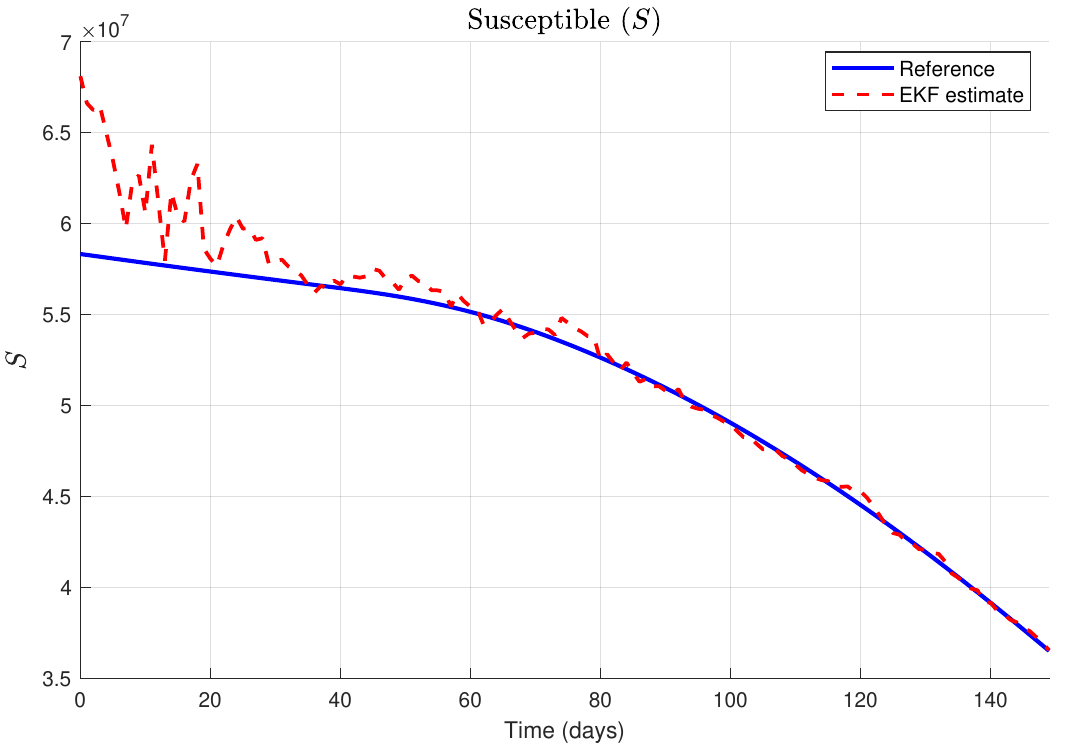}
    \caption{Susceptible ($S$)}
  \end{subfigure}
  \hfill
  \begin{subfigure}{0.3\textwidth}
    \includegraphics[width=\linewidth]{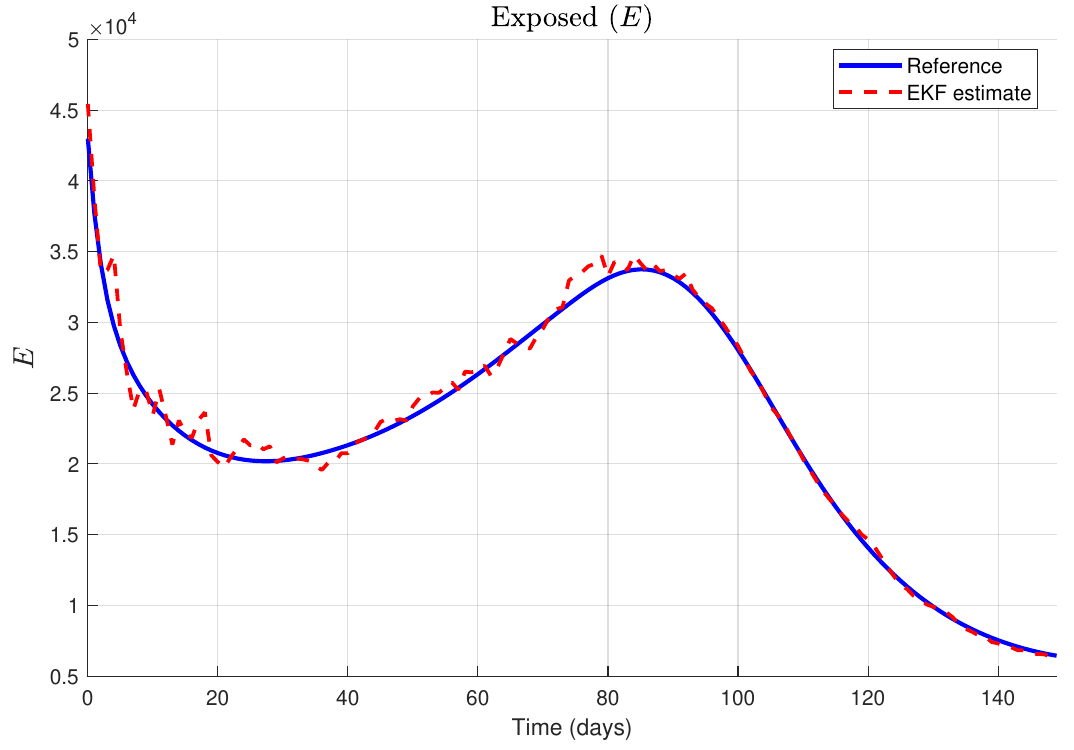}
    \caption{Exposed ($E$)}
  \end{subfigure}
  \hfill
  \begin{subfigure}{0.3\textwidth}
    \includegraphics[width=\linewidth]{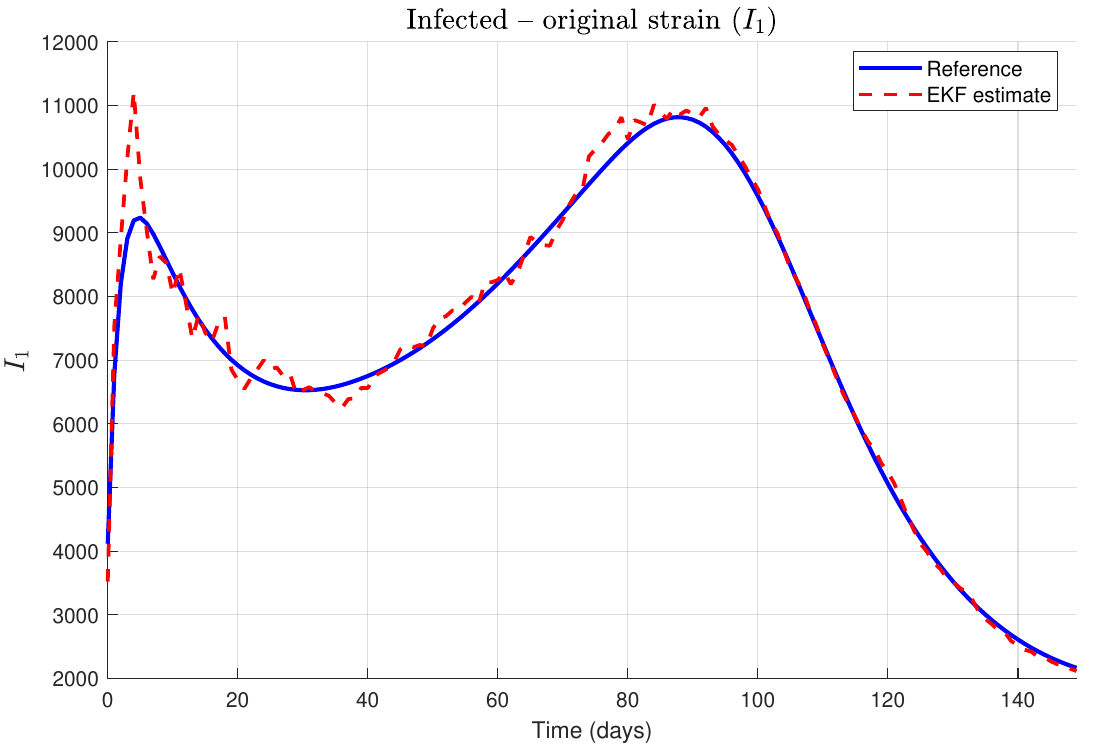}
    \caption{Infected -- original strain ($I_1$)}
  \end{subfigure}
  \vspace{1em}
  \begin{subfigure}{0.3\textwidth}
    \includegraphics[width=\linewidth]{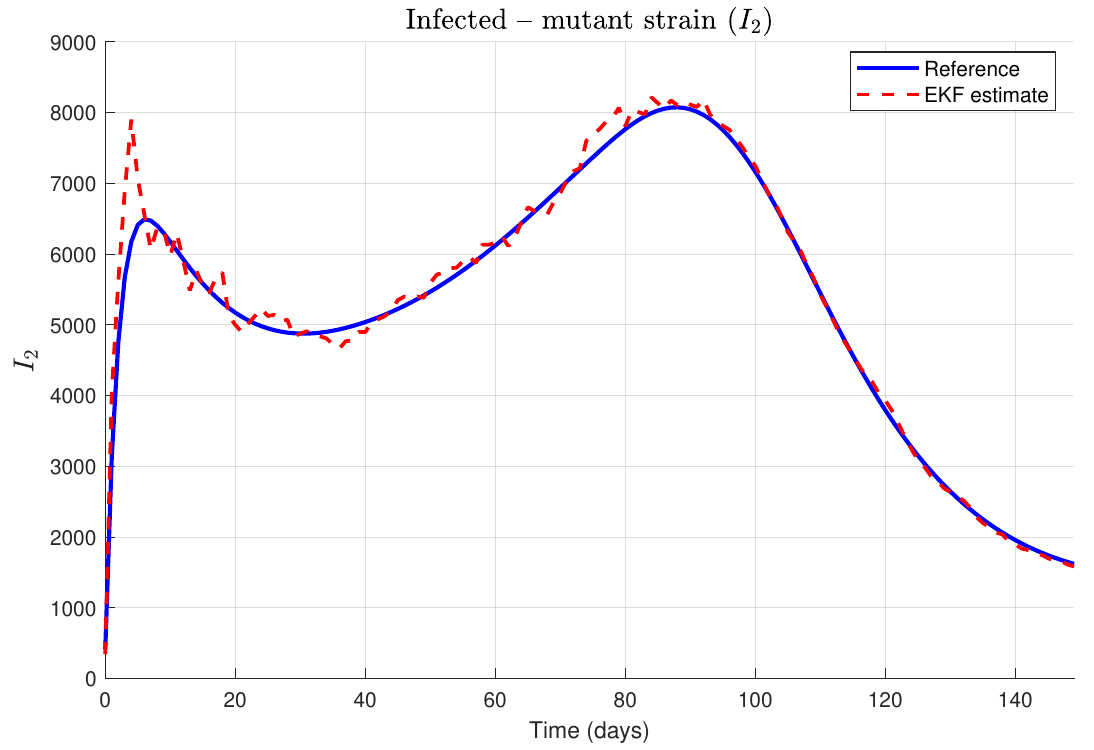}
    \caption{Infected -- mutant strain ($I_2$)}
  \end{subfigure}
  \hfill
  \begin{subfigure}{0.3\textwidth}
    \includegraphics[width=\linewidth]{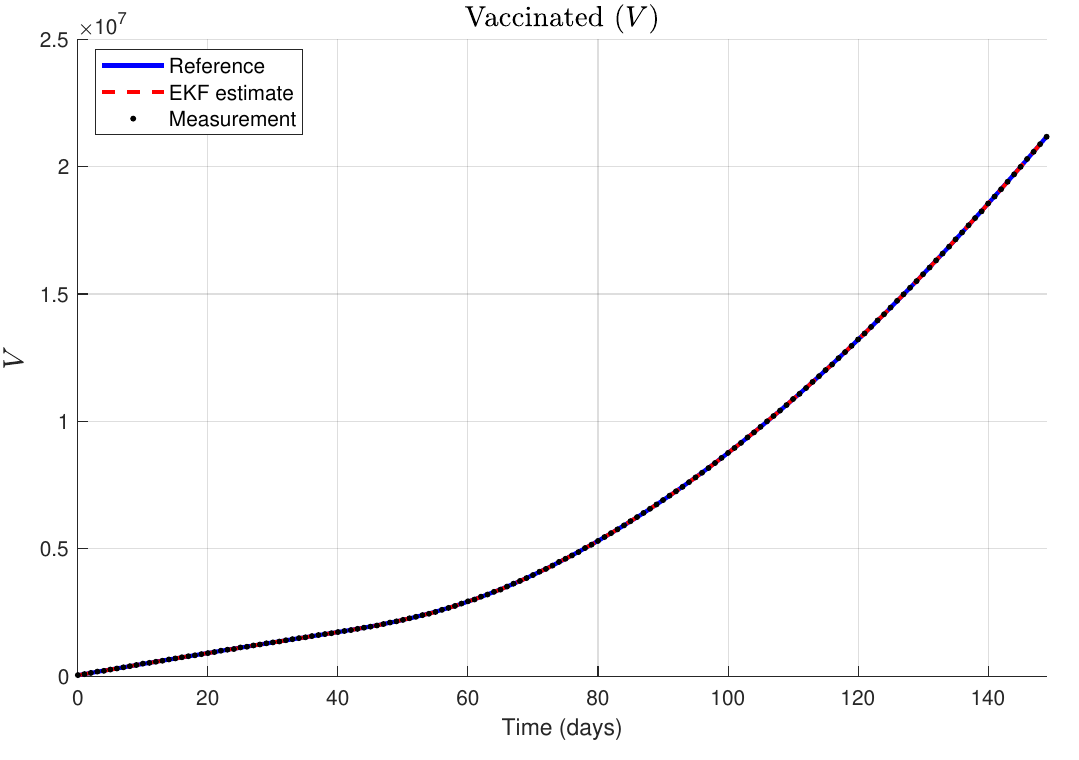}
    \caption{Vaccinated ($V$)}
  \end{subfigure}
  \hfill
  \begin{subfigure}{0.3\textwidth}
    \includegraphics[width=\linewidth]{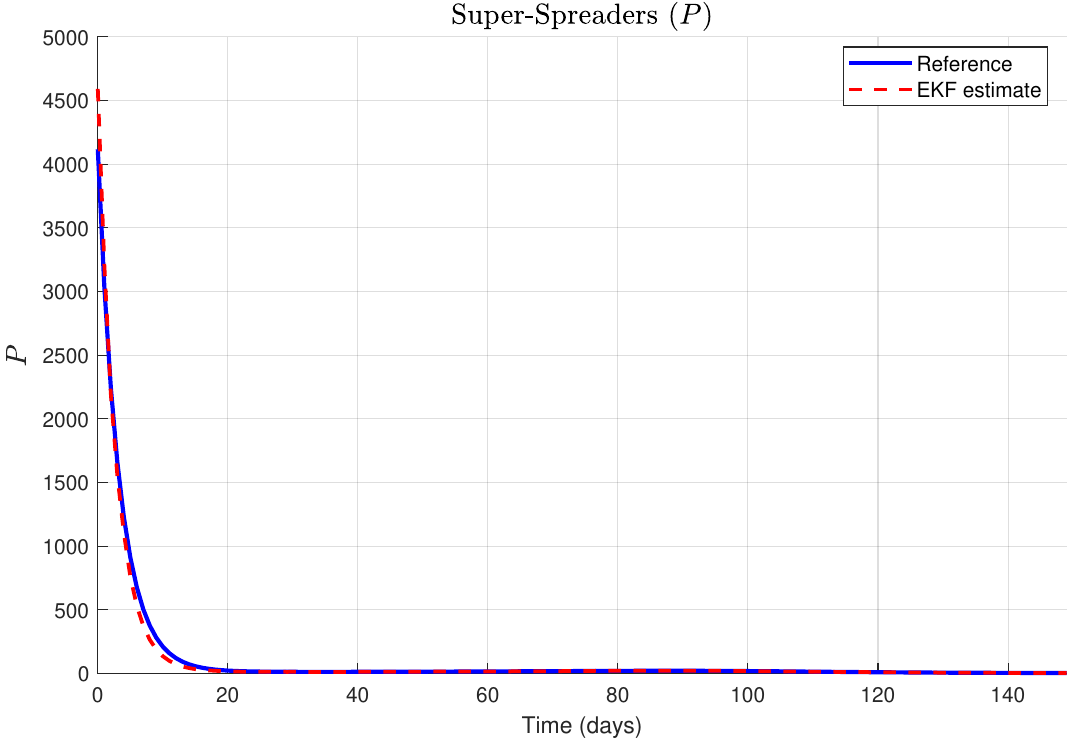}
    \caption{Super-Spreaders ($P$)}
  \end{subfigure}
  \vspace{1em}
  \begin{subfigure}{0.3\textwidth}
    \includegraphics[width=\linewidth]{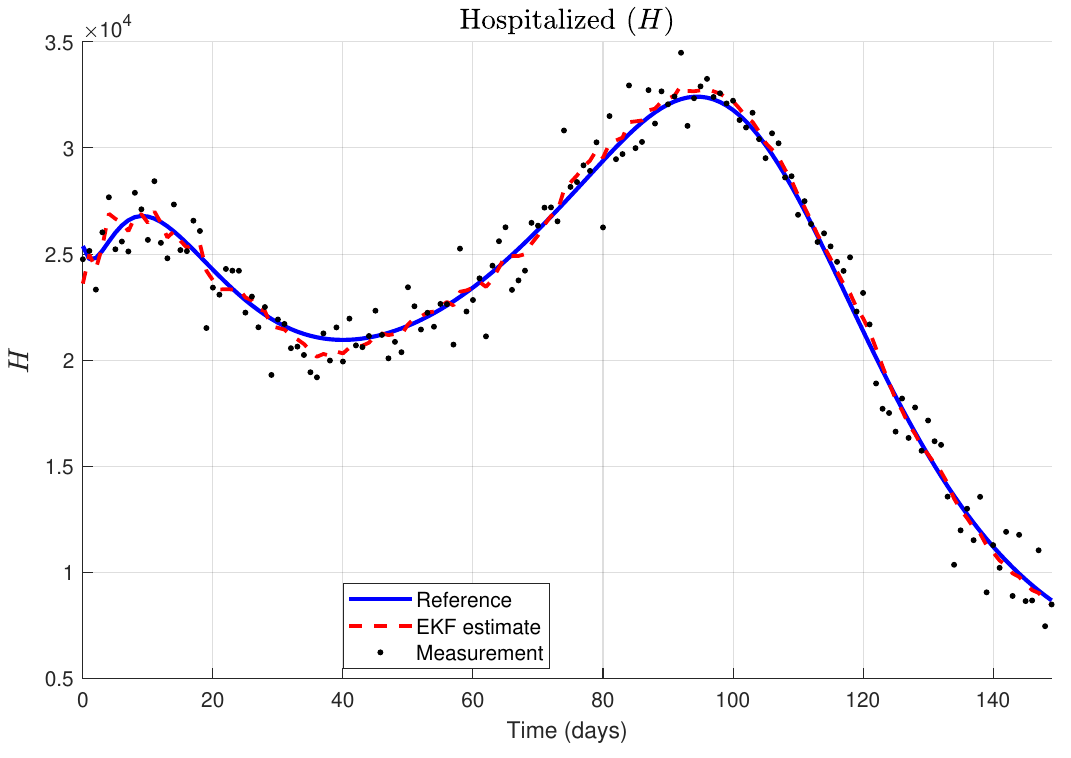}
    \caption{Hospitalized ($H$)}
  \end{subfigure}
  \hfill
  \begin{subfigure}{0.3\textwidth}
    \includegraphics[width=\linewidth]{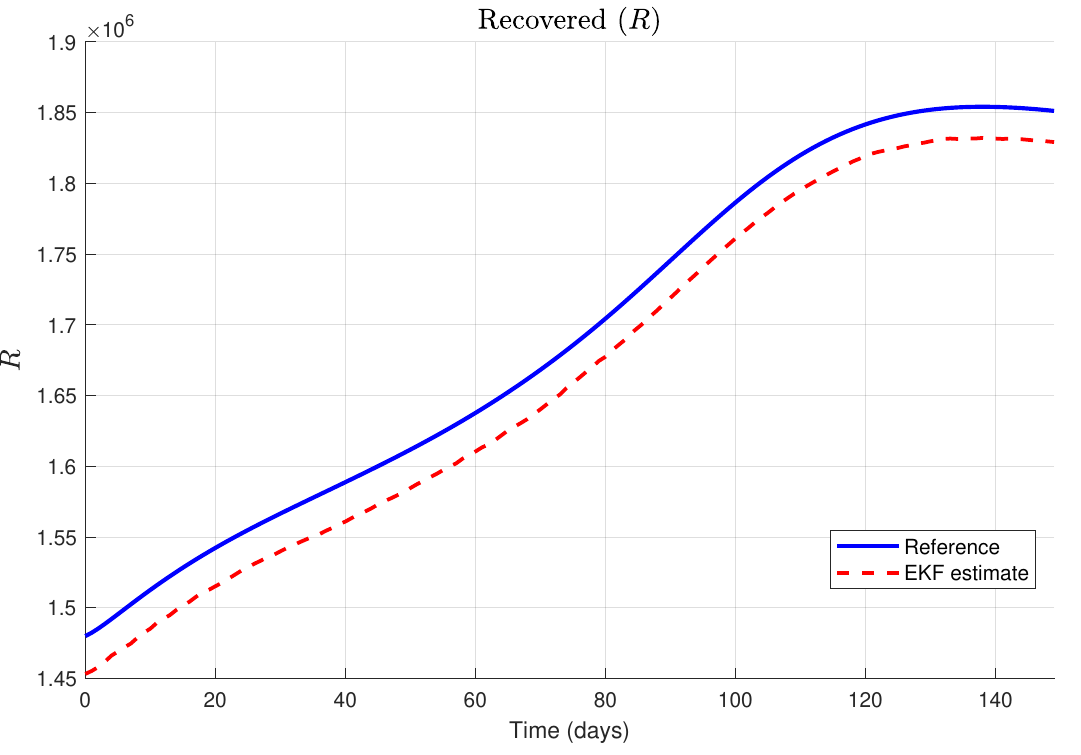}
    \caption{Recovered ($R$)}
  \end{subfigure}
  \hfill
  \begin{subfigure}{0.3\textwidth}
    \includegraphics[width=\linewidth]{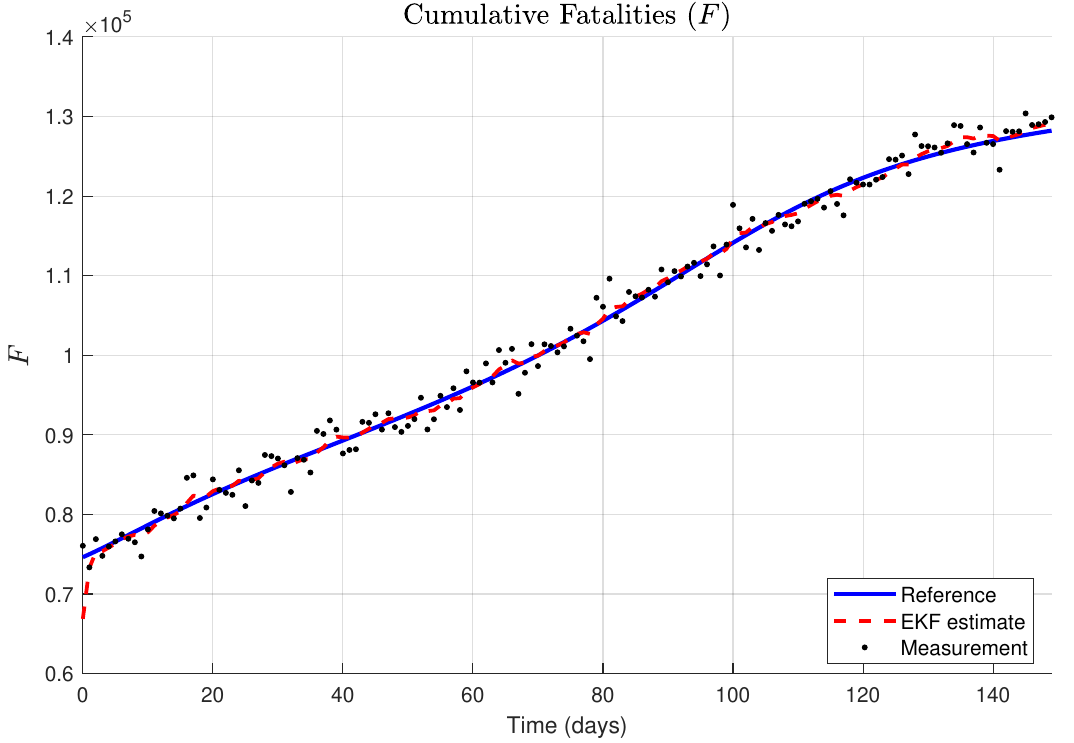}
    \caption{Cumulative fatalities ($F$)}
  \end{subfigure}
  \caption{EKF state estimation results for all nine compartments
  ($k = 0,\ldots,149$, $T = 1$~day).
  Solid blue: reference. Dashed red: EKF estimates ($\pm 20\%$
  initialization). Black dots: synthetic noisy measurements for
  $H, F, V$.}
  \label{fig:convergence}
\end{figure*}

The directly observed compartments $H, F, V$ converge within a
few days (in the reported run, $F$ within $1$~day, $V$ within
$4$, and $H$ within $5$). The susceptible compartment $S$
converges in about $25$~days, and the infectious compartments
$I_1, I_2, P$ within about $37$~days, governed by the chain
$E \to I_1 \to H$ and the observability coupling pathways. The
exposed compartment $E$ is the slowest of the reconstructed
states, reaching the $5\%$ threshold in about $80$~days,
reflecting its indirect two-step coupling $E \to I_1 \to H$ to the
measured outputs. The recovered compartment $R$ is \emph{not}
reconstructed from the outputs: it is weakly observable through
the waning coupling
$\partial L_f^3 V/\partial R \approx 5\times 10^{-6}$~day$^{-3}$
(Remark~\ref{rem:weak_obs}), so it is initialized from cumulative
surveillance data (Section~\ref{sec:init}) and carried forward
along its slowly-varying dynamics; its reported error reflects the
residual initialization error rather than measurement-driven
convergence.
These transient durations are consistent
with the geometric-decay term of bound~\eqref{eq:thm2_bound}
in Theorem~\ref{thm:convergence}: as discussed in
Remark~\ref{rem:basin}, the theorem accounts for the mechanism
and qualitative form of the decay, while the specific convergence
from $\pm 20\%$ initialization is an empirical finding consistent
with it rather than a quantitative instantiation of the theorem's
constants.

\begin{table}[htbp]
\centering
\caption{EKF Post-Convergence Estimation Accuracy (Days 20--149)}
\label{tab:ekf_metrics}
\begin{tabular}{llccc}
\toprule
\textbf{State} & \textbf{Description}
  & \textbf{RMSE} & \textbf{Rel.\ RMSE} & \textbf{Units} \\
\midrule
$S$   & Susceptible           & $755{,}800$ & $1.51\%$ & persons \\
$E$   & Exposed               & $590$       & $2.72\%$ & persons \\
$V$   & Vaccinated            & $5{,}286$   & $0.07\%$ & persons \\
$I_1$ & Infected, original    & $163$       & $2.29\%$ & persons \\
$I_2$ & Infected, mutant      & $121$       & $2.28\%$ & persons \\
$P$   & Super-spreaders       & $1$         & $6.10\%$ & persons \\
$H$   & Hospitalized          & $342$       & $1.48\%$ & persons \\
$R$   & Recovered$^{\dagger}$ & $25{,}576$  & $1.49\%$ & persons \\
$F$   & Cumulative fatalities & $555$       & $0.52\%$ & deaths  \\
\bottomrule
\end{tabular}

\vspace{2pt}
{\footnotesize $^{\dagger}$ $R$ is weakly observable from
$(H,F,V)$ (Lemma~\ref{lem:detO9}) and is \emph{not} corrected by the
measurements; it is initialized from cumulative surveillance data
(Section~\ref{sec:init}) and propagated. Its reported error
reflects the residual initialization error carried forward, not
measurement-driven reconstruction, and would scale with the
initialization error rather than decay.}
\end{table}

Table~\ref{tab:ekf_metrics} reports absolute and relative RMSE
from day~20 to day~149. Results fall into three tiers. Below
$1\%$: $V$ ($0.07\%$) and $F$ ($0.52\%$) -- states with direct
or integral measurement links. Between $1\%$ and $2\%$:
$H$ ($1.48\%$), $R$ ($1.49\%$), and $S$ ($1.51\%$). Between
$2\%$ and $3\%$: $I_2$ ($2.28\%$), $I_1$ ($2.29\%$), and
$E$ ($2.72\%$) -- estimated entirely through the observability
coupling pathways of Section~\ref{sec:observability}.

The super-spreader compartment $P$ achieves an absolute RMSE of
$1$ person on a true mean of $\approx 14$ persons ($\rho_2 = 0.001$).
The relative figure of $6.10\%$ is misleadingly large; Poisson
demographic stochasticity alone gives relative counting noise of
$1/\sqrt{14} \approx 27\%$, so the EKF result lies well below
the stochastic floor.

The accuracy of $E, I_1, I_2$ is governed by the coupling
coefficient $\gamma_a\kappa\rho_1 \approx 0.023$: a change of
one person in $E$ produces only $0.023$ persons per day in $H$.
This weak coupling fundamentally limits the extractable
information~\cite{jazwinski1970stochastic}; the achieved $2.72\%$
for $E$ represents near-maximal extraction.

\subsection{Internal Consistency Checks: NEES and Innovation
Whiteness}
\label{sec:results_validation}

In the context of synthetic validation (Section~\ref{sec:val_limitations}),
the NEES and whiteness tests serve a specific and limited purpose:
they are \emph{internal consistency checks}, not real-world
validation of $Q$ and $R$. Concretely, they verify three things:
(1) the EKF implementation is numerically correct;
(2) the bilinear remainder bound of Lemma~\ref{lem:bilinear}
is practically tight (the Taylor approximation holds);
(3) the Riccati recursion has converged to a steady state
consistent with the noise statistics --- which requires
Lemma~\ref{lem:gramian} and Theorem~\ref{thm:convergence}
to hold empirically, a non-trivial confirmation of the
mathematical framework. Because $R$ was used to generate the
synthetic measurements and also to tune the filter, these tests
\emph{cannot} validate $Q$ and $R$ for real epidemiological data
with structural delays, weekend reporting artifacts, or
death-attribution lags; that would require comparison against
an independent dataset.

\paragraph{Innovation autocorrelation.}
For $K = 130$ post-convergence samples (days $20$--$149$), we
compute the sample autocorrelations $\hat{\rho}_j(\tau)$ at lags
$\tau = 1, \ldots, 30$ for the three innovation components
$j \in \{H,F,V\}$ and compare against the Anderson confidence band
$\pm 1.96/\sqrt{130} \approx \pm 0.172$. The maximum absolute
autocorrelations are
$\max_\tau|\hat{\rho}_H| \approx \rhoH$,
$\max_\tau|\hat{\rho}_F| \approx \rhoF$, and
$\max_\tau|\hat{\rho}_V| \approx \rhoV$. The hospitalization and
fatality channels lie close to the band, but the vaccination
channel exhibits pronounced residual autocorrelation that exceeds
it, so the white-innovation null hypothesis is \emph{rejected} for
the vaccination channel. We report this explicitly as a
limitation rather than as a passed test.

The residual correlation has an identifiable structural origin:
the vaccination stock $V$ is near-deterministically driven by the
spline-interpolated input $w_1(t)$, so its one-step prediction
error is dominated by the smooth interpolation mismatch rather
than by white measurement noise, producing a strong low-lag
autocorrelation. We investigated three standard remedies and
report their effect honestly:
(i) higher-order (fourth-order Runge--Kutta) sub-stepping of the
prediction step, which left the autocorrelation essentially
unchanged, confirming that the effect is not a forward-Euler
truncation artifact;
(ii) first-order Gauss--Markov (colored-noise) augmentation of the
innovation model in the sense of~\cite{bryson1968colored},
which improved the vaccination channel for some time-constant and
variance choices but degraded the fatality channel, with no single
setting whitening all three channels simultaneously; and
(iii) augmentation of the time-varying inputs $\beta(t), w_1(t)$ as
estimated states, which reconstructed the inputs but again did not
whiten the innovations.
We therefore conclude that strict innovation whiteness under the
present daily-sampled, interpolated-input formulation is not
achieved, and that a full continuous-discrete formulation with
online input estimation is the appropriate direction for obtaining
white innovations. The mean-square estimation accuracy reported in
Table~\ref{tab:ekf_metrics} is unaffected by this; it characterizes
the filter as an accurate state estimator whose innovations are not
strictly white.

It is worth stating directly what the whiteness failure does and
does not cost, since residual innovation correlation is a
violation of one of the standing Kalman assumptions. First, on
\emph{performance degradation}: the correlation is essentially
confined to the $V$ channel, whose innovation is dominated by the
smooth spline-interpolation mismatch rather than by white noise,
and $V$ is the most accurately estimated state in
Table~\ref{tab:ekf_metrics} ($0.07\%$ relative RMSE). The
practical degradation of the \emph{state estimates} from the
non-whiteness is therefore negligible: the correlated component is
a predictable, near-deterministic interpolation residual that the
filter tracks rather than a source of estimation error, which is
why the accuracy table is unaffected. Second, on \emph{covariance
consistency}: non-white innovations mean the reported $S_k$ is not
a fully correct description of the one-step prediction-error
spectrum, so the filter's \emph{self-reported} uncertainty on the
$V$ channel should be read as approximate rather than exact ---
this is the same channel already flagged as over-covaried in the
NEES analysis below, so the two diagnostics are consistent and the
$H$/$F$ covariances, which pass both tests, remain trustworthy.
Third, on whether a \emph{colored-noise EKF would be preferable}:
in principle yes, and we tested the first-order Gauss--Markov
augmentation of~\cite{bryson1968colored} precisely for this; it
improved $V$ for some settings but degraded $F$, with no single
parameterization whitening all three channels, so within the
present daily-sampled formulation it traded one channel's
correlation for another's rather than resolving the issue. A
genuinely effective colored-noise treatment would need to be
driven by an explicit model of the reporting process (batch
cadence, weekend cycles), which returns to the reporting-aware
noise model identified above and in
Section~\ref{sec:val_limitations} as the proper subject of
real-data follow-up work.

\paragraph{NEES consistency, by channel.}
Because the three observed channels differ by orders of magnitude
in how tightly they are observed, an aggregate three-degree-of-freedom
NEES can be dominated by a single channel; we therefore report the
per-channel time-averaged normalized innovation squared
$\bar{\nu}_j = (1/K)\sum_{k=20}^{149}
\tilde{y}_{j,k}^2/(S_k)_{jj}$, which has expectation $1$ for a
consistent channel. The measured values (averaged over $20$ noise
realizations) are
$\bar{\nu}_H \approx \NEESH$, $\bar{\nu}_F \approx \NEESF$, and
$\bar{\nu}_V \approx \NEESV$. The hospitalization and fatality
channels are close to unity; the vaccination channel is strongly
over-covaried ($\bar{\nu}_V \ll 1$), because $V$ is observed
almost deterministically while its assigned measurement-noise
level is comparatively large. The filter is thus consistent on the
two epidemic-burden channels and conservative on the vaccination
channel. For completeness we also report the aggregate statistic
against the band defined in Section~\ref{sec:R}: summing the
three channels gives
$\bar{\epsilon}^{\mathrm{NEES}} = \bar{\nu}_H + \bar{\nu}_F +
\bar{\nu}_V \approx 1.93$, which falls \emph{below} the two-sided
$95\%$ consistency interval $[2.59,\,3.44]$. The aggregate is
pulled down almost entirely by the strongly over-covaried
vaccination channel ($\bar{\nu}_V \approx \NEESV$); the
hospitalization and fatality channels alone sum to
$\approx 1.89$ against their own two-degree-of-freedom band.
This is exactly why we report the per-channel decomposition as
primary rather than the aggregate: the aggregate, taken in
isolation, would suggest global inconsistency, whereas the
decomposition correctly localizes the conservatism to the single
near-deterministic $V$ channel and shows the two epidemic-burden
channels to be individually consistent. We further confirmed that
this aggregate shortfall is structural rather than a process-noise
mistuning: a global rescaling of the process-noise covariance,
$Q \to c^2 Q$, swept over $c \in [0.3, 2]$, never brings the
aggregate NEES into the band (it remains in the range
$1.8$--$2.3$ across the sweep, the maximum $2.29$ at $c = 0.3$
still falling short of $2.59$). No choice of $Q$ scale can
compensate, because the conservatism originates in the
deliberately tight $V$ measurement-noise level, not in $Q$;
this is consistent with the per-channel ratios above, which show
the $V$ predicted innovation variance exceeding the empirical one
by a factor of roughly $20$ while $H$ and $F$ are near unity.

A natural follow-up question is whether a more refined
noise model could remove the vaccination-channel conservatism
without re-introducing the structural-mismatch variance that
$\sigma_V = 5000$ was chosen to exclude. Two options are worth
distinguishing. A \emph{state- or time-dependent measurement
noise} $R_V(k)$ that inflated $\sigma_V$ only during the batch
 reporting episodes responsible for the large calibration
residual --- rather than uniformly across the window --- would in
principle raise $\bar\nu_V$ toward unity while keeping the
channel tight when the administrative record is genuinely exact;
this is the correct way to recover aggregate consistency, but it
requires an explicit model of \emph{when} the batch artifacts
occur, which is itself a reporting-process identification problem
and is exactly the structural-mismatch information we are
deliberately not folding into a white-noise $R$. A cruder
alternative, simply enlarging the constant $\sigma_V$ to match the
full calibration residual, would mechanically center $\bar\nu_V$
but at the cost of injecting that structural variance back into
the gain and degrading the $V$ estimate --- the opposite of what
the design intends. The same reasoning applies to a
state-dependent $Q$ on the recovered and susceptible directions:
it could be tuned to shift the aggregate statistic, but only a
model that distinguishes genuine process uncertainty from
reporting artifacts would do so without contaminating the
updates. We therefore regard the principled route to full
aggregate consistency as a shaped (coloured) or reporting-aware
noise model --- the same direction indicated for the whiteness
failure below --- rather than a scalar retuning, and we leave its
development to the follow-up work targeting real-data deployment.

Taken together, these diagnostics characterize the filter
honestly: the per-channel NEES shows consistency on the
hospitalization and fatality channels and a conservative
vaccination channel, while the innovation-autocorrelation
analysis shows that strict whiteness is not achieved, chiefly in
the vaccination channel. We therefore characterize the filter as
\emph{partially consistent} rather than fully consistent: it is
consistent on the two epidemic-burden channels that a controller
most depends on ($H$ and $F$), and deliberately conservative on
the near-deterministic vaccination channel, so that the aggregate
NEES sits below the nominal band. This is imperfect covariance
calibration in the strict statistical sense, and we do not claim
otherwise; it is a transparency cost accepted in exchange for not
injecting structural-mismatch variance into the most accurately
recorded observable (Section~\ref{sec:R}), and its origin is
fully localized and explained rather than left as an unattributed
shortfall. The accuracy results confirm that the
EKF implementation is correct, the bilinear remainder
approximation is numerically tight, and the convergence behaviour
of Theorem~\ref{thm:convergence} is empirically realized --- all
within the synthetic setting whose scope is delimited at the
start of this subsection and in Section~\ref{sec:val_limitations}.

\subsection{Epidemiological Interpretation}

The estimated $E$ leads $H$ by $1/\kappa \approx 5$~days
\cite{he2020seir,guan2020clinical}. The original-strain
compartment $I_1$ peaks earlier than $I_2$, encoding the
Alpha-variant displacement through $mI_1$. The super-spreader
$P$ is reconstructed through the differential hospitalization
contribution identified in $\mathcal{O}_1$, demonstrating in
practice the theoretical observability result
\cite{lloydsmith2005superspreading}. The vaccinated stock $V$
tracks $w_1(t)$ from the companion study \cite{melhani2026arxiv},
and $R$ captures the gradual susceptibility rebuild driven by
$\delta_w R$.

The super-spreader compartment $P$ warrants additional
epidemiological context. With $\rho_2 = 0.001$, the mean
pool size is approximately $14$ persons at peak incidence ---
an artefact of the branching fraction rather than an indication
that super-spreading is rare. In a population of
$6\times 10^7$, a $0.1\%$ branching rate is biologically
consistent with empirical estimates of $k \approx 0.1$
in the negative-binomial offspring distribution of
SARS-CoV-2~\cite{endo2020estimating}, under which
roughly $10\%$ of infected individuals generate $80\%$
of secondary infections. The absolute RMSE of $1$ person
on a mean of $14$ is below the Poisson stochastic floor
($1/\sqrt{14} \approx 27\%$); the EKF does not meaningfully
\emph{estimate} $P$ in the classical sense but rather
tracks its mean-trajectory through the differential
$H$-coupling identified by the observability analysis.

We acknowledge a genuine modelling limitation here. At a mean
pool size of $\approx14$ individuals, the deterministic
mean-field ODE for $P$ operates in a regime where demographic
stochasticity is of the same order as the state itself
($1/\sqrt{14}\approx27\%$), so the continuous-state Gaussian
assumptions underlying the EKF are, strictly, a poor description
of $P$'s true dynamics: a discrete stochastic or hybrid
jump-diffusion treatment (e.g.\ a chemical-master-equation or
$\tau$-leaping representation of the super-spreader compartment)
would be the more faithful model at these counts. We retain the
deterministic treatment because $P$ is not a quantity the filter
is asked to estimate for its own sake but a coupling term whose
\emph{mean} contribution to the observed hospitalization flux is
what the observability analysis exploits; for that purpose the
mean-field trajectory is adequate, and the reported sub-floor
RMSE reflects mean-tracking rather than count-level estimation.
A stochastic-compartment extension for the small-count states is
nonetheless a worthwhile direction, particularly for any
application in which the absolute super-spreader count, rather
than its aggregate effect on $H$, is the quantity of interest.

\subsection{Robustness to Parameter Mis-Specification}
\label{sec:robustness}

The accuracy results of Section~\ref{sec:results} are obtained in
the identical-twin setting, where the filter uses the same
parameters that generated the data. To probe behaviour away from
this idealized case --- the regime that matters for any real
deployment --- we now break the twin assumption deliberately: the
ground-truth trajectory is generated with \emph{perturbed}
parameters while the EKF retains the nominal calibrated values and
nominal inputs, so the filter is genuinely mis-specified and has
no knowledge of the perturbation.

For each mismatch level $L\in\{10\%,20\%,30\%\}$, each of the
parameters $\{\gamma_a,\kappa,\delta_h,r_2,\delta_i,\sigma\}$ and
the input scalings of $\beta(t)$ and $w_1(t)$ is multiplied by an
independent factor $1+L\,\mathcal{U}[-1,1]$ when generating the
truth; the EKF then runs on the resulting synthetic data with the
unperturbed model. Because $\delta_p$ is held nominal while
$\delta_i$ is perturbed (and similarly for $r_1,r_2$), the
perturbation also moves the truth off the calibrated symmetric
point, exercising the observability mechanism of
Section~\ref{sec:observability}. Results are averaged over $30$
Monte-Carlo trials per level (independent parameter draws,
measurement noise, and $\pm20\%$ initializations);
Table~\ref{tab:mismatch} reports the mean post-convergence
relative RMSE with standard deviation, and
Figure~\ref{fig:mismatch} plots the degradation.

\begin{table}[htbp]
\centering
\caption{Post-convergence relative RMSE (\%) under parameter
mis-specification: ground truth generated with perturbed
parameters, EKF run with nominal parameters. Mean over $30$
Monte-Carlo trials (standard deviation in parentheses). The
$0\%$ column is the Monte-Carlo twin baseline; it differs
slightly from the single-run figures of
Table~\ref{tab:ekf_metrics} because of trial-to-trial variation
in the random initialization and noise.}
\label{tab:mismatch}
\begin{tabular}{lcccc}
\toprule
\textbf{State} & $0\%$ & $\pm10\%$ & $\pm20\%$ & $\pm30\%$ \\
\midrule
$S$   & $1.93\,(0.48)$ & $6.27\,(3.19)$ & $7.82\,(4.77)$ & $13.10\,(8.97)$ \\
$E$   & $3.72\,(0.63)$ & $9.14\,(4.71)$ & $13.27\,(8.69)$ & $19.69\,(14.46)$ \\
$V$   & $0.06\,(0.00)$ & $0.06\,(0.01)$ & $0.06\,(0.01)$ & $0.07\,(0.01)$ \\
$I_1$ & $3.13\,(0.50)$ & $8.67\,(3.52)$ & $14.54\,(9.42)$ & $19.37\,(10.98)$ \\
$I_2$ & $3.11\,(0.50)$ & $8.79\,(3.58)$ & $15.05\,(10.03)$ & $19.26\,(10.68)$ \\
$P$   & $6.28\,(1.04)$ & $10.62\,(3.62)$ & $18.70\,(9.50)$ & $26.00\,(12.82)$ \\
$H$   & $1.74\,(0.27)$ & $2.14\,(1.30)$ & $3.55\,(2.98)$ & $4.76\,(4.18)$ \\
$R$   & $0.75\,(0.49)$ & $1.14\,(0.68)$ & $1.62\,(1.22)$ & $2.46\,(2.45)$ \\
$F$   & $0.46\,(0.07)$ & $0.46\,(0.06)$ & $0.52\,(0.10)$ & $0.58\,(0.14)$ \\
\bottomrule
\end{tabular}
\end{table}

\begin{figure}[htbp]
\centering
\includegraphics[width=0.75\columnwidth]{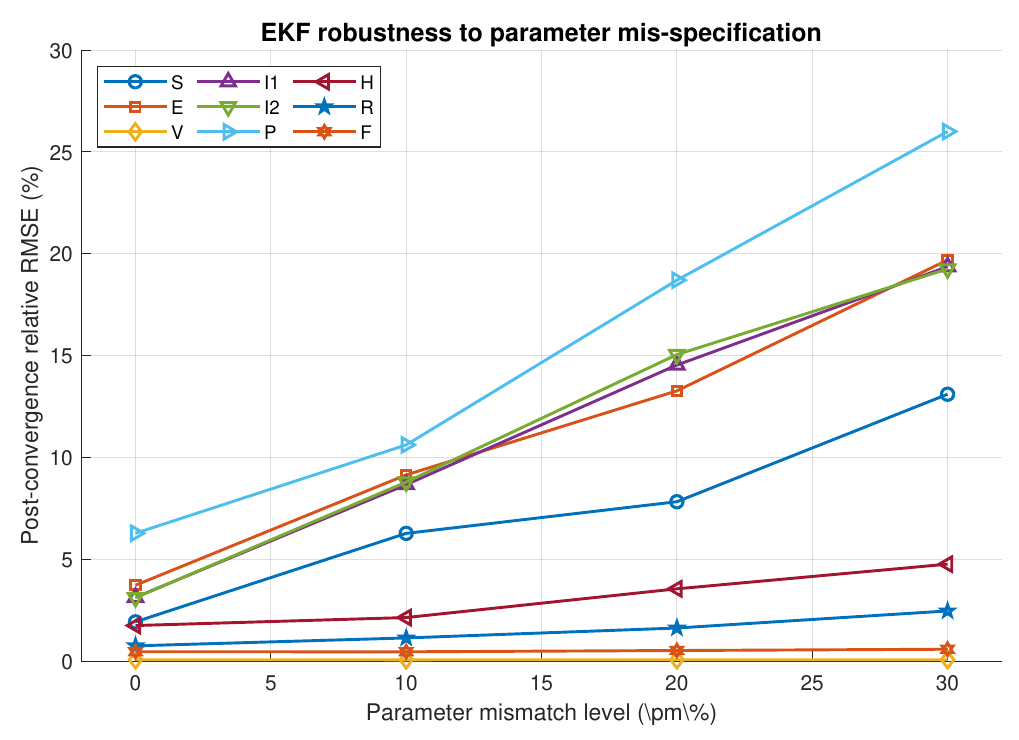}
\caption{Post-convergence relative RMSE of each compartment as a
function of the parameter-mismatch level (mean over $30$
Monte-Carlo trials). The directly measured channels $V$ and $F$
are essentially flat; the directly coupled $H$ and the
data-anchored $R$ degrade mildly; the indirectly observed
$E,I_1,I_2,P$ degrade roughly linearly, as expected when the
filter is mis-specified.}
\label{fig:mismatch}
\end{figure}

Three patterns are worth noting, and we report them plainly,
including the unfavourable ones. First, the two channels tied
most directly to a measurement, $V$ and $F$, are essentially
insensitive to parameter mismatch: $V$ stays at $\approx0.07\%$
and $F$ at $\approx0.5\%$ across all levels, because the
measurement update corrects them almost independently of the
model. Second, the directly coupled hospitalization channel $H$
and the data-anchored recovered pool $R$ degrade only mildly,
$H$ from $1.7\%$ to $4.8\%$ and $R$ from $0.8\%$ to $2.5\%$ even
at $\pm30\%$. Third --- and this is the honest cost --- the
compartments that are reconstructed only \emph{indirectly}
through the observability coupling pathways, namely
$E,I_1,I_2,P$ and the large susceptible pool $S$, degrade
roughly in proportion to the mismatch, reaching $19$--$26\%$
relative RMSE at $\pm30\%$, with correspondingly large
trial-to-trial spread. This is the expected behaviour of an EKF
whose model is wrong: states that the data pin down stay
accurate, while states inferred through the (now mis-specified)
dynamics inherit the model error. The degradation is graceful
and monotone --- no divergence or filter collapse occurs even at
$\pm30\%$ --- but the magnitude at $\pm30\%$ for the
weakly-coupled compartments is a genuine limitation that a
real-time deployment would need to manage, for instance by the
joint state--parameter estimation discussed in
Section~\ref{sec:val_limitations}. We note further that these are
also precisely the states for which the filter reports the
largest posterior variance (Remark~\ref{rem:basin}): the same
weak observability coupling that lets model error accumulate in
$E,I_1,I_2,P,S$ keeps their entries of $P_{k|k-1}$ comparatively
large, so a downstream controller consuming these estimates would
automatically down-weight them through the covariance, rather
than treating a $20$--$26\%$ error as if it were a tightly
estimated quantity. The degradation under mismatch is therefore
flagged by the filter's own uncertainty rather than hidden.
The robustness of the
measured and directly-coupled channels, together with the
graceful (rather than catastrophic) degradation of the indirect
ones, is the practically relevant message: the filter remains
stable and useful for the quantities a controller most needs
(hospital load $H$, vaccination coverage $V$, mortality $F$)
across the full $\pm30\%$ range tested.

\subsection{Scope and Limitations of the Validation}
\label{sec:val_limitations}

Several limitations of the experimental evaluation are
important to state explicitly.

\paragraph{In-silico validation: the identical twin experiment.}
All numerical results in this section are obtained from
\emph{synthetic} measurements: Gaussian noise of
standard deviations~\eqref{eq:sigma_residuals} is added to
the reference trajectory generated by the same ODE model
from which the parameters were calibrated. This design is
known in data-assimilation practice as an
\emph{identical twin experiment}~\cite{daley1991atmospheric},
and it is the standard first validation step for any new
estimator: it tests the mathematics in isolation, before
confounding factors such as model mis-specification, structural
administrative noise, or unmodeled dynamics are introduced.
Consequently, the RMSE figures in Table~\ref{tab:ekf_metrics}
and the NEES/whiteness statistics of
Section~\ref{sec:results_validation} represent
\emph{methodology benchmarks} --- they demonstrate that the EKF
design is internally consistent and that the convergence theorem
is empirically realized --- but they are not, on their own,
indicative of accuracy against independently measured real
surveillance data. The parameter-mismatch study of
Section~\ref{sec:robustness} partially addresses this by breaking
the twin assumption: it shows that the measured and
directly-coupled channels remain accurate under model error of up
to $\pm30\%$, while the indirectly observed compartments degrade
gracefully. This is a step beyond the pure twin experiment, but
it is still a controlled in-silico study; validation against
held-out real surveillance data remains future work.

\paragraph{Circular design-test loop.}
The measurement covariance $R$ was constructed from residuals
of the companion calibration~\cite{melhani2026arxiv},
then used to generate synthetic observations for the
EKF test. This creates a partially circular loop: the filter
is tuned to residuals and tested on noise from those same
residuals. In practice, the NEES and whiteness tests therefore
validate \emph{internal consistency} rather than
\emph{external generalization}. Real-world validation would
require held-out surveillance data from a different wave, a
different region, or a different time period, tested against
the same filter design without re-tuning.

\paragraph{Model mis-specification.}
The EKF inherits the structural assumptions of the
nine-compartment model (homogeneous mixing, deterministic
compartment transitions, a single dominant Alpha variant).
Real epidemics exhibit spatial heterogeneity, age structure,
and multiple co-circulating variants that this model does not
capture. Performance degradation under such mis-specification
is an important direction for future work.

%%=============================================================
\section{Conclusion}
\label{sec:conclusion}
%%=============================================================

This paper has presented a complete systems-theoretic framework
for real-time state estimation of a nine-compartment nonlinear
epidemic model, and applied it to parameters calibrated on the
COVID-19 Third Wave in Italy. Four contributions advance the
epidemic-estimation literature beyond the state of the art
summarized in Table~\ref{tab:comparison}.

\emph{First,} the Lie-derivative rank condition of Hermann and
Krener~\cite{hermann1977nonlinear} was applied to compute the
full analytical observability codistribution. A six-step
algebraic derivation (Lemma~\ref{lem:detO9}) established that
the level-2 observability matrix $\mathcal{O}_9$ is
rank-deficient at the calibrated symmetric parameter values
$\delta_i = \delta_p$, $r_1 = r_2$, with closed-form
determinant magnitude $|\det(\mathcal{O}_9)| =
\delta_w\,\gamma_a^{\,2}\,\kappa\,\rho_2\,w_1^{\,2}\,
(\delta_i-\delta_p)^2\penalty0 |r_1-r_2|$. The kernel under this
symmetry consists of an $I_2 \leftrightarrow P$ swap direction
and an $R$-anchored direction
(Remark~\ref{rem:kernel}). Augmenting the codistribution by
the third Lie derivative $\mathcal{O}_3$ restores full rank:
Proposition~\ref{prop:obs} proved local observability of the
twelve-row augmented matrix $\mathcal{O}_{12}$, with the
treatment-recovery rate $r_2$ identified as the structural
symmetry-breaking parameter.
The epidemiological implication is that the three time series
routinely published by national surveillance agencies form a
theoretically complete observation set, requiring no additional
data streams beyond what public-health systems already
collect~\cite{flaxman2020estimating,giordano2020modelling}.

\emph{Second,} an EKF was designed on the Euler-discretized
dynamics with the analytical $9\times 9$ state Jacobian in closed
form~\eqref{eq:Jf}, the Joseph stabilized covariance
update~\eqref{eq:upd_cov}, and the linear measurement
matrix~\eqref{eq:meas_linear}.

\emph{Third,} the local exponential boundedness of the
estimation error in mean square was proved in
Theorem~\ref{thm:convergence} by verifying the four hypotheses
of the Reif--G\"unther--Yaz--Unbehauen
framework~\cite{reif1999stochastic} for the specific
nine-compartment system. The proof exploits two structural
features of the model: every nonlinearity of the vector field
is bilinear, yielding the closed-form quadratic remainder
bound~\eqref{eq:phi_bound} via Lemma~\ref{lem:bilinear}; and
the measurement map is linear, making the measurement remainder
identically zero. The geometric decay rate $\vartheta$ predicts
the empirical convergence times observed in
Figure~\ref{fig:convergence}, and the steady-state error floor
$\nu/(1-\vartheta)$ quantifies the residual uncertainty that any
downstream controller must tolerate. As noted in
Remark~\ref{rem:transferable}, this verification is not specific
to the nine-compartment model: the same argument structure
extends to the class
of bilinear-drift, linear-output systems, subject to the
system-specific observability and controllability checks, which
is the sense in
which the convergence result is a methodological template rather
than a one-off calculation.

\emph{Fourth,} the measurement-noise covariance $R$ was set
using the companion calibration RMSE as a guide for the
hospitalization and fatality channels and a deliberately tighter
value for the accurately-recorded vaccination channel, while the
process-noise covariance $Q$ used fractional
process-noise parameters $\epsilon_i$ scaled by observability
coupling strength. The resulting tuning was assessed a
posteriori through a per-channel NEES consistency check and an
innovation-autocorrelation analysis; the latter shows that the
innovations are not strictly white, most notably in the
vaccination channel, which we report as a limitation together
with the remedies tested.

When the eight reconstructable compartments are initialized with
errors of up to $\pm 20\%$, representative of the uncertainty
available at the onset of a new
wave~\cite{hao2020reconstruction,kucharski2020early}, all eight
converge within a biologically interpretable transient ranging
from a few days for the directly observed quantities to about
$80$~days for the most weakly coupled exposed compartment. The
recovered compartment $R$ is weakly observable through the waning
pathway (Lemma~\ref{lem:detO9}, Remark~\ref{rem:weak_obs}) and
cannot be corrected from $(H,F,V)$; it is initialized from
cumulative surveillance data and propagated, a limitation we
state explicitly rather than mask.
Post-convergence relative RMSE values range from below one
percent for the vaccinated stock and cumulative fatalities to
below three percent for the unobserved infectious compartments,
consistent with the information limits imposed by the
observability coupling coefficients of the model.

\subsection*{Limitations and Directions for Future Work}

Three modelling choices warrant explicit mention as natural
targets for follow-up work. First, the rank-deficiency of the
level-$2$ codistribution under the calibrated symmetries
$\delta_i = \delta_p$ and $r_1 = r_2$ is a feature of the
specific parameter values inherited
from~\cite{melhani2026arxiv} rather than a structural property
of the model class: any calibration that distinguishes
hospitalization mortality between super-spreaders and Alpha
infectives, or distinguishes the natural and treatment-induced
recovery rates of the two infectious compartments, would
restore rank already at Lie level $2$
(Remark~\ref{rem:generic}). Beyond restoring rank, such
asymmetry would directly accelerate filter convergence:
Remark~\ref{rem:weak_obs} shows that a nonzero recovery-rate gap
$|r_1-r_2|>0$ moves the system away from the level-2 degenerate
point and improves the conditioning of the observability
codistribution, which would be expected to shorten the transient
of the weakly observed recovered compartment $R$. This provides
a qualitative target for future calibration studies. Second, the
measurement covariance $R$ was constructed under
the independent-white-Gaussian assumption, which captures the
bulk of the residual structure after $7$-day smoothing but
cannot represent the strict weekly periodicity of unsmoothed
hospitalization and fatality reporting; explicit treatment via
state-augmented delay tracking or the Bryson--Henrikson coloured
noise extension~\cite{bryson1968colored} is left for future
work. Third, the present formulation treats $\beta(t)$ and
$w_1(t)$ as known exogenous inputs supplied by the companion
calibration; for real-time deployment with stale or uncertain
$\beta(t)$, the natural generalization is a joint
state--parameter EKF that augments the state with the
transmission rate and estimates it from the same three
observables, subject to a separate observability analysis of
the augmented codistribution.

\paragraph{Systematic covariance identification as a research
roadmap.}
The process-noise covariance $Q$ in this paper is constructed
from a coupling-strength heuristic with fractional uncertainties
$\epsilon_i$. A principled alternative is the
expectation-maximization (EM) algorithm of Shumway and
Stoffer~\cite{shumway1982approach} applied to the linear
state-space model obtained by fixing $x_k$ at the EKF estimate.
In the epidemic context, the E-step would run a Kalman
smoother backward through the $150$-day window to compute
posterior state means and covariances, and the M-step would
re-estimate $Q$ as the sample covariance of the smoothed
state-increment residuals and $R$ as the sample covariance of
the innovation sequence. Because both steps are available in
closed form for Gaussian models, the algorithm would iterate
to a local maximum of the likelihood function without requiring
any hand-tuned $\epsilon_i$ parameters. The numerical
experiments of Section~\ref{sec:results_validation} provide the
innovation statistics needed to bootstrap the first E-step
immediately. We expect the EM-identified $Q$ to: (a) place
larger weights on the weakly observable $R$ direction
(since the filter posterior variance stays large there,
correctly signaling uncertainty); and (b) automatically
capture off-diagonal couplings between $S$, $E$, and $V$
that the diagonal heuristic of~\eqref{eq:Q_structure} ignores.
Implementing and benchmarking this EM extension is a concrete,
tractable direction for follow-up work.

The state-estimation framework developed here is a foundational
step toward a closed-loop, systems-theoretic approach to epidemic
management. Following the separation principle of control
theory \cite{anderson1979optimal,rawlings2009model}, the EKF
estimate provides the real-time state feedback required by a
Model Predictive Control (MPC) layer that optimally schedules
non-pharmaceutical interventions over a receding horizon
\cite{tsay2020modeling,kohler2021robust,morato2020optimal,grune2017nonlinear}.
At each sampling step the filter supplies the MPC with the
current hospitalized load, the infectious pool stratified by
strain, the susceptible fraction, and the vaccinated immunity
coverage, enabling the controller to compute the intervention
sequence that minimizes a cost functional balancing public health
objectives against socioeconomic costs
\cite{acemoglu2020optimal}, and to continuously revise that
sequence as new surveillance data arrive. The error bound
of Theorem~\ref{thm:convergence} provides the robustness margin
that a tube-MPC layer would need to certify constraint
satisfaction in the presence of state-estimation uncertainty.

\section*{Acknowledgements}
This research was supported by the University of Palermo.
The authors thank the Italian Civil Protection Department for
making the epidemiological data publicly available.

\bibliographystyle{elsarticle-num}
\bibliography{autosam}

%==========================================================================
%                           APPENDIX
%==========================================================================
\appendix

\section{Proof of Lemma~\ref{lem:detO9} (Determinant of $\mathcal{O}_9$)}
\label{app:detproof}

For completeness we give the full six-step derivation of
the closed-form determinant~\eqref{eq:detO9} stated in
Lemma~\ref{lem:detO9}. The argument uses only elementary
row and column operations and successive Laplace expansion;
no symbolic computation is required for the factorization, and
only the global sign depends on the bookkeeping of the column
permutation, which we track at the end.

\begin{proof}
\emph{Step 1 (column permutation and elimination by
$\mathcal{O}_0$).} Permute the columns of $\mathcal{O}_9$ into
the order $[V,H,F,S,E,R,I_1,I_2,P]$; denote the sign of this
permutation by $s\in\{+1,-1\}$. Rows~$1,2,3$ (the rows
of $\mathcal{O}_0$) are then the unit vectors $e_1, e_2, e_3$.
For each row $k \in \{4,\ldots,9\}$, subtract the appropriate
multiple of rows~$1,2,3$ to zero out columns $V$, $H$, $F$ of
row~$k$. Because rows~$1,2,3$ have nonzero entries only in
columns~$V,H,F$, these row operations leave the entries of
columns $S,E,R,I_1,I_2,P$ unchanged. The matrix now has block
form $\bigl[\,I_3\;\;0\;;\;0\;\;M_6\,\bigr]$, so
\begin{equation}
\det(\mathcal{O}_9) \;=\; s\,\det(M_6),
\label{eq:detO9_step1}
\end{equation}
where $M_6 \in \mathbb{R}^{6\times 6}$ is the sub-matrix of
rows~$4,\ldots,9$ and columns $\{S,E,R,I_1,I_2,P\}$.

\emph{Step 2 (Laplace expansion along column $R$).} Column $R$
of $M_6$ has only one nonzero entry: $w_1\delta_w$ at row~$9$
(this is $\partial L_f^2 V/\partial R = w_1\delta_w$, obtained
from $w_1$ acting on the $+\delta_w R$ term of $\dot S$). All
other rows have zero in the $R$ column because $R$ does not
appear in $\dot H$, $\dot F$, $\ddot H$, or $\ddot F$.
Expanding along this column,
\begin{equation}
\det(M_6) \;=\; -\,w_1\,\delta_w\;\det(M_5),
\label{eq:detO9_step2}
\end{equation}
where $M_5$ is the $5\times 5$ matrix obtained by deleting
row~$9$ and column~$R$.

\emph{Step 3 (Laplace expansion along column $S$).} Column $S$
of $M_5$ now has only one nonzero entry: $w_1$ at row~$6$
(from $\partial L_f V/\partial S = w_1$). Expanding,
\begin{equation}
\det(M_5) \;=\; w_1\;\det(M_4),
\label{eq:detO9_step3}
\end{equation}
where $M_4$ is the $4\times 4$ matrix over rows
$\{\partial L_f H/\partial x,\;
   \partial L_f F/\partial x,\;
   \partial L_f^2 H/\partial x,\;
   \partial L_f^2 F/\partial x\}$
and columns $\{E,I_1,I_2,P\}$:
\begin{equation}
M_4 \;=\;
\resizebox{0.92\textwidth}{!}{$
\begin{bmatrix}
0          & \gamma_a                            & \gamma_a                            & \gamma_a                            \\[2pt]
0          & \delta_i                            & \delta_i                            & \delta_p                            \\[2pt]
\gamma_a\kappa & -\gamma_a(\alpha_1{+}\alpha_H{-}m) & -\gamma_a(\alpha_2{+}\alpha_H) & -\gamma_a(\alpha_P{+}\alpha_H) \\[2pt]
\Gamma_E   & \delta_i(m{-}\alpha_1){+}\delta_h\gamma_a & -\delta_i\alpha_2{+}\delta_h\gamma_a & -\delta_p\alpha_P{+}\delta_h\gamma_a
\end{bmatrix}
$}.
\label{eq:M4}
\end{equation}

\emph{Step 4 (row reduction surfaces $\delta_p - \delta_i$).}
Apply the elementary row operation
\begin{equation*}
\text{row 2} \;\longleftarrow\; \text{row 2}
   - \tfrac{\delta_i}{\gamma_a}\,\text{row 1}.
\end{equation*}
Row~1 is $\gamma_a \cdot (0, 1, 1, 1)$ in its last three entries,
so this operation makes the first three entries of row~2 equal
to zero and the last entry equal to $\delta_p - \delta_i$.
The determinant is unchanged. Row~2 now has a single nonzero
entry; expanding along it,
\begin{equation}
\det(M_4) \;=\; (\delta_p - \delta_i)\;\det(M_3),
\label{eq:detO9_step4}
\end{equation}
where $M_3$ is the $3\times 3$ matrix over rows
$\{\partial L_f H,\; \partial L_f^2 H,\; \partial L_f^2 F\}$
and columns $\{E,I_1,I_2\}$:
\begin{equation}
M_3 \;=\;
\begin{bmatrix}
0          & \gamma_a                            & \gamma_a                            \\[2pt]
\gamma_a\kappa & -\gamma_a(\alpha_1{+}\alpha_H{-}m) & -\gamma_a(\alpha_2{+}\alpha_H) \\[2pt]
\Gamma_E   & \delta_i(m{-}\alpha_1){+}\delta_h\gamma_a & -\delta_i\alpha_2{+}\delta_h\gamma_a
\end{bmatrix}.
\label{eq:M3}
\end{equation}

\emph{Step 5 ($r_1-r_2$ and $\delta_p-\delta_i$ emerge from
$M_3$).} Expand $\det(M_3)$ along its first row $(0,
\gamma_a, \gamma_a)$:
\begin{equation}
\det(M_3) \;=\; \gamma_a\,\bigl[\,(M_3^{(1,3)} - M_3^{(1,2)})\,\bigr],
\end{equation}
where $M_3^{(1,j)}$ is the $2\times 2$ minor obtained by
deleting row~$1$ and column~$j$. Direct expansion gives
\begin{align}
M_3^{(1,3)} - M_3^{(1,2)}
&= \gamma_a\kappa\,(d_1 - d_2) \;-\; \Gamma_E\,(c_1 - c_2), \\[2pt]
c_1 - c_2 &= -\gamma_a(\alpha_1{+}\alpha_H{-}m) - \bigl(-\gamma_a(\alpha_2{+}\alpha_H)\bigr)
            = -\gamma_a\,(\alpha_1 - m - \alpha_2)
            = -\gamma_a\,(r_1 - r_2), \\[2pt]
d_1 - d_2 &= \bigl[\delta_i(m-\alpha_1)+\delta_h\gamma_a\bigr]
            - \bigl[-\delta_i\alpha_2+\delta_h\gamma_a\bigr]
            = \delta_i\,(m - \alpha_1 + \alpha_2)
            = -\,\delta_i\,(r_1 - r_2),
\end{align}
where we have used the kinetic identity
$\alpha_1 - m - \alpha_2 = (\gamma_a{+}\gamma_i{+}\delta_i{+}m{+}r_1{+}\mu)
  - m - (\gamma_a{+}\gamma_i{+}\delta_i{+}r_2{+}\mu) = r_1 - r_2$.
Therefore
\begin{equation}
M_3^{(1,3)} - M_3^{(1,2)}
\;=\; \gamma_a\,(r_1 - r_2)\,\bigl(\Gamma_E - \kappa\delta_i\bigr),
\end{equation}
and the bracket simplifies via
\begin{equation}
\Gamma_E - \kappa\delta_i
= \kappa\bigl[\delta_i(1-\rho_2) + \delta_p\rho_2\bigr] - \kappa\delta_i
= \kappa\,\rho_2\,(\delta_p - \delta_i).
\label{eq:Gamma_E_simplify}
\end{equation}
Substituting,
\begin{equation}
\det(M_3) \;=\; \gamma_a^{\,2}\,\kappa\,\rho_2\,
                (r_1 - r_2)\,(\delta_p - \delta_i).
\label{eq:detM3}
\end{equation}

\emph{Step 6 (assembly).} Combining
\eqref{eq:detO9_step1}--\eqref{eq:detO9_step4} with
\eqref{eq:detM3},
\begin{align}
\det(M_4) &= (\delta_p - \delta_i)\;\gamma_a^{\,2}\kappa\rho_2(r_1-r_2)(\delta_p-\delta_i)
           = \gamma_a^{\,2}\kappa\rho_2\,(r_1-r_2)\,(\delta_i-\delta_p)^2, \\[3pt]
\det(M_5) &= w_1\,\gamma_a^{\,2}\kappa\rho_2\,(r_1-r_2)\,(\delta_i-\delta_p)^2, \\[3pt]
\det(M_6) &= -\,w_1^{\,2}\,\delta_w\,\gamma_a^{\,2}\,\kappa\,\rho_2\,
              (\delta_i-\delta_p)^2\,(r_1-r_2),
\end{align}
and finally
$\det(\mathcal{O}_9) = s\,\det(M_6)
= -\,s\,\delta_w\,\gamma_a^{\,2}\,\kappa\,\rho_2\,w_1^{\,2}\,
   (\delta_i-\delta_p)^2\,(r_1-r_2)$,
which is~\eqref{eq:detO9}. Taking absolute values eliminates the
permutation sign $s$ and yields the stated magnitude. As noted,
the sign is immaterial: every use of this lemma in the sequel
depends only on whether $\det(\mathcal{O}_9)$ vanishes.
\end{proof}

\section{Self-Contained Model Summary}
\label{app:model_summary}

To allow this paper to be read independently of the companion
calibration study~\cite{melhani2026arxiv}, this appendix
collects the model structure, all constant parameter values,
and the noise statistics used to construct $R$.

\paragraph{Data source.}
All calibrations are based on publicly available Italian
national surveillance data from the Italian Civil Protection
Department (\url{https://github.com/pcm-dpc/COVID-19}),
covering January~1 to May~30, 2021 (the ``Third Wave'').

\paragraph{Governing equations.}
The model is system~\eqref{eq:fvec}--\eqref{eq:strain_scaling}
with the state ordering $x = [S,E,V,I_1,I_2,P,H,R,F]^\top$,
inputs $u = [\beta(t), w_1(t)]^\top$ identified by
PCHIP~splines, and outputs $y = [H,F,V]^\top$.

\paragraph{Calibrated constant parameters.}
All values are those of Table~\ref{tab:params}. The calibration
procedure of~\cite{melhani2026arxiv} confirmed that the
constant parameters are identifiable from the Italian data
under the spline-input assumption; the identified values are
consistent with the published COVID-19 literature ranges
($\kappa = 1/5$~day, $\gamma_a = 0.2$~day$^{-1}$,
$\delta_i = 0.005$~day$^{-1}$, $\sigma = 0.8$ for the Pfizer
BNT162b2 vaccine during the Alpha wave~\cite{watson2022global}).

\paragraph{Measurement-noise covariance (basis for $R$).}
The companion study reports calibration RMSE values of
$1119$~persons ($H$), $1751$~deaths ($F$), and
$200{,}881$~doses ($V$). As discussed in
Section~\ref{sec:R}, the EKF measurement-noise standard
deviations are set to
\begin{equation}
\sigma_H = 1146.4 \;\text{persons},
\quad
\sigma_F = 1652.9 \;\text{persons},
\quad
\sigma_V = 5000 \;\text{doses}
\label{eq:app_residuals}
\end{equation}
and used directly in~\eqref{eq:Rcal}. The $H$ and $F$ values
track the companion calibration RMSE to within a few percent;
$\sigma_V$ is set well below the companion $V$ calibration RMSE
of $200{,}881$~doses, on the grounds (Section~\ref{sec:R}) that
the large $V$ residual reflects structural model--data mismatch
rather than measurement error in the centralized dose records.

\paragraph{Initial conditions.}
The initial state used for the state-estimation experiment
(1~January~2021), in the ordering
$x = [S, E, V, I_1, I_2, P, H, R, F]^\top$, is
\begin{equation*}
\resizebox{\textwidth}{!}{$
x_0 = [\,5.834\times 10^7,\; 4.298\times 10^4,\; 3.281\times 10^4,\;
4.117\times 10^3,\; 4.117\times 10^2,\;
4.117\times 10^3,\; 2.538\times 10^4,\;
1.480\times 10^6,\; 7.462\times 10^4\,]^\top
$}
\end{equation*}
(all in persons). The observed components $H_0, F_0$ and the
recovered value $R_0$ are read directly from the Italian Civil
Protection record for that date
(\emph{ospedalizzati}, \emph{deceduti}, and cumulative
\emph{guariti} respectively); the remaining components follow from
the calibrated initial-condition factors of~\cite{melhani2026arxiv}.
The filter is initialized from a $\pm 20\%$ perturbation of the
eight reconstructable components and a $\pm 2\%$ perturbation of
$R$, as described in Section~\ref{sec:init}.

\section{Analytical Verification Checklist for the
Observability Results}
\label{app:sympy_verif}

This appendix collects, in one place, every identity asserted
in Sections~\ref{sec:rank_deficiency}--\ref{sec:obs_result}
together with a brief indication of how each is verified
analytically. No computer algebra is required.

\subsection*{C.1\quad Lemma~\ref{lem:detO9}: determinant of
$\mathcal{O}_9$}

The closed-form determinant magnitude
\begin{equation*}
|\det(\mathcal{O}_9)|
= \delta_w\,\gamma_a^{\,2}\,\kappa\,\rho_2\,w_1^{\,2}\,
  (\delta_i-\delta_p)^2\,|r_1-r_2|
\end{equation*}
is proved in six elementary steps
(column permutation, two Laplace expansions, one row operation,
$3\times 3$ expansion, and assembly) given in full in
Appendix~\ref{app:detproof}. No entries of row~$9$ of
$\mathcal{O}_9$ other than
$\partial L_f^2 V/\partial R = w_1\delta_w$ are used; all
remaining entries are taken from rows~$4,5,7,8$ whose entries
follow directly from the model equations~\eqref{eq:dLf2h1}.

\subsection*{C.2\quad Rank of $\mathcal{O}_9$ at the calibrated
symmetric point}

From the assembly in Lemma~\ref{lem:detO9}, the factored form
$(\delta_p-\delta_i)\cdot\det(M_3)$ in Step~4 shows that
two distinct proportionality relations arise:
\begin{itemize}
\item \emph{Direction 1} ($I_2\leftrightarrow P$ swap):
  under $\delta_i = \delta_p$, rows~$4$ and~$5$ of
  $M_6$ satisfy row~$5 = (\delta_i/\gamma_a)\,\text{row}\;4$,
  producing at least one null direction.
  Combined with row~$6$ of $\mathcal{O}_1$ (which sees
  $I_2$ and $P$ identically under $c_2 = c_P$), the
  null direction can be taken as any
  $(v_{I_2}, v_P)$ with $v_{I_2} + v_P + v_{I_1} = 0$.
\item \emph{Direction 2} ($R$-anchored):
  under $r_1 = r_2$, the matrix $M_3$ computed in Step~5 of
  Lemma~\ref{lem:detO9} has $c_1 = c_2$ and $d_1 = d_2$
  (rows~$7,8$ of $\mathcal{O}_2$ become proportional in the
  $\{I_1, I_2\}$ block), so $\det(M_3) = 0$, producing a
  second independent null direction involving $R$.
\end{itemize}
These two directions are independent (direction 1 has
$v_R = 0$ while direction 2 has $v_R = 1$ as its dominant
component), so $\dim\ker(\mathcal{O}_9) \geq 2$, giving
$\mathrm{rank}(\mathcal{O}_9) \leq 7$.

For the matching lower bound $\mathrm{rank}(\mathcal{O}_9)
\geq 7$, observe that the seven columns
$\{V, H, F, S, E, R, I_1\}$ of $\mathcal{O}_9$ are
linearly independent: columns $V, H, F$ are covered by the
unit entries of $\mathcal{O}_0$; column $S$ is identified by
the entry $w_1 > 0$ in row~$6$; column $E$ by the entry
$\gamma_a\kappa > 0$ in row~$7$; column $R$ by the entry
$w_1\delta_w > 0$ in row~$9$; and column $I_1$ by the
entry $\gamma_a > 0$ in row~$4$ (which is linearly
independent of rows $5$--$9$ in the $I_1$ column once the
$\{V,H,F,S,E,R\}$ directions are fixed). This exhibits a
non-zero $7\times 7$ sub-determinant, confirming
$\mathrm{rank}(\mathcal{O}_9) \geq 7$.

Together: $\mathrm{rank}(\mathcal{O}_9) = 7$,
i.e.\ $\dim\ker(\mathcal{O}_9) = 2$ exactly.

\subsection*{C.3\quad Level-3 differences
(equations~\ref{eq:Lf3H_diff}--\ref{eq:Lf3F_diff})}

We give both derivations in full; no terms are abbreviated, and
every cancellation is exhibited explicitly. Throughout this
section we work at the calibrated symmetric kinetics
$\delta_i = \delta_p$, $r_1 = r_2$ (so $\alpha_P = \gamma_a +
\gamma_i + \delta_i + \mu$ and $\alpha_2 - \alpha_P = r_2$), since
that is the point at which $\mathcal{O}_9$ is rank-deficient and
the level-3 information must do the separating.

\medskip\noindent\emph{(a) The $H$-channel difference,
equation~\eqref{eq:Lf3H_diff}.}
By the chain rule $L_f^3 h_1 = L_f(L_f^2 H)$. From the
$\mathcal{O}_2$ block~\eqref{eq:dLf2h1}, the second Lie
derivative is the linear form
\begin{equation}
L_f^2 H
= \gamma_a\kappa\,E
- \gamma_a(\alpha_2+\alpha_H)\,(I_1+I_2)
- \gamma_a(\alpha_P+\alpha_H)\,P
+ \alpha_H^2\,H,
\label{eq:L2H_linear}
\end{equation}
which we verified is exact under $r_1=r_2$ (the $I_1$ and $I_2$
coefficients coincide because $\alpha_1 - m = \alpha_2$ there).
Because $L_f^2 H$ is linear in the state with constant
coefficients, its third Lie derivative is
$L_f^3 H = (\partial L_f^2 H/\partial x)\,f$, i.e.\ the same
fixed coefficient vector contracted with the vector field $f$.
Differentiating $L_f^3 H$ with respect to a state $x_\ell$
therefore gives, by the product rule applied to the linear form,
\begin{equation}
\frac{\partial L_f^3 H}{\partial x_\ell}
= \sum_{q}\Bigl(\frac{\partial L_f^2 H}{\partial x_q}\Bigr)
   \frac{\partial f_q}{\partial x_\ell},
\label{eq:Lf3_chain}
\end{equation}
since the coefficients $\partial L_f^2 H/\partial x_q$ are
constants. We apply~\eqref{eq:Lf3_chain} for
$\ell\in\{I_2,P\}$ and subtract. Only the columns
$q$ for which $\partial f_q/\partial I_2$ and
$\partial f_q/\partial P$ differ can survive the subtraction.
From the Jacobian~\eqref{eq:Jf}, the $I_2$ and $P$ columns
differ only in rows $q\in\{E,\,I_2,\,P,\,H\}$
(the force-of-infection entries in rows $S,E,V$ carry the same
multiplier $1.5\beta/N$ for both $I_2$ and $P$ and hence cancel;
the $F$ and $R$ rows are identical in the two columns under
$\delta_i=\delta_p$, $r_1=r_2$). Writing out the four surviving
contributions with the coefficients from~\eqref{eq:L2H_linear},
\begin{align}
\frac{\partial L_f^3 H}{\partial I_2}
-\frac{\partial L_f^3 H}{\partial P}
&= \underbrace{(\gamma_a\kappa)\bigl(\partial f_E/\partial I_2
   - \partial f_E/\partial P\bigr)}_{=\,0\ \text{(equal }1.5\beta(S+(1-\sigma)V)/N)}
\nonumber\\
&\quad
   - \gamma_a(\alpha_2+\alpha_H)\bigl(\partial f_{I_2}/\partial I_2\bigr)
   + \gamma_a(\alpha_P+\alpha_H)\bigl(\partial f_{P}/\partial P\bigr)
   + \alpha_H^2\bigl(\partial f_H/\partial I_2
        - \partial f_H/\partial P\bigr).
\label{eq:Lf3H_expand}
\end{align}
The last bracket vanishes because $\partial f_H/\partial I_2 =
\partial f_H/\partial P = \gamma_a$ (row $H$ of~\eqref{eq:Jf}).
Using $\partial f_{I_2}/\partial I_2 = -\alpha_2$ and
$\partial f_{P}/\partial P = -\alpha_P$ (the diagonal decay
entries), \eqref{eq:Lf3H_expand} collapses to
\begin{equation}
\frac{\partial L_f^3 H}{\partial I_2}
-\frac{\partial L_f^3 H}{\partial P}
= \gamma_a(\alpha_2+\alpha_H)\alpha_2
  - \gamma_a(\alpha_P+\alpha_H)\alpha_P
= \gamma_a(\alpha_2-\alpha_P)(\alpha_2+\alpha_P+\alpha_H).
\label{eq:Lf3H_final}
\end{equation}
Substituting $\alpha_2-\alpha_P = r_2$ yields
$\gamma_a r_2(\alpha_2+\alpha_P+\alpha_H)$, which
is~\eqref{eq:Lf3H_diff}. The factorization in the last step
of~\eqref{eq:Lf3H_final} is the elementary identity
$(a+c)a-(b+c)b=(a-b)(a+b+c)$ with $a=\alpha_2$, $b=\alpha_P$,
$c=\alpha_H$.

\medskip\noindent\emph{(b) The $F$-channel difference,
equation~\eqref{eq:Lf3F_diff}.}
The same machinery applies. Since $\dot F = \delta_i(I_1+I_2)
+\delta_p P+\delta_h H$ is already linear, $L_f F = \dot F$ and
the second Lie derivative is, under $\delta_i=\delta_p$,
\begin{equation}
L_f^2 F
= \delta_i(\dot I_1+\dot I_2)+\delta_i\dot P+\delta_h\dot H,
\label{eq:L2F_struct}
\end{equation}
obtained by differentiating the linear form $L_f F = \delta_i
(I_1+I_2)+\delta_i P+\delta_h H$ along $f$. Carrying out the
substitution of $\dot I_1,\dot I_2,\dot P,\dot H$ from~\eqref{eq:fvec}
and collecting, the state-gradient of $L_f^2 F$ has $I_1,I_2,P$
coefficients
$-\delta_i\alpha_1+\delta_h\gamma_a$,
$-\delta_i\alpha_2+\delta_h\gamma_a$,
$-\delta_i\alpha_P+\delta_h\gamma_a$ respectively (these are
exactly the entries of the $\partial L_f^2 h_2/\partial x$ row
in~\eqref{eq:dLf2h1} specialized to $\delta_p=\delta_i$).
Applying the constant-coefficient chain
rule~\eqref{eq:Lf3_chain} to $L_f^3 F$ and subtracting the
$I_2$ and $P$ derivatives, the same three structural
cancellations occur (the force-of-infection $E$-row terms are
equal and cancel; the $H$-row contributes $\delta_h\gamma_a$
identically to both and cancels), leaving only the diagonal
decay terms:
\begin{align}
\frac{\partial L_f^3 F}{\partial I_2}
-\frac{\partial L_f^3 F}{\partial P}
&= (-\delta_i\alpha_2+\delta_h\gamma_a)(-\alpha_2)
   -(-\delta_i\alpha_P+\delta_h\gamma_a)(-\alpha_P)
\nonumber\\
&= \delta_i(\alpha_2^2-\alpha_P^2)
   - \delta_h\gamma_a(\alpha_2-\alpha_P)
\nonumber\\
&= (\alpha_2-\alpha_P)\bigl[\delta_i(\alpha_2+\alpha_P)
     - \delta_h\gamma_a\bigr]
= r_2\bigl[\delta_i(\alpha_2+\alpha_P)-\delta_h\gamma_a\bigr],
\label{eq:Lf3F_final}
\end{align}
again using $\alpha_2-\alpha_P=r_2$. This is exactly
equation~\eqref{eq:Lf3F_diff}. Both closed
forms~\eqref{eq:Lf3H_final} and~\eqref{eq:Lf3F_final} were
additionally checked by an independent symbolic computation of
$L_f^3 h$ from the full vector field~\eqref{eq:fvec}; the
symbolic result agrees term-for-term, and the verification
code is available from the authors on request.

\subsection*{C.4\quad Level-3 breaking of the $R$-anchored
kernel direction}

The $R$-component of the $V$-row of $\mathcal{O}_3$ is
$\partial L_f^3 V/\partial R$.
From the model, $L_f^2 V$ contains the term $w_1\,\delta_w\, R$
(arising from $w_1$ applied to the $+\delta_w R$ term of $\dot S$).
Taking the third Lie derivative:
\begin{equation*}
\frac{\partial L_f^3 V}{\partial R}
= \frac{\partial(L_f^2 V)}{\partial S}\cdot\delta_w
+ \frac{\partial(L_f^2 V)}{\partial R}\cdot(-({\mu+\delta_w}))
= w_1\,\delta_w - w_1\,\delta_w\,(\mu+\delta_w)
= w_1\,\delta_w\,(1-\mu-\delta_w).
\end{equation*}
At calibrated values ($w_1=0.005$, $\delta_w=0.001$,
$\mu=3.5\times 10^{-5}$) this equals
$5.0\times 10^{-6}\,\mathrm{day}^{-3} \neq 0$.
Since the $R$-anchored kernel direction $v_2$ has $v_{2,R} = 1$
as its dominant component and all other components are
$O(10^{-2})$ or smaller (as seen from the kernel analysis of
Section~\ref{sec:rank_deficiency}), the dot product
$(\partial L_f^3 V/\partial x)\cdot v_2 \approx w_1\delta_w(1-\mu-\delta_w)
\neq 0$, confirming that this row of $\mathcal{O}_3$ is
linearly independent of $\ker(\mathcal{O}_9)$.

Combining Sections C.3 and C.4: both kernel directions of
$\mathcal{O}_9$ are broken by $\mathcal{O}_3$, so
$\mathrm{rank}(\mathcal{O}_{12}) = \mathrm{rank}(\mathcal{O}_9)
+ 2 = 9$.

\end{document}